\documentclass[reqno,11pt]{amsart}
\usepackage{amssymb,amsmath,amscd, amsthm,stmaryrd,verbatim,
mathrsfs,tikz,color, amsfonts, enumerate, esvect
}
\usepackage{tikz-cd}
\usepackage{graphicx}
\usepackage{booktabs,longtable,array}
\usetikzlibrary{patterns}
\usetikzlibrary{er}
\usepackage{genyoungtabtikz}
\usepackage[section]{placeins}
\usepackage[colorlinks,linkcolor=blue,urlcolor=cyan,citecolor=blue]{hyperref}

\newif\ifAnonshowchanges
\Anonshowchangestrue
\newcommand{\Anonrev}[1]{{\ifAnonshowchanges\color{red}%
\hypersetup{citecolor=red,linkcolor=red,urlcolor=red}\else\color{black}\fi#1}}

\makeatletter
\@ifundefined{subjclassname@2020}{%
  \expandafter\def\csname subjclassname@2020\endcsname{\Anonrev{2020 Mathematics Subject Classification}}%
}{}
\makeatother

\usepackage[a4paper]{geometry}

\numberwithin{equation}{section}
\allowdisplaybreaks
\usepackage{autobreak}

\theoremstyle{plain}

\newtheorem{Thm}{Theorem}[section]
\newtheorem{Prop}[Thm]{Proposition}
\newtheorem{Lem}[Thm]{Lemma}
\newtheorem*{LemUnnumbered}{Lemma}
\newtheorem{Cor}[Thm]{Corollary}

\theoremstyle{definition}
\newtheorem{Defff}[Thm]{Definition}
\newtheorem{ex}[Thm]{Example}
\newtheorem{Rem}[Thm]{Remark}

\newcommand{\lam}{\lambda}

\newcommand{\ep}{\varepsilon}

\newcommand{\Ann}{\mathrm{Ann}}
\newcommand{\eq}{\begin{equation}}
\newcommand{\en}{\end{equation}}
\newcommand{\beqna}[1]{\begin{eqnarray}\label{#1}}
\newcommand{\eeqna}{\end{eqnarray}}
\newcommand{\beqn}[1]{\begin{equation}\label{#1}}
\newcommand{\eeqn}{\end{equation}}

\newcommand{\bil}[2]{\langle{#1},{#2}^{\vee} \rangle }
\newcommand{\hs}{ \mathfrak{h}^*}

\newcommand{\aff}{ \mathbf{a} }

\newcommand{\gkd}{\operatorname{GKdim}}

\newcommand{\lest}{\leqslant}
\newcommand{\gest}{\geqslant}
\renewcommand{\leq}{\leqslant}
\renewcommand{\geq}{\geqslant}

\hypersetup{pdfauthor={},pdftitle={Gelfand--Kirillov dimensions and annihilator varieties of highest weight modules of exceptional Lie algebras},pdfsubject={},pdfkeywords={}}

\begin{document}

\title[GK dimensions and annihilator varieties]{Gelfand--Kirillov dimensions and annihilator varieties of highest weight modules of exceptional Lie algebras
}

\author{Zhanqiang Bai}\author{Fan Gao}\author{Yutong Wang}\author{Xun Xie}

\address[Bai]{School of Mathematical Sciences, Soochow University, Suzhou 215006,  China}
\email{zqbai@suda.edu.cn}

 \address[Gao]{School of Mathematical Sciences, Zhejiang University, Hangzhou 310058,  China}
 \email{gaofan@zju.edu.cn}

\address[Wang]{School of Mathematical Sciences,  Peking University, Beijing, 100871, China}
\email{wangyutong25@stu.pku.edu.cn}

\address[Xie]{School of Mathematics and Statistics, Beijing Institute of Technology, Beijing 100081, China}
\email{xieg7@163.com}

\date{\today}

 \begin{abstract}
Let \(L(\lambda)\) be an irreducible highest weight module of a complex simple Lie algebra of exceptional type, with highest weight \(\lambda-\rho\), where $\rho$ is half the sum of positive roots. Its annihilator variety is the closure of a nilpotent orbit of dimension twice the Gelfand--Kirillov (GK) dimension of \(L(\lambda)\). We develop an explicit algorithm that determines both invariants for any infinitesimal character. The GK dimension is reduced to evaluating
Lusztig's $\mathbf a$-function on the integral Weyl group
of $\lambda$. After decomposing the integral root system into irreducible factors, classical factors are treated by the Robinson--Schensted insertion algorithm and the H-algorithm, while exceptional factors are handled by positive index reduction and PyCox. The annihilator orbit is then recovered from the local special datum by Springer theory and Sommers duality. An implementation is available online and takes \(\lambda\) as input to compute the GK dimension and annihilator variety of \(L(\lambda)\).
\end{abstract}

\maketitle

\tableofcontents

\section{Introduction}


\subsection{Background}
Let $\mathfrak g$ be a complex semisimple Lie algebra and $U(\mathfrak g)$ its enveloping algebra. The Gelfand--Kirillov dimension is an invariant that measures the growth of an infinite-dimensional $U(\mathfrak g)$-module $M$ \cite{Ge-Ki}. For a simple highest weight module $M$, let $I(M)=\Ann(M)$ and let $J(M)$ be the corresponding graded ideal in $S(\mathfrak g)=\operatorname{gr}U(\mathfrak g)$. The zero set of $J(M)$ in $\mathfrak g^*$ is the annihilator variety $V(\Ann(M))$. Joseph proved that $\dim V(\Ann(M))=2\gkd M$ in this setting \cite{Jo78}. Borho--Brylinski \cite{BoB1} showed that for highest weight modules with regular integral infinitesimal character this variety is the closure of a nilpotent orbit, and Joseph \cite{Jo85} extended the result to arbitrary infinitesimal character.

GK dimensions and annihilator varieties are closely related to Weyl group representations, Kazhdan--Lusztig cells, representations of Lie groups, and $W$-algebras; see, for example, \cite{BarV82,BMSZ-typeC,BMSZ-ABCD,BMSZ-counting,GS,GSK,GT,Mc96,Lo15,LY}. Nishiyama--Ochiai--Taniguchi \cite{NOT} and Enright--Willenbring \cite{EW} studied GK dimensions for unitary highest weight Harish-Chandra modules in the dual pair settings, while Bai--Hunziker \cite{Bai-Hu} obtained a uniform formula for unitary highest weight modules. Using Lusztig's $\aff$-function \cite{lusztig1977symbol,Lu79,Lusztig1982A-class-II,lusztig1984char,lusztig1985cellsI,lusztig1985A_n}, Bai--Xie \cite{BX} developed an algorithm in type $A$, extending Joseph's regular integral calculation \cite{Jo78}; Bai--Xiao--Xie \cite{BXX} subsequently gave explicit formulas for all classical types and for highest weight Harish-Chandra modules of exceptional type. 
GK dimensions are also useful in the study of reducibility of generalized Verma modules. Bai--Xiao \cite{BXiao} showed that the reducibility of a scalar generalized Verma module can be detected by comparing the GK dimension of the corresponding highest weight module with the dimension of the nilradical of the parabolic subalgebra.

For annihilator varieties with trivial infinitesimal character, type $A$ was treated by Joseph \cite{Jo-com}, types $B$ and $C$ by Garfinkle \cite{Garfinkle3}, and type $D$ by McGovern--Pietraho \cite{MP23}. Building on the Robinson--Schensted methods of \cite{BX,BXX}, Bai--Ma--Wang \cite{BMW} described the annihilator orbit attached to a highest weight module in the classical types.

For exceptional Lie algebras and arbitrary infinitesimal character, two features must be handled simultaneously. The integral root system of a nonintegral weight may be reducible, so the computation must be reduced to integral data for its irreducible components. In addition, the annihilator orbit depends not only on the abstract type of the integral root system but also on its embedding in the corresponding root system.

Our purpose is to make Joseph's description of annihilator varieties effective in exceptional type. The local calculation determines a special representation of the integral Weyl group. We then recover the annihilator orbit by a compatibility between $j$-induction and Sommers duality, keeping the embedding of the integral root system throughout the computation. In the following, we  provide some detailed descriptions of the algorithms.

\subsection{Computation of Gelfand--Kirillov dimensions}

The GK dimension of a highest weight module $L(\lam)$ with highest weight $\lam-\rho$ is determined by Lusztig's $\aff$-function on the integral Weyl group. More precisely,
\[
\gkd L(\lam)=|\Phi^+|-\aff(w_\lam),
\]
where $w_\lam$ is an element of the integral Weyl group $W_{[\lam]}$, with minimal length such that $w_\lam^{-1}\lam$ is antidominant, and $\aff(w_\lam)$ is the $\aff$-function value of $w_\lam$ in $W_{[\lam]}$. We recall this formula in Proposition \ref{pr:main1}, following \cite[Prop.~3.8]{BX}.



For arbitrary $\lambda$, we decompose the integral root system into irreducible components,
\[
\Phi_{[\lambda]}=\Phi_{[\lambda]_1}\sqcup\cdots\sqcup\Phi_{[\lambda]_k}.
\] Fix the root system isomorphisms
\[
\phi_i:\Phi_{[\lambda]_i}\simeq \Phi_i,
\qquad \phi_i(\Phi_{[\lambda]_i}^+)=\Phi_i^+,
\]
to standard irreducible root systems, and let $\lambda'_{(i)}$ be the transported restriction of $\lambda$ to $\Phi_i$. The parenthesized subscript records the irreducible factor, rather than a coordinate of a weight. Then every $\lambda'_{(i)}$ is integral for $\Phi_i$. As explained in Section \ref{a-value-compu}, this gives the following reduction for the $\aff$-function.

\begin{Thm}\label{T:main1}
For any weight $\lam$, with the notation above,
\[
\aff(w_{\lambda})=\sum_{1\lest i\lest k}\aff(w_{\lambda'_{(i)}}).
\]
\end{Thm}
Combined with
\[
\gkd L(\lam)=|\Phi^+|-\aff(w_\lam),
\]
the GK dimension of $L(\lambda)$ is therefore reduced to integral computations on irreducible root systems. For classical factors, the Robinson--Schensted (RS) insertion algorithm computes $\aff(w_{\lambda'_{(i)}})$ directly from $\lambda'_{(i)}$, without first determining $w_{\lambda'_{(i)}}$; see \cite{BXX}, building on \cite{BarV82}. For an exceptional factor, we first determine $w_{\lambda'_{(i)}}$ by the positive index reduction in Lemma \ref{find-w-lambda}. We then use Geck's package PyCox \cite{ge} to identify the unique special representation attached to its two-sided cell and compute its $\aff$-value.

\subsection{Annihilator varieties via Sommers duality}

By the results of Borho--Brylinski and Joseph \cite{BoB1,Jo85}, the annihilator variety
$V(\operatorname{Ann}(L(\lambda)))$ is the closure of a nilpotent orbit. We denote this orbit by
$\mathcal O_{{\rm Ann}(L(\lambda))}$ and call it the annihilator orbit. Moreover, it follows from \cite{Jo78} that
\[
\dim \mathcal O_{{\rm Ann}(L(\lambda))}
=
2\operatorname{GKdim}(L(\lambda)).
\]
Accordingly, the GK dimension determines the annihilator variety only when the corresponding orbit dimension occurs uniquely. Since distinct nilpotent orbits can have the same dimension, further information is required in general.

We use $\pi_w$ to denote the unique special representation of $W_{[\lambda]}$ corresponding to the two-sided cell containing $w$. For a Weyl group representation $\sigma$ occurring in the Springer correspondence with trivial local system, let $\mathcal O_{\rm Spr}^G(\sigma)$ denote the corresponding nilpotent orbit. By Joseph's result \cite[Thm.~3.10]{Jo85}, recalled in Proposition~\ref{2dim}, we have
\[
\mathcal{O}_{{\rm Ann}(L(\lambda))}
=
\mathcal{O}_{\rm Spr}^{G}\!\left(
j_{W_{[\lambda]}}^{W(G)}(\pi_{w_\lambda})
\right).
\]
Computing the $j$-induction in this formula directly can be cumbersome. We recall the relevant properties of $j$-induction and Sommers duality in Sections~\ref{subsec:j-induction} and~\ref{subsec:sommers-duality}. Theorem~\ref{P:Som} identifies the $j$-induction above with the corresponding Sommers duality, thereby reducing the computation to the dual integral root system. Let $G^\vee$ be the Langlands dual group. The dual integral root system $\Phi_{[\lambda]}^\vee$ is realized as the root system of a pseudo-Levi subgroup $H^\vee=C_{G^\vee}(t_\lambda)^\circ$, where $t_\lambda$ is a semisimple element corresponding to $\lambda$. Via the natural identification $W_{[\lambda]}\simeq W(H^\vee)$, let $\pi_{w_\lambda}^\vee$ be the special representation corresponding to $\pi_{w_\lambda}$, and let
\[
\mathcal O_\lambda^{H^\vee}:=\mathcal O_{\rm Spr}^{H^\vee}(\pi_{w_\lambda}^\vee)
\]
be the corresponding special nilpotent orbit of $H^\vee$.
Write the Lusztig--Spaltenstein dual of $\mathcal O_\lambda^{H^\vee}$ in Bala--Carter form as
\[
d_{\rm LS}^{H^\vee}(\mathcal O_\lambda^{H^\vee})
={\rm sat}_{L^\vee}^{H^\vee}(\mathcal O_{L^\vee}),
\]
where $L^\vee$ is a Levi subgroup of $H^\vee$ and $\mathcal O_{L^\vee}$ is a distinguished nilpotent orbit of $L^\vee$. Since $H^\vee$ is a pseudo-Levi subgroup of $G^\vee$, the subgroup $L^\vee$ is also a pseudo-Levi subgroup of $G^\vee$. Our compatibility theorem shows that the Sommers dual of the pair $(L^\vee,\mathcal O_{L^\vee})$ is precisely the annihilator orbit.

\begin{Thm}\label{T:main2}
With the notation above,
\[
\mathcal O_{{\rm Ann}(L(\lambda))}
=d_{\rm Som}^{G^\vee}(L^\vee,\mathcal O_{L^\vee}).
\]
\end{Thm}

This theorem follows from Joseph's formula \cite[Thm.~3.10]{Jo85} and Theorem \ref{P:Som}; see Corollary \ref{cor:direct-sommers}.

\subsection{Reduction to integral data}

Both computations start from the same local data. Let
\(
\pi_\lambda=\pi_{w_\lambda}
\)
and let
\(
\mathcal O_\lambda^H=\mathcal O_{\rm Spr}^{H}(\pi_\lambda)
\)
be the corresponding special nilpotent orbit for the group $H$ with root system $\Phi_{[\lambda]}$. Then
\[
\gkd L(\lambda)
=|\Phi^+|-|\Phi_{[\lambda]}^+|+\frac12\dim\mathcal O_\lambda^H.
\]
The same special representation, transferred to $H^\vee$, determines the annihilator orbit by Theorem \ref{T:main2}. Hence both invariants are controlled by the integral root system together with the local  two-sided cell.

The general reduction theorem is stated in Section \ref{S:gen}. If
\[
\Phi_{[\lambda]}=\Phi_{[\lambda]_1}\sqcup\cdots\sqcup\Phi_{[\lambda]_k},
\]
then the GK dimension is obtained from the integral weights $\lambda'_{(i)}$ by
\[
\gkd L(\lambda)=|\Phi^+|-\sum_i\aff(w_{\lambda'_{(i)}}),
\]
and the annihilator variety is obtained from the corresponding embedded pseudo-Levi datum by Lusztig--Spaltenstein duality followed by Sommers duality. The embedding
\[
\Phi_{[\lambda]}\hookrightarrow\Phi
\]
is retained throughout the computation. This distinction is essential in types $E_7$ and $E_8$, where root subsystems of the same abstract type can occur in nonconjugate embeddings.

The case-by-case analysis therefore reduces to the integral data for irreducible exceptional root systems. We treat $G_2,F_4,E_6,E_7,E_8$ in this order. In the integral case, once the special representation of the two-sided cell containing $w_\lambda$ is known, both invariants are determined:
\[
\gkd L(\lambda)=|\Phi^+|-\aff(\pi_{w_\lambda}),
\qquad
\mathcal O_{{\rm Ann}(L(\lambda))}=\mathcal O_{\rm Spr}^G(\pi_{w_\lambda}).
\]
The nonintegral cases are then obtained from the integral data together with the embedding information. In Section \ref{sec:pseudo-nonintegral}, we make this reduction explicit for pseudo-maximal nonintegral cases and conclude with some  examples illustrating the general reduction beyond the pseudo-maximal setting.

As in the classical case treated in \cite{BMW}, the algorithm is implemented in an online program. Given $\lambda$, the program returns $\gkd L(\lambda)$ and $\mathcal O_{{\rm Ann}(L(\lambda))}$. It is available at
\begin{center}
\url{http://liealgebra.slashblade.top/lie/GKdim}
\end{center}

\subsection{Organization}
In Section \ref{sec:sommers}, we recall Springer correspondence and $j$-induction and prove the compatibility with Sommers duality. In Section \ref{a-value-compu}, we describe the local $\aff$-function computation. In Section \ref{S:gen}, we  give the general reduction and the resulting algorithm.

In Section \ref{sec:embedding-examples}, we elaborate on Theorem \ref{thm:general-reduction}, which essentially arises from Theorem \ref{T:main2} above, and explain how the Sommers duality is utilized in the computation. In particular, we discuss how the embedding information is handled in practice. 

In Sections \ref{sec:E-cases} and \ref{sec:pseudo-nonintegral}, we give some integral and nonintegral exceptional calculations, respectively. For the latter, we focus on the pseudo-maximal cases and illustrate with some further examples. 

Finally, the appendix contains the computational reference data. We follow Bourbaki's notation and coordinates for root systems throughout \cite{Bour}.

\subsection*{Acknowledgments}
We would like to thank  Meinolf Geck for helpful discussions about PyCox. We would also like to thank Binyong Sun for an enlightening discussion on annihilator varieties.
Z. Bai  is partially supported  by NSFC Grant No. 12671039. F. Gao is partially supported by the National Key $\textrm{R}\,\&\,\textrm{D}$ Program of China (No. 2022YFA1005300) and  NSFC Grant No.  12171422. X. Xie is partially supported by 
NSFC Grant No. 12171030 and 12431002.

\subsection*{AI disclosure} AI was not involved in the first arXiv version of the paper. In preparing this revised version, ChatGPT was used primarily to shorten the manuscript and to check selected calculations and minor mathematical or typographical details. The mathematical ideas, conceptual framework, and proof strategies were developed by the authors, who verified the AI-assisted content and take full responsibility for the manuscript.
%
%

\section{Springer correspondence, \texorpdfstring{$j$-induction}{j-induction}, and Sommers duality}\label{sec:sommers}

In this section, we recall some basic results  used in the study of annihilator varieties and in the proof of Theorem \ref{P:Som}. For further details, see \cite{Ca85, lusztig1984char, som-98}.

Let $G$ be a connected complex reductive group with Lie algebra $\mathfrak g$. We fix a nondegenerate invariant symmetric bilinear form on $\mathfrak g$ and use it to identify $\mathfrak g$ with $\mathfrak g^*$, and hence adjoint with coadjoint nilpotent orbits.

\subsection{Special representations and Springer correspondence}

Recall that Lusztig singled out a subclass of representations of Weyl groups called \emph{special} in \cite{lusztig1984char}. The family attached to each two-sided cell contains a unique special representation. For classical types, special representations may be described by special symbols (equivalently, special partitions), while in exceptional types they are listed in \cite[Chap.~4]{lusztig1984char}.

A nilpotent orbit is called \emph{special} if it corresponds to a special representation under the Springer correspondence. Denote by $\mathcal{N}(G)$ the set of nilpotent orbits in $\mathfrak{g}$. For $\mathcal{O} \in \mathcal{N}(G)$, we write $A_\mathcal{O}:=G_{\rm ad}^x/(G_{\rm ad}^x)^\circ$ for the component group of the centralizer of any $x\in \mathcal{O}$ in the adjoint group $G_{\rm ad}$ of $G$. Consider
\[\mathcal{N}^{\rm en}(G):=\{(\mathcal{O}, \eta): \mathcal{O} \in \mathcal{N}(G), \eta \in {\rm Irr}(A_\mathcal{O}) \}.\]
We use the Springer correspondence in the convention of \cite[\S~13.3]{Ca85}. For the exceptional types, we also use the explicit character and orbit tables in \cite[Tables~C.3--C.6]{GP}. With this convention, the Springer correspondence gives an injective map
\begin{equation} \label{Spr}
  {\rm Spr}_G: {\rm Irr}(W(G)) \longrightarrow \mathcal{N}^{\rm en}(G)  
\end{equation}
denoted by ${\rm Spr}_G(\sigma) = (\mathcal{O}_{\rm Spr}^G(\sigma), \eta_\sigma)$. In particular ${\rm Spr}_G(\mathbf{1})=(\mathcal{O}_{\rm reg}, \mathbf{1})$ and ${\rm Spr}_G(\varepsilon_W)=(0, \mathbf{1})$. 
For every orbit $\mathcal{O} \in \mathcal{N}(G)$, the pair $(\mathcal{O}, \mathbf{1})$ lies in the image of ${\rm Spr}_G$; see, for example, \cite[\S~13.3]{Ca85}. Hence we have
a well-defined map
\[
{\rm Spr}_{G, \mathbf{1}}^{-1}: \mathcal{N}(G) \longrightarrow {\rm Irr}(W(G))
\] given by ${\rm Spr}_{G,\mathbf{1}}^{-1}(\mathcal{O}):={\rm Spr}_G^{-1}(\mathcal{O}, \mathbf{1})$. If $\sigma \in {\rm Irr}(W(G))$ is a special representation, then one has ${\rm Spr}_G(\sigma) = (\mathcal{O}_{\rm Spr}^G(\sigma), \mathbf{1})$, and we call $\mathcal{O}_{\rm Spr}^G(\sigma)$ a special orbit.

The Springer correspondence depends only on the Lie algebra $\mathfrak g$, so we may also write
\[{\rm Spr}_\mathfrak{g}:={\rm Spr}_G, \quad \mathcal{O}_{\rm Spr}^\mathfrak{g}:=\mathcal{O}_{\rm Spr}^G, \quad {\rm Spr}_{\mathfrak{g}, \mathbf{1}}^{-1}:={\rm Spr}_{G, \mathbf{1}}^{-1}.\]

\subsection{\texorpdfstring{$j$-induction}{j-induction}}\label{subsec:j-induction}

Let $W$ be the Weyl group attached to a root system $\Phi$ (with a fixed set of simple roots $\Delta$) in a finite-dimensional real vector space $V$. Then $W$ acts on the space $P_k(V)$ of degree $k$ homogeneous polynomials on $V$.  
Following Joseph \cite{J80-2}, a representation $\sigma\in {\rm Irr}(W)$ is called \emph{univalent} if it
occurs with multiplicity one in  $P_{b(\sigma)}(V)$ where $b(\sigma)$ is the minimal degree  such that $\sigma$ occurs in $P_{b(\sigma)}(V)$.
The number $b(\sigma)$ is the \emph{$b$-invariant} of $\sigma$.
Let ${\rm Irr}(W)^{\rm uv} \subseteq {\rm Irr}(W)$ be the set of univalent representations. 
Under the Springer correspondence, the representation associated with a nilpotent orbit and the trivial local system is always univalent \cite[Cor.~4]{BM}. Equivalently, with ${\rm Spr}_{G,\mathbf{1}}^{-1}$ defined above,
we have \[
{\rm Im}\bigl({\rm Spr}_{G,\mathbf{1}}^{-1}\bigr)
\subseteq {\rm Irr}(W)^{\rm uv}.
\]

Suppose $W'$ is a subgroup of $W$ generated by reflections in a root subsystem of $\Phi$. The $j$-induction (or truncated induction) used below is the standard one,  see \cite[\S~11.2]{Ca85} and \cite[Chap.~4]{lusztig1984char}. For $\sigma'\in{\rm Irr}(W')^{\rm uv}$, the representation $j_{W'}^W(\sigma')$ is the unique irreducible constituent of $\operatorname{Ind}_{W'}^W(\sigma')$ with $b$-invariant equal to $b(\sigma')$. Thus
\[
j_{W'}^W:{\rm Irr}(W')^{\rm uv}\longrightarrow{\rm Irr}(W)^{\rm uv},
\qquad
b\bigl(j_{W'}^W(\sigma')\bigr)=b(\sigma').
\]
The definitions of ``univalent'' and ``$j$-induction'' are independent of the choice of the reflection representation $V$. 

The $j$-induction satisfies the property of induction by stages as follows. 
\begin{Lem}[{\cite[Prop.~11.2.4]{Ca85}}]\label{j.1}
Let $W''\subseteq W'$ be two subgroups of $W$ generated by two root subsystems (which need not be parabolic subgroups). 
Then 
\[
j^W_{W''}=j^W_{W'}\circ j^{W'}_{W''}.
\]
\end{Lem}

\subsection{Sommers duality and a compatibility theorem}\label{subsec:sommers-duality}

We recall Sommers duality and fix the notation used below. Let $G$ be a connected reductive algebraic group over $\mathbb{C}$ with Lie algebra $\mathfrak{g}$. A pseudo-Levi subgroup $M\subseteq G$ is the connected centralizer of a semisimple element. Up to  Weyl group conjugacy, its root system is generated by a subset of the extended Dynkin diagram, see \cite{som-98,Sommers}. We consider $G$-conjugacy classes of pairs $(M,\mathcal O_M)$ with $M$ pseudo-Levi and $\mathcal O_M\in\mathcal N(M)$ distinguished. Let $G^\vee$ be the Langlands dual group of $G$. Sommers \cite{Sommers} defined a natural map
\[d_{\rm Som}^G: \{(M, \mathcal{O}_M) \}/G \longrightarrow \mathcal{N}(G^\vee)\]
that extends the Barbasch--Vogan dual map $d_{\rm BV}^G: \mathcal{N}(G) \to \mathcal{N}^{\rm spe}(G^\vee)$ \cite{BarV82}, where $\mathcal{N}^{\rm spe}(G) \subseteq \mathcal{N}(G)$ denotes the subset of special nilpotent orbits of $G$.

Recall also the Lusztig--Spaltenstein duality map \cite{Sp82}
$d_{\rm LS}^G: \mathcal{N}(G) \longrightarrow \mathcal{N}^{\rm spe}(G)$. Moreover, one has a natural isomorphism 
\[\tau_{G, G^\vee}: W(G) \longrightarrow W(G^\vee)\]
of the two Weyl groups, which induces a bijection (using the same notation)
\[\tau_{G, G^\vee}: {\rm Irr}(W(G)) \longrightarrow {\rm Irr}(W(G^\vee)).\]
Via the Springer correspondence, we use the same symbol for the induced bijection
\[\tau_{G, G^\vee}: \mathcal{N}^{\rm spe}(G) \longrightarrow \mathcal{N}^{\rm spe}(G^\vee).\]
One has $d_{\rm BV}^G:= \tau_{G, G^\vee} \circ d_{\rm LS}^G$. These maps fit into the following commutative diagram:
\[\begin{tikzcd}
\mathcal{N}^{\rm spe}(G) \ar[r, "{\tau_{G,G^\vee}}"] & \mathcal{N}^{\rm spe}(G^\vee)  \ar[r, hook] & \mathcal{N}(G^\vee) \\
& \mathcal{N}(G) \ar[lu, "{d_{\rm LS}^G}"] \ar[u, "{d_{\rm BV}^G}"] \ar[r, hook] & {\{(M, \mathcal{O}_M) \}/G } \ar[u, "{d_{\rm Som}^G}"].
\end{tikzcd}
\]
Here the bottom injection is provided by the Bala--Carter classification \cite{BC76,CM}: every $\mathcal{O} \in \mathcal{N}(G)$ corresponds uniquely to a $G$-conjugacy class of pairs $(L, \mathcal{O}_L)$, where $L \subseteq G$ is a Levi subgroup and $\mathcal{O}_L$ is distinguished in $L$. In this case, we write
\[{\rm sat}_L^G(\mathcal{O}_L):= G \cdot \mathcal{O}_L \in \mathcal{N}(G)\] for the nilpotent orbit thus obtained. Achar later gave a further extension of $d_{\rm Som}^G$ in \cite{ac03}.

Let $M \subseteq G$ be a pseudo-Levi subgroup and $\mathcal{O}_{M}$ a distinguished orbit for $M$. Then it was shown by Sommers \cite[\S~6]{Sommers} that 
\begin{equation} \label{E:dSom}
d_{\rm Som}^{G}(M, \mathcal{O}_{M})
=
\mathcal{O}_{\rm Spr}^{G^\vee}\!\left(
\tau_{G,G^\vee}\!\left(
j_{W(M)}^{W(G)}\!\left(
{\rm Spr}_{M,\mathbf{1}}^{-1}\!\left(d_{\rm LS}^{M}(\mathcal{O}_{M})\right)
\right)
\right)
\right).
\end{equation}

The dual group $M^\vee$ of a pseudo-Levi subgroup $M$ need not itself be pseudo-Levi in $G^\vee$, but there is a natural isomorphism $W(M)\simeq W(M^\vee)$. We regard
\[
W(M^\vee)=\tau_{G,G^\vee}(W(M))\subseteq W(G^\vee)
\]
as the reflection subgroup generated by the dual root subsystem. For $\sigma\in{\rm Irr}(W(M))^{\rm uv}$, set $\sigma^\vee:=\tau_{M,M^\vee}(\sigma)$. Then
\begin{equation} \label{E:transp}
j_{W(M^\vee)}^{W(G^\vee)} (\sigma^\vee)
=\tau_{G,G^\vee}\!\left(j_{W(M)}^{W(G)}(\sigma)\right).
\end{equation}

\begin{Thm} \label{P:Som}
With the notation above, assume $\sigma \in {\rm Irr}(W(M))$ is a special representation of $W(M)$. Furthermore, let $L\subseteq M$ be a Levi subgroup of $M$ and $\mathcal{O}_L$ be a distinguished orbit of $L$ such that $d_{\rm LS}^M\bigl(\mathcal{O}_{\rm Spr}^M(\sigma)\bigr) = {\rm sat}_L^M(\mathcal{O}_L)$.
Then 
\[\mathcal{O}_{\rm Spr}^{G^\vee}\!\left(j_{W(M^\vee)}^{W(G^\vee)}(\sigma^\vee)\right) = d_{\rm Som}^G(L, \mathcal{O}_L).\]
\end{Thm}
\begin{proof}

Since $L$ is a Levi subgroup of the pseudo-Levi subgroup $M$, it is itself a pseudo-Levi subgroup of $G$ by \cite[Lem.~3.9]{CE12} and hence the right-hand side is well-defined.


By \eqref{E:transp},
\[\mathcal{O}_{\rm Spr}^{G^\vee}\!\left(j_{W(M^\vee)}^{W(G^\vee)}(\sigma^\vee)\right) = \mathcal{O}_{\rm Spr}^{G^\vee}\!\left(\tau_{G,G^\vee}\!\left(j_{W(M)}^{W(G)}(\sigma)\right)\right).\]
Since $\sigma$ is special, $\mathcal O_{\rm Spr}^M(\sigma)$ is a special nilpotent orbit, and $d_{\rm LS}^M$ restricts to an involution on the set of special orbits of $M$. Applying $d_{\rm LS}^M$ to the hypothesis therefore gives
\[\mathcal O_{\rm Spr}^M(\sigma)=d_{\rm LS}^M\bigl({\rm sat}_{L}^M(\mathcal O_L)\bigr),\]
and hence
\[\sigma = {\rm Spr}_{M, \mathbf{1}}^{-1}\!\left(d_{\rm LS}^M\bigl({\rm sat}_{L}^M(\mathcal{O}_L)\bigr)\right).\]
Consider the induction of nilpotent orbits ${\rm ind}_L^M: \mathcal{N}(L) \longrightarrow \mathcal{N}(M)$. By \cite[Thm.~8.3.1]{CM} (see also the proof of Proposition~7 in \cite[\S~6]{Sommers}), one has
\[
d_{\rm LS}^M\bigl({\rm sat}_L^M(\mathcal{O}_L)\bigr)
=
{\rm ind}_L^M\bigl(d_{\rm LS}^L(\mathcal{O}_L)\bigr).
\] On the other hand, the Springer correspondence intertwines the maps  $j_{W(L)}^{W(M)}$ and ${\rm ind}_L^M$ (see \cite{LuSp1}), i.e.,
\[{\rm Spr}_{M, \mathbf{1}}^{-1} \circ {\rm ind}_L^M = j_{W(L)}^{W(M)} \circ {\rm Spr}_{L, \mathbf{1}}^{-1}.\]
Combining these identities with $j_{W(M)}^{W(G)} \circ j_{W(L)}^{W(M)} = j_{W(L)}^{W(G)}$ yields
\[
\begin{aligned}
    \mathcal{O}_{\rm Spr}^{G^\vee}\!\left(j_{W(M^\vee)}^{W(G^\vee)}(\sigma^\vee)\right)
    & = \mathcal{O}_{\rm Spr}^{G^\vee}\!\left(
    \tau_{G,G^\vee}\!\left(
    j_{W(L)}^{W(G)}\!\left(
    {\rm Spr}_{L,\mathbf{1}}^{-1}\!\left(d_{\rm LS}^L(\mathcal{O}_L)\right)
    \right)
    \right)
    \right) \\
     & = d_{\rm Som}^G(L, \mathcal{O}_L),
\end{aligned}\]
where the last equality is \eqref{E:dSom}.

 This completes the proof.                                      
\end{proof}

\section{Computing Lusztig's \texorpdfstring{$\aff$}{a}-function}\label{a-value-compu}

In this section, we prove Theorem \ref{T:main1} and describe the $\aff$-function computations used later. The relation with GK dimensions is recorded in Section \ref{sec:pre}.


\subsection{Integral Weyl groups and Lusztig's \texorpdfstring{$\aff$}{a}-function}

Let $\mathfrak{g}$ be a complex simple Lie algebra with Cartan subalgebra $\mathfrak{h}$, root system $\Phi$, positive system $\Phi^+$, simple system $\Delta$, and Weyl group $W$. 

Let $\langle-, -\rangle: \mathfrak{h}^* \times \mathfrak{h} \to \mathbb{C}$ be the canonical pairing.
For $\lambda\in\mathfrak{h}^*$, set
\[
\Phi_{[\lambda]}=\{\alpha\in\Phi:\langle\lambda,\alpha^\vee\rangle\in\mathbb Z\},
\qquad
W_{[\lambda]}=\{w\in W:w\lambda-\lambda\in\mathbb Z\Phi\}.
\]
Then $\Phi_{[\lambda]}$ is a root system with Weyl group $W_{[\lambda]}$ by \cite{Hum08}. Let $\Delta_{[\lambda]}$ be the simple system of $\Phi_{[\lambda]}\cap\Phi^+$. 

After  decomposing $\Phi_{[\lambda]}$ into orthogonal irreducible components, the simple system of each component is determined uniformly by the following observation.

\begin{Lem}\label{integral-simple-system}
Let $\Psi$ be an irreducible component of $\Phi_{[\lambda]}$, put $\Psi^+=\Psi\cap\Phi^+$, and set
\[
\rho_\Psi=\frac12\sum_{\alpha\in\Psi^+}\alpha.
\]
Then
\[
\Delta_\Psi=\{\alpha\in\Psi^+:\langle\rho_\Psi,\alpha^\vee\rangle=1\}
\]
is the simple system of $\Psi^+$.
\end{Lem}

\begin{proof}
For a positive root $\alpha$ of $\Psi$, the integer $\langle\rho_\Psi,\alpha^\vee\rangle$ is the height of the positive coroot $\alpha^\vee$. It is equal to $1$ precisely when $\alpha$ is simple.
\end{proof}

We call $\mu$ antidominant if
\[
\langle\mu,\alpha^\vee\rangle\lest 0\qquad(\alpha\in\Delta_{[\mu]}).
\]
When $\lambda$ is integral for $G$, we use the standard terminology
\[
\lambda\text{ dominant integral}\iff
\langle\lambda,\alpha_i^\vee\rangle\in\mathbb Z_{\geq 0}
\quad(\alpha_i\in\Delta),
\]
and
\[
\lambda\text{ regular (or strictly) dominant integral}\iff
\langle\lambda,\alpha_i^\vee\rangle\in\mathbb Z_{>0}
\quad(\alpha_i\in\Delta).
\]
Since throughout the paper $L(\lambda)$ has highest weight $\lambda-\rho$, one has
\[
L(\lambda)\text{ is finite-dimensional}
\quad\Longleftrightarrow\quad
\lambda\text{ is regular dominant integral}.
\]
We use this shifted convention throughout.

 Set $J=\{\alpha\in\Delta_{[\mu]}\mid\langle\mu, \alpha^\vee\rangle=0\}$. Denote by $W_J$ the Weyl group generated by reflections $s_\alpha$ with $\alpha\in J$. Let $\ell_{[\mu]}$ be the length function on $W_{[\mu]}$ with respect to the chosen $\Delta_{[\mu]}$. In particular, we have $\ell_{[\mu]}=\ell$, the length function on $W$, when $\mu$ is integral. Put
\begin{equation*}\label{ceq1}
	W_{[\mu]}^J:=\{w\in W_{[\mu]}\mid \ell_{[\mu]}(ws_\alpha)=\ell_{[\mu]}(w)+1\ \mbox{for all}\ \alpha\in J\}.
\end{equation*}
Thus $W_{[\mu]}^J$ consists of the shortest representatives of the cosets $wW_J$ with $ w\in W_{[\mu]} $.

From \cite[Prop.~3.5]{Hum08}, for any $\lambda\in\mathfrak{h}^*$, there exists a unique antidominant weight $\mu\in\hs$ and a unique $w_{\lambda}\in W_{[\mu]}^J$ such that $\lambda=w_{\lambda}\mu$.

Now we recall Lusztig's $\aff$-function \cite{Lus03}.
Let $S=\{s_\alpha:\alpha\in\Delta\}$ and $A=\mathbb Z[v,v^{-1}]$, where $v$ is an indeterminate. The Hecke algebra $\mathcal H=\mathcal H_A(W)$ is the $A$-algebra generated by $T_s$ ($s\in S$), subject to the braid relations and
\[
(T_s-v)(T_s+v^{-1})=0\qquad(s\in S).
\]
For a reduced expression $w=s_1\cdots s_r$, put
\[
T_w=T_{s_1}\cdots T_{s_r},\qquad T_e=1.
\]
This element is independent of the reduced expression, and $\{T_w:w\in W\}$ is the standard $A$-basis of $\mathcal H$. The bar involution is determined by $\overline v=v^{-1}$ and $\overline{T_w}=T_{w^{-1}}^{-1}$. We write $\{C_w:w\in W\}$ for the Kazhdan--Lusztig basis of $\mathcal H$.
Write
\[
C_xC_y=\sum_{z\in W}h_{x,y,z}C_z.
\]
With $\deg$ denoting the highest degree in $v$, we have
\[
\aff(z)=\max\{\deg h_{x,y,z}:x,y\in W\}.
\]
We use the standard notation $\leq_L,\leq_R,\leq_{LR}$ and
$\sim_L,\sim_R,\sim_{LR}$ for the Kazhdan--Lusztig preorders and cell relations. The following standard properties are well-known; see, for example, \cite[\S~15]{Lus03}.

\begin{Lem}\label{alem1}
Let $x,w\in W$. Then
\begin{itemize}
\item[(1)] $\aff(w)=\aff(w^{-1})$.
\item[(2)] If $x\lest_{LR}w$, then $\aff(x)\geq\aff(w)$. Hence $\aff(x)=\aff(w)$ whenever $x\sim_{LR}w$.
\item[(3)] If $w_I$ is the longest element of the parabolic subgroup generated by $I\subseteq S$, then $\aff(w_I)=\ell(w_I)$.
\item[(4)] If $W=W_1\times W_2$ and $w=(w_1,w_2)$, then
\[
\aff(w)=\aff(w_1)+\aff(w_2).
\]
\end{itemize}
\end{Lem}

If $\pi$ is the special representation attached to a two-sided cell $\mathcal C$, we write $\aff(\pi)$ for the common $\aff$-value of the elements of $\mathcal C$. With this convention, $\aff(\pi_w)=\aff(w)$ whenever $\pi_w$ denotes the special representation attached to the two-sided cell of $w$.

We also put
\[
\mathcal C=\{w\in W\setminus\{1\}:w\text{ has a unique reduced expression}\},
\qquad
\mathcal Cw_0=\{ww_0:w\in\mathcal C\},
\] 
where $w_0$ denotes the longest element of $W$.
By \cite[Prop.~3.8]{Lus83}, $\mathcal C$ and $\mathcal Cw_0$ are two-sided cells.

\subsection{Proof of Theorem \ref{T:main1}}\label{nonint}

The proof is based on two elementary compatibilities: the $\aff$-function is invariant under isomorphisms of Coxeter systems, and it is additive for direct products.

\begin{Thm}\label{transf}
Let $\Psi$ and $\Psi'$ be isomorphic root systems, and let $\phi:\Psi\rightarrow\Psi'$ be an isomorphism carrying the chosen positive system of $\Psi$ to that of $\Psi'$. By abuse of notation, also write $\phi$ for the induced isomorphisms of the corresponding Cartan spaces, their duals, and the Weyl groups. If $\nu$ is an integral weight for $\Psi$, then $\phi(\nu)$ is an integral weight for $\Psi'$ and
\[
\phi(w_\nu)=w_{\phi(\nu)},\qquad
\aff(w_\nu)=\aff(w_{\phi(\nu)}).
\]
\end{Thm}


\begin{proof}
Write $W=W(\Psi)$ and $W'=W(\Psi')$, and let $\Delta$ and
$\Delta'$ be the chosen simple systems. Since $\phi$ preserves
positive systems, $\phi(\Delta)=\Delta'$. The induced maps satisfy
\[
\langle\phi(\eta),\phi(\alpha)^\vee\rangle
=\langle\eta,\alpha^\vee\rangle,
\qquad
\phi(s_\alpha\eta)=s_{\phi(\alpha)}\phi(\eta).
\]
In particular, $\phi$ preserves integrality and carries
antidominant weights to antidominant weights.

Write $\nu=w_\nu\mu$, where $\mu$ is antidominant, and put
\[
J=\{\alpha\in\Delta:
       \langle\mu,\alpha^\vee\rangle=0\},
\qquad
J'=\phi(J).
\]
Then $J'$ is the set of simple roots whose coroots pair trivially
with $\phi(\mu)$. The induced Coxeter system isomorphism
$W\longrightarrow W'$ preserves length and carries $W_J$ onto
$W'_{J'}$. Hence it carries the shortest representatives of the
right cosets in $W/W_J$ to those in $W'/W'_{J'}$.
Since $w_\nu\in W^J$, we obtain
\[
\phi(\nu)=\phi(w_\nu)\phi(\mu),
\qquad
\phi(w_\nu)\in (W')^{J'}.
\]
The uniqueness of the antidominant representative and the
corresponding shortest coset representative therefore gives
\[
\phi(w_\nu)=w_{\phi(\nu)}.
\]

Let $\mathcal H=\mathcal H_A(W)$ and $\mathcal H'=\mathcal H_A(W')$ have standard bases $\{T_w:w\in W\}$ and $\{T'_{w'}:w'\in W'\}$, respectively. Since $\phi$ preserves the Coxeter matrix, the assignment
\[
T_w\longmapsto T'_{\phi(w)}
\]
respects the braid and quadratic relations and defines an $A$-algebra isomorphism.
It commutes
with the bar involutions and preserves Bruhat order. By the
bar-invariance and triangularity characterization of the
Kazhdan--Lusztig basis, it sends $C_w$ to $C'_{\phi(w)}$.

Writing the structure constants of these bases as $h_{x,y,z}$
and $h'_{x',y',z'}$, respectively, we consequently have
\[
h'_{\phi(x),\phi(y),\phi(z)}=h_{x,y,z}
\qquad(x,y,z\in W).
\]
Since $\phi:W\longrightarrow W'$ is bijective and the parameter
is unchanged, the definition of Lusztig's $\aff$-function gives
\[
\aff(\phi(z))
=\max_{x,y\in W}
  \deg h'_{\phi(x),\phi(y),\phi(z)}
=\max_{x,y\in W}\deg h_{x,y,z}
=\aff(z).
\]
Taking $z=w_\nu$ and using the first equality proves
$\aff(w_\nu)=\aff(w_{\phi(\nu)})$.
\end{proof}
Now decompose
\[
\Phi_{[\lambda]}=\Phi_{[\lambda]_1}\sqcup\cdots\sqcup\Phi_{[\lambda]_k}
\]
into irreducible components. Choose root system isomorphisms
\[
\phi_i:\Phi_{[\lambda]_i}\xrightarrow{\sim}\Phi_i,
\qquad \phi_i(\Phi_{[\lambda]_i}^+)=\Phi_i^+
\]
to standard irreducible root systems, and put
\[
\lambda'_{(i)}=\phi_i\bigl(\lambda|_{\mathfrak h_{\Phi_{[\lambda]_i}}}\bigr).
\]
Here the parenthesized subscript $(i)$ labels the irreducible factor. Each $\lambda'_{(i)}$ is an integral weight of type $\Phi_i$.

\begin{Cor}\label{transf-1}
Under the natural isomorphism
\[
W_{[\lambda]}\simeq W(\Phi_1)\times\cdots\times W(\Phi_k),
\]
the element $w_\lambda$ corresponds to
\[
\bigl(w_{\lambda'_{(1)}},\ldots,w_{\lambda'_{(k)}}\bigr).
\]
Consequently,
\[
\aff(w_\lambda)=\sum_{1\lest i\lest k}\aff(w_{\lambda'_{(i)}}).
\]
\end{Cor}

\begin{proof}
The antidominant chamber for $\Phi_{[\lambda]}$ is the product of the antidominant chambers of its irreducible components. Hence the minimal element carrying the antidominant representative to $\lambda$ is the product of the corresponding minimal elements for the restrictions of $\lambda$. Theorem \ref{transf} identifies these elements with $w_{\lambda'_{(i)}}$. The equality of $\aff$-values follows from Lemma \ref{alem1}(4).
\end{proof}

\begin{proof}[Proof of Theorem \ref{T:main1}]
Let
\[
\lambda'=\phi(\lambda|_{\mathfrak{h}_{\Phi_{[\lambda]}}})=\bigl(\lambda'_{(1)},\ldots,\lambda'_{(k)}\bigr)
\]
under the chosen isomorphism $\Phi_{[\lambda]}\simeq\prod_i\Phi_i$. Theorem \ref{transf} and Corollary \ref{transf-1} give
\[
\aff(w_\lambda)=\aff(w_{\lambda'})
=\sum_{1\lest i\lest k}\aff(w_{\lambda'_{(i)}}),
\]
as required.
\end{proof}

\subsection{Integral factors: positive index reduction}\label{inte-a}

We next compute the $\aff$-value for each integral factor. For an integral weight of classical type, the Robinson--Schensted algorithms recalled below allow us to compute the $\aff$-value directly. For an exceptional factor, however, there is no analogous RS algorithm, and our first task is to determine explicitly the Weyl group element $w_\lambda$. The following elementary reduction provides an efficient way to do this. Once $w_\lambda$ is obtained, its two-sided cell and the corresponding special representation can be determined by PyCox, and hence so can its $\aff$-value.

We number the simple roots as
$\Delta=\{\alpha_1,\ldots,\alpha_n\}$ according to Bourbaki's convention,
and let $\{\omega_i\}$ be the corresponding fundamental weights. We write $s_i=s_{\alpha_i}$ for the corresponding simple reflections. Write
\[
\lambda=\sum_{1\lest i\lest n}k_i\omega_i,
\qquad
k_i=\langle\lambda,\alpha_i^\vee\rangle\in\mathbb Z.
\]
Thus the fundamental weight coordinates of $\lambda$ are precisely the simple coroot pairings. If all $k_i\lest0$, then $\lambda$ is antidominant and $w_\lambda=1$. Otherwise choose an index $i$ with $k_i>0$ and replace $\lambda$ by
\[
s_i\lambda=\lambda-k_i\alpha_i.
\]
With the Cartan matrix convention $A_{ij}=\langle\alpha_i,\alpha_j^\vee\rangle$, the new coordinates are
\[
k_j'=\langle s_i\lambda,\alpha_j^\vee\rangle
      =k_j-k_iA_{ij}\qquad(1\lest j\lest n).
\]
In particular, the $i$-th coordinate changes from $k_i$ to $-k_i$. We repeat this operation whenever a positive coefficient occurs. We call this procedure the \emph{positive index reduction}. Geometrically, at each step we cross a wall on the positive side of the current weight and move toward the antidominant Weyl chamber. The following lemma shows that the procedure terminates and that the product of the chosen reflections is the element $w_\lambda$ required in the GK dimension formula.

\begin{Lem}\label{find-w-lambda}
For any integral weight $\lambda$, the positive index reduction terminates at an antidominant weight $\mu$. If the successive reflections are $s_{i_1},\ldots,s_{i_r}$, then
\[
w_\lambda=s_{i_1}s_{i_2}\cdots s_{i_r}
\]
is the minimal length element such that $w_\lambda^{-1}\lambda$ is antidominant.
\end{Lem}

\begin{proof}
This follows from \cite[Lem.~3.1]{do23};
see also \cite[Thm.~4.3.1(iv)]{BB05}
and \cite[Prop.~4.1]{Er95}.
We include the short argument needed here.

For a weight $\nu$ in the $W$-orbit of $\lambda$, put
\[
N(\nu)=\#\{\beta\in\Phi^+:\langle\nu,\beta^\vee\rangle>0\}.
\]
Suppose that $\langle\nu,\alpha_i^\vee\rangle>0$. Since
$s_i$ permutes $\Phi^+\setminus\{\alpha_i\}$ and sends
$\alpha_i$ to $-\alpha_i$, while
\[
\langle s_i\nu,\beta^\vee\rangle
=
\langle\nu,(s_i\beta)^\vee\rangle,
\]
the contributions from $\Phi^+\setminus\{\alpha_i\}$ are merely
permuted. On the other hand,
\[
\langle s_i\nu,\alpha_i^\vee\rangle
=
-\langle\nu,\alpha_i^\vee\rangle<0.
\]
Since $\alpha_i$ contributes to $N(\nu)$, it follows that
\[
N(s_i\nu)=N(\nu)-1.
\]
Hence each step of the positive index reduction decreases the nonnegative integer $N(\nu)$ by one, so the process will stop after finitely many steps. At the final weight $\mu$, all pairings with simple coroots are nonpositive; hence $\mu$ is antidominant. Moreover, if $r$ reflections have been used, then $r=N(\lambda)$ and
\[
\mu=s_{i_r}\cdots s_{i_1}\lambda=w_\lambda^{-1}\lambda.
\]

It remains to prove minimality. Let $x\in W$ be any element such that $x^{-1}\lambda$ is antidominant. If $\beta\in\Phi^+$ satisfies $\langle\lambda,\beta^\vee\rangle>0$, then
\[
\langle x^{-1}\lambda,(x^{-1}\beta)^\vee\rangle>0.
\]
Since $x^{-1}\lambda$ is antidominant, this forces $x^{-1}\beta\in\Phi^-$. Thus
\[
\{\beta\in\Phi^+:\langle\lambda,\beta^\vee\rangle>0\}
\subseteq
\{\beta\in\Phi^+:x^{-1}\beta\in\Phi^-\},
\]
and consequently $N(\lambda)\lest\ell(x)$. On the other hand, the above expression for $w_\lambda$ uses exactly $r=N(\lambda)$ simple reflections, so $\ell(w_\lambda)\lest r$. Applying the preceding inequality to $w_\lambda$ gives
\[
r=N(\lambda)\lest\ell(w_\lambda)\lest r.
\]
Therefore $\ell(w_\lambda)=r$, and $w_\lambda$ has minimal length among all $x\in W$ for which $x^{-1}\lambda$ is antidominant.
\end{proof}

\begin{Rem}
By \cite[Lem.~3.1]{do23}, one may choose any index whose coefficient is positive at each step. The endpoint $\mu$ is the unique antidominant weight in $W\lambda$.
The product $w_\lambda=s_{i_1}\cdots s_{i_r}$ is the unique
element of minimal length in the coset
\[
w_\lambda W_J=\{x\in W:x\mu=\lambda\}.
\]
Consequently, $w_\lambda$ is independent of the choices made
during the reduction, although its reduced expression may vary. For convenience, we have designed   the following web page for an interested reader to apply our algorithm:
    \begin{center}
\textcolor{blue}{http://test.slashblade.top:5000/lie/antidominant}
    \end{center}
We write $\lam=\sum\limits_{0\lest i\lest n-1}k_i\omega_i:=[k_0,\cdots,k_{n-1}]$, where $k_i=\langle \lam, \alpha_i^{\vee}\rangle$. Here  the order of simple roots is the same as that in \cite{ge}. In our web page, we only need to input $``k_0,\cdots,k_{n-1}"$, and the corresponding $w_{\lambda}$ will appear as the output.
\end{Rem}


\begin{ex}\label{ex:e7-positive-index}
Let $\lam=(2,-1,0,-5,-6,-8,12,-12)$ be an integral weight of type $E_7$. In fundamental weight coordinates,
\[
\lam=[-1,1,-3,1,-5,-1,-2].
\]
Choosing successively the positive coordinates with indices $4,2,4$ in the Bourbaki numbering \cite{Bour} gives
\[
\begin{aligned}
s_4\lam&=[-1,2,-2,-1,-4,-1,-2],\\
s_2s_4\lam&=[-1,-2,-2,1,-4,-1,-2],\\
s_4s_2s_4\lam&=[-1,-1,-1,-1,-3,-1,-2].
\end{aligned}
\]
The last weight is antidominant, so $w_\lam=s_4s_2s_4$ by Lemma \ref{find-w-lambda}. 
\end{ex}


\subsection{Classical factors and the Robinson--Schensted algorithm}\label{RS-algorithm}

For a classical irreducible factor, one can compute $\aff(w_\lambda)$ directly from the integral weight $\lambda$, without first finding $w_\lambda$. We recall the formulation from \cite{BXX}.

For  a totally ordered set $ \Gamma $, we  denote by $ \mathrm{Seq}_n (\Gamma)$ the set of sequences $ x=(x_1,x_2,\cdots, x_n) $   of length $ n $ with $ x_i\in\Gamma $. Throughout, we usually  use $\Gamma=\mathbb Z$ or a coset $z+\mathbb Z\subset\mathbb C$, ordered by $z+m\leq z+n$ if and only if $m\leq n$.
 \begin{Defff}[Robinson--Schensted insertion algorithm]
 For an element  $ x \in  \mathrm{Seq}_n (\Gamma)$, we write  $x=(x_1,\dots,x_n)$. We associate to $x $ a  Young tableau  $ P(x) $ as follows. Let $ P_0 $ be an empty Young tableau. Assume that we have constructed a Young tableau $ P_k $ associated to $ (x_1,\dots,x_k) $, $ 0\leq k<n $. Then $ P_{k+1} $ is obtained by adding $ x_{k+1} $ to $ P_k $ as follows. Firstly we add $ x_{k+1} $ to the first row of $ P_k $ by replacing the leftmost entry $ x_j $ in the first row which is \textit{strictly} bigger than $ x_{k+1} $.  (If there is no such  entry $ x_j $, we just add a box with entry $x_{k+1}  $ to the right side of the first row, and end this process). Then add this $ x_j $ to the next row in the same way of adding $x_{k+1} $ to the first row.  Finally we put $P(x)=P_n$.
\end{Defff}

The algorithm was introduced by Robinson \cite{Rob38} and independently rediscovered in a different form by Schensted \cite{Sc61}. 
 
We use $p(x)=[p_1,\dots, p_k]$ to denote the shape of $P(x)$, where $p_i$ is the number of boxes in the $i$-th row of  $P(x)$.
Since each insertion adds one box, $p(x)$ is a partition of $n$.


	For a Young diagram $P$, use $ (k,l) $ to denote the box in the $ k $-th row and the $ l $-th column.
	We say that the box $ (k,l) $ is \textit{even} (resp. \textit{odd}) if $ k+l $ is even (resp. odd). Let $ p_i ^{\rm ev}$ (resp. $ p_i^{\rm od} $) be the numbers of even (resp. odd) boxes in the $ i $-th row of the Young diagram $ P $.
	A direct count gives
	\begin{equation}\label{eq:ev-od}
	p_i^{\rm ev}=\begin{cases}
	\left\lceil \frac{p_i}{2} \right\rceil,&\text{ if } i \text{ is odd},\\
	\left\lfloor \frac{p_i}{2} \right\rfloor,&\text{ if } i \text{ is even},
	\end{cases}
	\quad p_i^{\rm od}=\begin{cases}
	\left\lfloor \frac{p_i}{2} \right\rfloor,&\text{ if } i \text{ is odd},\\
	\left\lceil \frac{p_i}{2} \right\rceil,&\text{ if } i \text{ is even}.
	\end{cases}
	\end{equation}
	Here for $ a\in \mathbb{R} $, $ \lfloor a \rfloor $ is the largest integer $ n $ such that $ n\leq a $, and $ \lceil a \rceil$ is the smallest integer $n$ such that $ n\geq a $. For convenience, we set
	\begin{equation*}
	p^{\rm ev}=(p_1^{\rm ev},p_2^{\rm ev},\cdots)\quad\mbox{and}\quad p^{\rm od}=(p_1^{\rm od},p_2^{\rm od},\cdots).
	\end{equation*}
	
	For $ x=(x_1,x_2,\cdots,x_n)\in \mathrm{Seq}_n (\Gamma) $, set
	\begin{equation*}
	\begin{aligned}
	{x}^-=&(x_1,x_2,\cdots,x_{n-1}, x_n,-x_n,-x_{n-1},\cdots,-x_2,-x_1).
	\end{aligned}
	\end{equation*}
	
	\begin{Prop}[{\cite[Thm.~1.5]{BXX}}]\label{integral}
		Let $\lambda=(\lambda_1, \lambda_2, \cdots, \lambda_n)\in \mathfrak{h}^*$ be  integral. Then
		
		\[	\aff(w_{\lambda})=\left\{
		\begin{array}{ll}
		\sum\limits_{i\gest 1}(i-1)p(\lambda)_i, &\textnormal{if}\:\Phi =A_{n-1},\\
        \sum\limits_{i\gest 1}(i-1)p((\lambda)^-)_i^{\rm od}, &\textnormal{if}\:\Phi=B_{n}/C_n,\\
		\sum\limits_{i\gest 1}(i-1)p((\lambda)^-)_i^{\rm ev}, &\textnormal{if}\:\Phi=D_{n}.
		\end{array}	
		\right.
		\]

	\end{Prop}

\subsection{Exceptional factors and PyCox}\label{subsec:pycox}



For an exceptional irreducible factor, positive index reduction
determines the required Weyl group element $w$. We then use
PyCox \cite{ge} to compute its $\aff$-value and identify the
special representation $\pi_w$ attached to its two-sided cell.
The function \texttt{klcellrepelm} returns the $\aff$-value,
the left-cell character, and the special representation in
separate fields. We use the field \texttt{special} to identify
$\pi_w$; the left-cell character may be reducible.

We record products of simple reflections using Geck's indexing
convention \cite{ge}. The product $w=s_{i_1}\cdots s_{i_k}$ is
represented by
\[
\begin{cases}
[i_1-1,\ldots,i_k-1],
& \text{in types } A_n,E_n,F_4,G_2,\\
[n-i_1,\ldots,n-i_k],
& \text{in types } B_n,C_n,D_n.
\end{cases}
\]

For irreducible characters, we use Carter's labels
$\phi_{d,b}$, with primes where needed, as in
\cite[\S~13.3]{Ca85}. When recording computational data,
we give the corresponding PyCox labels in parentheses.

\begin{ex}
    In the case of $F_4$, we use $w=[1,3,2,1]$ to denote the element $w=s_2s_4s_3s_2$. Using PyCox, we have:
   \begin{verbatim}
    >>> W = coxeter("F", 4)
    >>> print(klcellrepelm(W, [1,3,2,1]))
    {'size': 9, 'character': [['9_1', 1]], 'a': 2,
    'special': '9_1', 'index': 22, 'elms': False, 
    'distinv': False}
\end{verbatim}

Thus $\aff(w)=2$. In this example the left cell character is already irreducible, and the field ``special'' is $9_1$. Thus, in the convention used throughout the paper, $\pi_w=\phi_{9,2}\;(9_1)$, where the first label is Carter's and the parenthesized label is the PyCox label; see \cite[Table~C.3]{GP} and \cite[p.~428]{Ca85}. The corresponding special nilpotent orbit is $\mathcal{O}_w=F_4(a_2)$.
\end{ex}

	


Thus, for exceptional types, PyCox supplies the $\aff$-value needed in Theorem \ref{T:main1}, and when necessary, it also determines the corresponding special representation and nilpotent orbit.

\begin{Rem}\label{represent}
 In PyCox, we can find a representative element for a given $\bf a$ value $k$ in the following way:
    \begin{verbatim}
>>>W = coxeter("F", 4)
>>>left_cell_infos = klcellreps(W)
... for left_cell_info in left_cell_infos:
... if left_cell_info['a'] == 10:
>>>print(f"Index: {left_cell_info['index']}")
Index: 19
Index: 21
Index: 23
...for elm in left_cell_infos[19]['elms']:
>>>print(W.coxelmtoword(elm))
[1, 2, 1, 2, 3, 2, 1, 0, 2, 1, 3, 2, 1, 0]
[0, 1, 2, 1, 2, 3, 2, 1, 0, 2, 1, 3, 2, 1, 0]
[1, 0, 2, 1, 0, 3, 2, 1, 0, 2, 1, 2, 3, 2, 1, 0, 2, 1, 2]
[0, 2, 1, 0, 3, 2, 1, 0, 2, 1, 2, 3, 2, 1, 0, 2, 1, 2]
[0, 1, 0, 3, 2, 1, 0, 2, 1, 2, 3, 2, 1, 0, 2, 1, 2]
[0, 1, 0, 2, 1, 2, 3, 2, 1, 0, 2, 1, 3, 2, 1, 0]
[0, 2, 1, 0, 2, 1, 2, 3, 2, 1, 0, 2, 1, 3, 2, 1, 0]
[0, 1, 0, 2, 1, 0, 2, 1, 2, 3, 2, 1, 0, 2, 1, 2]
[1, 0, 2, 1, 0, 2, 1, 2, 3, 2, 1, 0, 2, 1, 3, 2, 1, 0]    
\end{verbatim}
In the above process, from the first step, we can find out all the left cells (with different Index numbers) with the same $\bf a$ value $10$. Then from the second step, we can find out all the Weyl group elements in the left cell with Index $19$. Then we choose a shortest one as a representative with the given $\bf a$ value $10$: 
\[
x_{14}^{F_4}=[1,2,1,2,3,2,1,0,2,1,3,2,1,0].
\]
In types $E_6,E_7,E_8$, more than one two-sided cell may have the same $\mathbf a$-value. We therefore inspect all left cells with that value and distinguish the resulting two-sided cells by their Springer data. The explicit representatives used in Section \ref{sec:E-cases} are recorded in Appendix \ref{app:cell-reps}.


\end{Rem}

Now the computation of $\aff(w_{\lambda})$ is  as follows: determine the irreducible components of $\Phi_{[\lambda]}$, transport them to standard root systems, form the integral weights $\lambda'_{(i)}$, compute each $\aff(w_{\lambda'_{(i)}})$ by Proposition \ref{integral} in classical type or by Lemma \ref{find-w-lambda} and PyCox in exceptional type, and finally add the values by Corollary \ref{transf-1}.

\begin{ex}\label{ex:ca}
    Suppose $\lam=(2,1,1.1,3,0.9,1.9,4,2.1)$ is a weight for $\Phi=D_8$. Then $\Phi^+_{[\lam]}=\{\ep_i\pm \ep_{j}\mid i<j \in \{1,2,4,7\}\}\cup \{\ep_3+\ep_5,\ep_3+\ep_6,\ep_3-\ep_8,\ep_5-\ep_6,\ep_5+\ep_8,\ep_6+\ep_8\}$ and $\Phi_{[\lambda]}=\Phi_{[\lambda]_1}\sqcup \Phi_{[\lambda]_2}\simeq D_4\times A_3$.  The simple system of $\Phi_{[\lambda]_1}$ is $\Delta_{[\lambda]_1}=\{\gamma_1=\ep_1- \ep_{2},\gamma_2=\ep_2- \ep_{4},\gamma_3=\ep_4- \ep_{7},\gamma_4=\ep_4+ \ep_{7}\}$ and the simple system of $\Phi_{[\lambda]_2}$ is $\Delta_{[\lambda]_2}=\{\beta_1=\ep_3-\ep_8,\beta_2=\ep_8+\ep_6,\beta_3=\ep_5-\ep_6\}$. Suppose a standard simple system for $D_4\times A_3$ is $\{\alpha_1=\ep_1-\ep_2,\alpha_2=\ep_2-\ep_3,\alpha_3=\ep_3-\ep_4,\alpha_4=\ep_3+\ep_4\}\sqcup \{\alpha'_1=\ep_1-\ep_2,\alpha'_2=\ep_2-\ep_3,\alpha'_3=\ep_3-\ep_4\}$. We define an isomorphism
     $\phi: \Phi_{[\lam]}\rightarrow D_4\times A_3$ such that   $\phi(\gamma_i)=\alpha_i$ (for $1\leq i\leq 4$) and $\phi(\beta_i)=\alpha'_i$ (for $1\leq i\leq 3$).
    Then we have 
   \[\lambda|_{\mathfrak{h}_{\Phi_{[\lambda]}}}=\omega_1-2\omega_2-\omega_3+7\omega_4-\omega'_1+4\omega'_2-\omega'_3,\]
    where $\{\omega_i\mid 1\leq i\leq 4\}$ are the fundamental weights for $\Delta_{[\lam]_1}$ and $\{\omega'_i\mid 1\leq i\leq 3\}$ are the fundamental weights for $\Delta_{[\lam]_2}$. Denote $\lambda'=\phi(\lambda|_{\mathfrak{h}_{\Phi_{[\lambda]}}})$.
In the factor notation introduced above, the two transported weights are
  \[\lambda'_{(1)}=\lam'|_{D_4}=\phi(\omega_1)-2\phi(\omega_2)-\phi(\omega_3)+7\phi(\omega_4)=(2,1,3,4),\]
 which is an integral weight of $D_4$ with $\aff(w_{\lambda'_{(1)}})=3$ by the RS algorithm in \cite{BXX}, and
 \[\lambda'_{(2)}=\lam'|_{A_3}=-\phi(\omega'_1)+4\phi(\omega'_2)-\phi(\omega'_3)=(1,2,-2,-1),\]
 which is an integral weight of $A_3$ with $\aff(w_{\lambda'_{(2)}})=2$ by the RS algorithm in \cite{BXX}.
 Therefore we have $\aff(w_{\lambda})=3+2=5$.

\end{ex}
\begin{Rem}
 In the above Example \ref{ex:ca},  if we write \begin{align*}
    \lam'|_{A_3}+(0.1,0.1,0.1,0.1)=(1.1,2.1,-1.9,-0.9):=x,
\end{align*}
then the Young tableaux $P(\lam'|_{A_3})$ and $P(x)$ have the same shape. This is just the algorithm used in \cite{BXX} for the case of type $D$.
\end{Rem}

\section{Reduction to integral data for GK dimensions and annihilator varieties}\label{sec:pre}\label{AVHWM}\label{S:gen}

In this section, we relate the local Weyl group computation to the two invariants of interest. The irreducible factors of $\Phi_{[\lambda]}$ control the local calculation, while the embedding $\Phi_{[\lambda]}\hookrightarrow\Phi$ must be retained when determining the annihilator orbit.

\subsection{The local special datum and GK dimensions}\label{local-special-data}

Recall that the element $w_\lambda$ and the integral Weyl group $W_{[\lambda]}$ were defined in Section \ref{a-value-compu}. The following formula is the bridge from the algorithm of the preceding section to GK dimensions.

\begin{Prop}[{\cite[Prop.~3.8]{BX}}]\label{pr:main1}
	Let $ \lam\in\hs $. Suppose that $\lambda=w_{\lambda}\mu$, where $\mu$ is antidominant and $w_{\lambda}\in W_{[\mu]}^J$. Then
\begin{equation*}
	\gkd L(\lambda)=|\Phi^+|-\aff_{[\lambda]}(w_{\lambda}),
\end{equation*}
	where $\aff_{[\lambda]}$ is the $\aff$-function on $W_{[\lambda]}=W_{[\mu]}$.
\end{Prop}

Combining Proposition \ref{pr:main1} with Theorem \ref{T:main1}, for any weight $\lambda$ we have
\[
\gkd L(\lambda)=|\Phi^+|-\sum_i\aff(w_{\lambda'_{(i)}}).
\]

\begin{Lem}\label{lem:pseudolevi-existence}
For every $\lambda\in\mathfrak h^*$ there exists a pseudo-Levi subgroup $H^\vee\subseteq G^\vee$ whose root system is $\Phi_{[\lambda]}^\vee$.
\end{Lem}

\begin{proof}
Via the dual root datum, we identify $\mathfrak h^*$ with the Lie
algebra of a maximal torus $T^\vee\subseteq G^\vee$. Set
\[
t_\lambda=\exp(2\pi i\lambda)\in T^\vee
\]
and let
\[
H^\vee=C_{G^\vee}(t_\lambda)^\circ.
\]
Since a pseudo-Levi subgroup is, by definition, the identity component
of the centralizer of a semisimple element, $H^\vee$ is a pseudo-Levi
subgroup of $G^\vee$.

We now determine its root system. Since $t_\lambda\in T^\vee$, the torus
$T^\vee$ is contained in $C_{G^\vee}(t_\lambda)$ and hence is a maximal
torus of $H^\vee$. Under the dual root datum, the roots of $G^\vee$
with respect to $T^\vee$ are the coroots $\alpha^\vee$, with
$\alpha\in\Phi$. From Carter \cite[Thm. 3.5.3]{Ca85}, the root subgroup $U_{\alpha^\vee}$ is contained in
$C_{G^\vee}(t_\lambda)$ precisely when $t_\lambda$ acts trivially on
its root space. Equivalently,
\[
\alpha^\vee(t_\lambda)=1.
\]
By the definition of $t_\lambda$,
\[
\alpha^\vee(t_\lambda)
=
\exp\bigl(2\pi i\langle\lambda,\alpha^\vee\rangle\bigr).
\]
Thus
\[
U_{\alpha^\vee}\subseteq H^\vee
\quad\Longleftrightarrow\quad
\exp\bigl(2\pi i\langle\lambda,\alpha^\vee\rangle\bigr)=1
\quad\Longleftrightarrow\quad
\langle\lambda,\alpha^\vee\rangle\in\mathbb Z.
\]
It follows that
\[
\Phi(H^\vee,T^\vee)
=
\{\alpha^\vee\in\Phi^\vee:
\langle\lambda,\alpha^\vee\rangle\in\mathbb Z\}
=
\Phi_{[\lambda]}^\vee,
\]
as required.
\end{proof}

Let $H^\vee\subseteq G^\vee$ be the pseudo-Levi subgroup supplied by Lemma \ref{lem:pseudolevi-existence}, and let $H$ denote its Langlands dual group. Thus $H$ has root system $\Phi_{[\lambda]}$.  Let
\[
\pi_\lambda:=\pi_{w_\lambda}\in {\rm Irr}(W(H))
\]
be the special representation attached to the two-sided cell containing $w_\lambda$, and write
\[
\mathcal O_\lambda^H:=\mathcal O_{\rm Spr}^{H}(\pi_\lambda).
\]
We call
\[
\mathscr D(\lambda):=(\Phi_{[\lambda]}\hookrightarrow \Phi,\,\mathcal O_\lambda^H)
\]
the \emph{local special datum} of $\lambda$.

\begin{Prop}\label{prop:local-gk}
Keep the above notation. Then
\begin{equation}\label{E:local-gk}
\gkd L(\lambda)
=|\Phi^+|-|\Phi_{[\lambda]}^+|+\frac12\dim \mathcal O_\lambda^H.
\end{equation}
In particular, the GK dimension is determined by the local special datum $\mathscr D(\lambda)$.
\end{Prop}

\begin{proof}
Since $\pi_\lambda$ is the special representation attached to the two-sided cell of $w_\lambda$, Lusztig's theory of special representations \cite[Chap.~4]{lusztig1984char} gives
\[
\aff_{[\lambda]}(w_\lambda)=\aff(\pi_\lambda)=b(\pi_\lambda),
\]
where $b$ is the $b$-invariant defined in
Section \ref{subsec:j-induction}.
Under the Springer correspondence, $b(\pi_\lambda)$ is the dimension of the Springer fiber over an element of $\mathcal O_\lambda^H$ \cite[\S~13.3]{Ca85}. Hence
\[
\aff_{[\lambda]}(w_\lambda)
=|\Phi_{[\lambda]}^+|-\frac12\dim\mathcal O_\lambda^H.
\]
The result follows from Proposition \ref{pr:main1}.
\end{proof}

\subsection{From the local special datum to the annihilator orbit}

The Springer correspondence and $j$-induction needed here were recalled in Section \ref{sec:sommers}. We now state Joseph's formula in the form needed below.

Let $\lambda\in \mathfrak{h}^*$, with associated root system $\Phi_{[\lambda]}$ and Weyl subgroup $W_{[\lambda]}$. Then $\lam$ is an integral weight for $\Phi_{[\lambda]}$. Suppose that $\lambda=w_{\lambda}\mu$, where $\mu$ is antidominant and $w_{\lambda}\in W_{[\mu]}^J$. Then we use $\pi_{w_{\lam}}$ to denote the special representation of $W_{[\lambda]}$ corresponding to the two-sided cell of $W_{[\lambda]}$ containing $w_{\lam}$. 
Joseph extended the result of \cite{BoB1} to arbitrary infinitesimal character using the Springer correspondence and $j$-induction:

 %

\begin{Prop}[{\cite[Thm.~3.10]{Jo85}}]\label{2dim}
Let $\mu\in \mathfrak{h}^*$ be an antidominant element and   $w\in W_{[\mu]}^J$. 
    Then $\tilde{\pi}_w:=j_{W_{[\mu]}}^W(\pi_w)$ is an
    irreducible $W$-module.   This $W$-module $\tilde{\pi}_w$ corresponds to  a nilpotent orbit $\mathcal{O}_{\tilde{\pi}_w}$
    with trivial local system via the Springer correspondence. Furthermore,
\[V(\Ann (L(w\mu)))=\overline{\mathcal{O}_{\rm Spr}^\mathfrak{g} ( \tilde{\pi}_w)}.\]
\end{Prop}

For later use, we record Joseph's result in the notation attached to an arbitrary highest weight module $L(\lambda)$. It gives
\begin{equation} \label{E:jos}
\mathcal{O}_{{\rm Ann}(L(\lambda))} = \mathcal{O}_{\rm Spr}^G\!\left(j_{W_{[\lambda]}}^{W(G)}(\pi_{w_\lambda})\right).
\end{equation}
Equation \eqref{E:jos} determines the annihilator variety, but evaluating the $j$-induction case by case can require substantial character calculations (see, e.g., Alvis \cite{Alv}). Theorem \ref{P:Som} provides a more direct way to carry out this step through Sommers duality.

Henceforth, $\mathcal O_{\Ann(L(\lambda))}$ denotes the nilpotent orbit determined by $L(\lambda)$, so that
\[
V(\Ann(L(\lambda)))=\overline{\mathcal O}_{\Ann(L(\lambda))}.
\]
Theorem \ref{P:Som} converts Joseph's $j$-induction formula directly into Sommers duality.

The following consequence of Theorem \ref{P:Som} is the form used below.

\begin{Cor}\label{cor:direct-sommers}
Let $H^\vee\subseteq G^\vee$ be the embedded pseudo-Levi subgroup with root system $\Phi_{[\lambda]}^\vee$, and let
\[
\mathcal{O}_\lambda^{H^\vee}:=\mathcal{O}_{\rm Spr}^{H^\vee}(\pi_{w_\lambda}^\vee).
\]
By the Bala--Carter classification recalled in Section \ref{subsec:sommers-duality}, choose a Levi subgroup $L^\vee\subseteq H^\vee$ and a distinguished nilpotent orbit $\mathcal O_{L^\vee}$ such that
\[
d_{\rm LS}^{H^\vee}(\mathcal{O}_\lambda^{H^\vee})
={\rm sat}_{L^\vee}^{H^\vee}(\mathcal{O}_{L^\vee}).
\]
Then
\begin{equation}\label{E:direct-sommers}
\mathcal{O}_{{\rm Ann}(L(\lambda))}
=d_{\rm Som}^{G^\vee}(L^\vee,\mathcal{O}_{L^\vee}).
\end{equation}
\end{Cor}

\begin{proof}
Put $\pi=\pi_{w_\lambda}$. By the definition in Section \ref{subsec:sommers-duality},
\[
\pi^\vee=\tau_{H,H^\vee}(\pi)
\]
is a special representation of $W(H^\vee)$. Since the canonical
Weyl group isomorphisms for dual root data are inverse to one another,
\[
(\pi^\vee)^\vee
=\tau_{H^\vee,H}\bigl(\tau_{H,H^\vee}(\pi)\bigr)
=\pi.
\]

The subgroup $H^\vee$ is pseudo-Levi in $G^\vee$, and
$L^\vee$ is a Levi subgroup of $H^\vee$. Moreover,
$\mathcal O_{L^\vee}$ is distinguished and, by assumption,
\[
d_{\rm LS}^{H^\vee}
\bigl(\mathcal O_{\rm Spr}^{H^\vee}(\pi^\vee)\bigr)
=
{\rm sat}_{L^\vee}^{H^\vee}
\bigl(\mathcal O_{L^\vee}\bigr).
\]
Thus the hypotheses of Theorem \ref{P:Som} are satisfied with
ambient group $G^\vee$, pseudo-Levi subgroup $H^\vee$, and
special representation $\pi^\vee$. Applying that theorem and
identifying the double dual groups with $G$ and $H$, respectively,
gives
\[
d_{\rm Som}^{G^\vee}
\bigl(L^\vee,\mathcal O_{L^\vee}\bigr)
=
\mathcal O_{\rm Spr}^{G}\!
\left(
j_{W(H)}^{W(G)}(\pi)
\right).
\]
Here $W(H)$ is identified with the reflection subgroup
$W_{[\lambda]}\subseteq W(G)$ through the fixed integral root
system; the compatibility with dual transport is
\eqref{E:transp}.

Finally, Joseph's formula \eqref{E:jos} yields
\[
\begin{aligned}
\mathcal O_{{\rm Ann}(L(\lambda))}
&=
\mathcal O_{\rm Spr}^{G}\!
\left(
j_{W_{[\lambda]}}^{W(G)}(\pi_{w_\lambda})
\right)\\
&=
d_{\rm Som}^{G^\vee}
\bigl(L^\vee,\mathcal O_{L^\vee}\bigr),
\end{aligned}
\]
as required.
\end{proof}

The two computations therefore have a common starting point. The local special orbit $\mathcal O_\lambda^H$ gives the GK dimension through \eqref{E:local-gk}. For the annihilator orbit, first transport $\pi_{w_\lambda}$ to $W(H^\vee)$ and form $\mathcal O_\lambda^{H^\vee}$. Its Lusztig--Spaltenstein dual gives an embedded generalized Bala--Carter datum, whose Sommers dual is the annihilator orbit.

\begin{Cor}\label{cor:integral-case}
Suppose that $\lambda$ is integral for $G$. Then $W_{[\lambda]}=W(G)$, and the two-sided cell of $w_\lambda$ determines both invariants:
\[
\gkd L(\lambda)=|\Phi^+|-\aff(\pi_{w_\lambda}),
\qquad
\mathcal O_{{\rm Ann}(L(\lambda))}=\mathcal O_{\rm Spr}^{G}(\pi_{w_\lambda}).
\]
\end{Cor}

\subsection{General reduction to integral irreducible data}\label{sec:general-reduction}

The preceding results give the reduction theorem used in the case computations.

\begin{Thm}\label{thm:general-reduction}
Let $\lambda\in\mathfrak h^*$ and write
\[
\Phi_{[\lambda]}=\Phi_{[\lambda]_1}\sqcup\cdots\sqcup\Phi_{[\lambda]_k}.
\]
Choose the root system isomorphisms $\phi_i:\Phi_{[\lambda]_i}\xrightarrow{\sim}\Phi_i$ to standard irreducible root systems, with $\phi_i(\Phi_{[\lambda]_i}^+)=\Phi_i^+$, and let $\lambda'_{(i)}$ be the transported restrictions of $\lambda$.

\begin{enumerate}
\item Under the induced Coxeter system isomorphism,
\[
w_\lambda\longleftrightarrow
(w_{\lambda'_{(1)}},\ldots,w_{\lambda'_{(k)}}),
\qquad
\aff(w_\lambda)=\sum_i\aff(w_{\lambda'_{(i)}}),
\]
and hence
\begin{equation}\label{E:general-gk}
\gkd L(\lambda)=|\Phi^+|-\sum_i\aff(w_{\lambda'_{(i)}}).
\end{equation}
Thus the GK dimension problem is reduced to integral computations for irreducible root systems.

\item Let $H^\vee=C_{G^\vee}(t_\lambda)^\circ$, so that $\Phi(H^\vee)=\Phi_{[\lambda]}^\vee$, and let
\[
\mathcal O_\lambda^{H^\vee}
=\mathcal O_{\rm Spr}^{H^\vee}(\pi_{w_\lambda}^\vee).
\]
By the Bala--Carter theorem in $H^\vee$, choose a Levi subgroup $L^\vee\subseteq H^\vee$ and a distinguished nilpotent orbit $\mathcal O_{L^\vee}$ such that
\[
d_{\rm LS}^{H^\vee}(\mathcal O_\lambda^{H^\vee})
={\rm sat}_{L^\vee}^{H^\vee}(\mathcal O_{L^\vee}).
\]
Then
\begin{equation}\label{E:general-algorithm}
d_{\rm LS}^{H^\vee}(\mathcal O_\lambda^{H^\vee})
={\rm sat}_{L^\vee}^{H^\vee}(\mathcal O_{L^\vee}),
\qquad
(L^\vee,\mathcal O_{L^\vee})
\xrightarrow{\ d_{\rm Som}^{G^\vee}\ }
\mathcal O_{{\rm Ann}(L(\lambda))}.
\end{equation}
In particular,
\[
\mathcal O_{{\rm Ann}(L(\lambda))}
=d_{\rm Som}^{G^\vee}(L^\vee,\mathcal O_{L^\vee}).
\]
\end{enumerate}
Consequently, both invariants are determined by the embedded integral root system together with the local  two-sided cell. The only local  case-by-case input needed later is the integral data for the irreducible factors, and the related embedding information needed for the annihilator orbit is handled separately by the Sommers procedure of Section \ref{sec:embedding-examples}.
\end{Thm}

\begin{proof}
Part (1) is Theorem \ref{T:main1} together with Proposition \ref{pr:main1}. Part (2) follows from Lemma \ref{lem:pseudolevi-existence} and Corollary \ref{cor:direct-sommers}. The final assertion follows from Proposition \ref{prop:local-gk} and part (2).
\end{proof}

\noindent\textbf{Computational procedure.}
Theorem \ref{thm:general-reduction} gives the following procedure for a weight $\lambda\in\mathfrak h^*$.
\begin{enumerate}
\item Determine $\Phi_{[\lambda]}$, its positive simple system, and its irreducible components. Choose the isomorphisms $\phi_i$ of Section \ref{nonint} and compute the transported weights $\lambda'_{(i)}$. Retain the embedded roots and the inverse isomorphisms, including their induced maps on coroots.

\item For each factor, determine the special representation $\pi_i$ attached to the two-sided cell of $w_i=w_{\lambda'_{(i)}}$, and record $a_i=\aff(\pi_i)$. Use the Robinson--Schensted and H-algorithms for classical factors, and positive index reduction followed by PyCox for exceptional factors. If $\pi_i$ is recorded by a type-$D$ partition that is very even, also retain its local branch as determined by the representation.

\item Compute
\[
\gkd L(\lambda)=|\Phi^+|-\sum_i a_i.
\]
Under the fixed Weyl group identification, form the external tensor product $\pi_{w_\lambda}=\boxtimes_i\pi_i$. Let $H^\vee=C_{G^\vee}(t_\lambda)^\circ$ be as in Lemma \ref{lem:pseudolevi-existence}. Transport this representation to $W(H^\vee)$ and form
\[
\pi_{w_\lambda}^\vee=\tau_{H,H^\vee}(\pi_{w_\lambda}),
\qquad
\mathcal O_\lambda^{H^\vee}
=\mathcal O_{\rm Spr}^{H^\vee}(\pi_{w_\lambda}^\vee).
\]

\item Take the Lusztig--Spaltenstein dual in $H^\vee$ and choose its Bala--Carter datum:
\[
d_{\rm LS}^{H^\vee}(\mathcal O_\lambda^{H^\vee})
={\rm sat}_{L^\vee}^{H^\vee}(\mathcal O_{L^\vee}),
\]
where $L^\vee$ is a Levi subgroup of $H^\vee$ and $\mathcal O_{L^\vee}$ is distinguished. Pull the Levi roots back from standard dual coordinates via the inverses of the induced dual root system isomorphisms, retaining the embedding of $(L^\vee,\mathcal O_{L^\vee})$ in $G^\vee$. In the cases with primed or double primed labels, we will need weighted Dynkin diagrams as in
Section \ref{sec:local-branches} to distinguish the labels. Finally, we have
\[
\mathcal O_{{\rm Ann}(L(\lambda))}
=d_{\rm Som}^{G^\vee}(L^\vee,\mathcal O_{L^\vee}),
\qquad
V(\Ann(L(\lambda)))
=\overline{\mathcal O_{{\rm Ann}(L(\lambda))}}.
\]
\end{enumerate}


\begin{Cor}\label{cor:embedded-data}
Suppose that the local special data $\mathscr D(\lambda)$ and
$\mathscr D(\mu)$ are $W$-conjugate. Then
\[
\gkd L(\lambda)=\gkd L(\mu),\qquad
\mathcal O_{{\rm Ann}(L(\lambda))}
=
\mathcal O_{{\rm Ann}(L(\mu))}.
\]
\end{Cor}

\begin{proof}
The equality of GK dimensions follows from Proposition
\ref{prop:local-gk}. The equality of annihilator orbits follows from
Theorem \ref{thm:general-reduction}(2), since Lusztig--Spaltenstein
duality and Sommers duality depend only on the $W$-conjugacy class of
the embedded local special datum.
\end{proof}



\begin{ex}\label{ex:nonint-f4-gk}
Let $\mathfrak g=F_4$ and
\[
\lambda=\left(\frac74,\frac14,\frac54,-\frac34\right).
\]
In the standard realization of $F_4$ used in Section \ref{f4-gkd}, set
\[
\delta:=\tfrac12(\ep_1-\ep_2+\ep_3+\ep_4).
\]
Then the integral root system has simple system
\[
\Delta_{[\lambda]}
=
\underbrace{\{\alpha_2,\alpha_1,\delta\}}_{B_3\text{-factor}}
\sqcup
\underbrace{\{\alpha_4\}}_{\widetilde A_1\text{-factor}},
\]
where
\[
\alpha_1=\ep_2-\ep_3,\qquad
\alpha_2=\ep_3-\ep_4,\qquad
\alpha_4=\tfrac12(\ep_1-\ep_2-\ep_3-\ep_4).
\]
Thus \(\Phi_{[\lambda]}\cong B_3\times \widetilde{A}_1\), where the tilde in \(B_3\times\widetilde{A}_1\) is attached to the connected component corresponding to the short roots. In a separate copy of $\mathbb R^3$, take the standard $B_3$ simple system
\[
\beta_1=\ep_1-\ep_2,\qquad
\beta_2=\ep_2-\ep_3,\qquad
\beta_3=\ep_3,
\]
and let $\gamma_1$ be the simple root of an independent standard $A_1$ factor. The correspondence
\[
\alpha_2\mapsto\beta_1,\qquad
\alpha_1\mapsto\beta_2,\qquad
\delta\mapsto\beta_3,\qquad
\alpha_4\mapsto\gamma_1
\]
preserves the Cartan matrix. Hence it defines  an isomorphism preserving the chosen positive
systems
\[
\phi:\Phi_{[\lambda]}\xrightarrow{\sim}B_3\times  A_1.
\]
 The transported weights are
\[
\lambda'_{(1)}=(2,0,1),
\qquad
\lambda'_{(2)}=\left(\frac12,-\frac12\right).
\]
Then from Proposition \ref{integral}, we have 
\[
\aff(w_{\lambda'_{(1)}})=4,
\qquad
\aff(w_{\lambda'_{(2)}})=1.
\]
Hence Theorem \ref{thm:general-reduction}(1) yields
\[
\gkd L(\lambda)=24-(4+1)=19.
\]
\end{ex}


\begin{ex}\label{ex:nonint-e8-gk}
Let $\mathfrak g=E_8$ and
\[
\lambda=\left(\frac12,-\frac32,-3,-2,-1,-4,-5,-19\right).
\]
The integral root system has simple system
\[
\Delta_{[\lambda]}
=
\underbrace{\{\alpha_8,\alpha_7,\alpha_6,\alpha_5,
\ep_3+\ep_4,\alpha_1,\alpha_3\}}_{E_7\text{-factor}}
\sqcup
\underbrace{\{\alpha_2\}}_{A_1\text{-factor}}.
\]
Thus \(\Phi_{[\lambda]}\simeq E_7\times A_1\). In a separate copy of $\mathbb R^8$, take the standard $E_7$ simple system
\[
\Delta_{E_7}=\{\beta_1,\ldots,\beta_7\},
\]
where
\[
\begin{aligned}
\beta_1&=\tfrac12(\ep_1-\ep_2-\ep_3-\ep_4-\ep_5-\ep_6-\ep_7+\ep_8),\\
\beta_2&=\ep_1+\ep_2,\qquad
\beta_3=\ep_2-\ep_1,\\
\beta_4&=\ep_3-\ep_2,\qquad
\beta_5=\ep_4-\ep_3,\\
\beta_6&=\ep_5-\ep_4,\qquad
\beta_7=\ep_6-\ep_5.
\end{aligned}
\]
Let \(\gamma_1\) be the simple root of an independent standard $A_1$ factor. The isomorphism preserving the chosen positive
systems
\[
\phi:\Phi_{[\lambda]}\xrightarrow{\sim}E_7\times A_1
\]
is determined by
\[
\begin{aligned}
&\phi(\alpha_8)=\beta_1,\quad
\phi(\alpha_7)=\beta_3,\quad
\phi(\alpha_6)=\beta_4,\quad
\phi(\alpha_5)=\beta_2,\\
&\phi(\ep_3+\ep_4)=\beta_5,\quad
\phi(\alpha_1)=\beta_6,\quad
\phi(\alpha_3)=\beta_7,\quad
\phi(\alpha_2)=\gamma_1.
\end{aligned}
\]
The transported weights are
\[
\lambda'_{(1)}=(2,-1,0,-5,-6,-8,12,-12),
\qquad
\lambda'_{(2)}=\left(-\frac12,\frac12\right).
\]
The $E_7$ factor is precisely the weight considered in Example \ref{ex:e7-positive-index}. In fundamental weight coordinates,
\[
\lambda'_{(1)}=[-1,1,-3,1,-5,-1,-2].
\]
Writing $s_i=s_{\beta_i}$ on the standard $E_7$ factor, positive index reduction gives
\[
s_4s_2s_4\lambda'_{(1)}=[-1,-1,-1,-1,-3,-1,-2],
\]
which is antidominant. Hence
\[
w_{\lambda'_{(1)}}=s_4s_2s_4=[3,1,3]_{\rm PyCox},
\qquad
\aff(w_{\lambda'_{(1)}})=3.
\]
The $A_1$ factor is already antidominant, so
\[
w_{\lambda'_{(2)}}=e,
\qquad
\aff(w_{\lambda'_{(2)}})=0.
\]
Therefore Theorem \ref{thm:general-reduction}(1) gives
\[
\gkd L(\lambda)=120-(3+0)=117.
\]
\end{ex}

In applying part (2) of Theorem \ref{thm:general-reduction}, we use the dual of the chosen root system isomorphism
\[
\phi:\Phi_{[\lambda]}\longrightarrow \Phi_1\times\cdots\times\Phi_k.
\]
It induces
\[
\phi^\vee:\Phi_{[\lambda]}^\vee\longrightarrow
\Phi_1^\vee\times\cdots\times\Phi_k^\vee,
\qquad
\phi^\vee(\alpha^\vee)=\phi(\alpha)^\vee.
\]
If the standard calculation produces a Levi root subsystem with simple system $J_{\rm std}^\vee$, its embedded simple roots are
\[
(\phi^\vee)^{-1}(J_{\rm std}^\vee)
\subseteq\Phi_{[\lambda]}^\vee\subseteq\Phi^\vee.
\]
Thus the output is the actual embedded pseudo-Levi subgroup $L^\vee\subseteq G^\vee$, not merely its abstract Dynkin type. This issue is relevant for decorated Bala--Carter labels in types $E_7$ and $E_8$ and is discussed in Section \ref{sec:embedding-examples}.

\section{Sommers data and embedded pseudo-Levi subgroups in the general reduction}\label{sec:embedding-examples}
In this section, we explain the practical use of Theorem \ref{thm:general-reduction}. First we  obtain the local dual orbits and their Bala--Carter data. Then we distinguish the embedded pseudo-Levi classes that have the same abstract Dynkin type in $E_7$ and $E_8$.

\subsection{Computing the Sommers datum}\label{sec:practical-sommers}


For each irreducible factor we record the local special representation and the data needed for its Lusztig--Spaltenstein dual. The classical factors are described by partitions and the H-algorithm in \cite{BMW}, whereas exceptional factors are read from the special character returned by PyCox.

For $x\in\mathrm{Seq}_n(\Gamma)$, the Robinson--Schensted algorithm in Section \ref{RS-algorithm} produces a Young tableau $P(x)$ with shape
$
{\bf p}(x)=p(x)=[p_1,\ldots,p_r].
$ The H-algorithm of \cite{BMW} associates to $p(x^-)$ a special partition $p_X(x^-)^s$ in type $X=B,C,D$. It is characterized by having the same odd boxes as $p(x^-)$ in types $B$ and $C$, and the same even boxes in type $D$.

With the notation of Section \ref{nonint}, write
\[
\Phi_{[\lambda]}\simeq\Phi_1\times\cdots\times\Phi_k,
\]
and use $\lambda'_{(i)}$ for the transported integral weight on the factor $\Phi_i$. For a factor of type $A$, denote
${\bf p}_i={\bf p}_i^\vee=p(\lambda'_{(i)}).$
For a factor of type $X_i=B,C,D$, denote
\[
{\bf p}_i=p_{X_i}((\lambda'_{(i)})^-)^s,
\quad
{\bf p}_i^\vee=p_{X_i^\vee}((\lambda'_{(i)})^-)^s.
\]
For an exceptional factor, Lemma \ref{find-w-lambda} and PyCox determine the corresponding special character $\chi_i$ and special nilpotent orbit $\mathcal O_{\chi_i}$.

Let $\pi_i$ be the special representation attached to $w_{\lambda'_{(i)}}$, and let $\pi_i^\vee=\tau_{H_i,H_i^\vee}(\pi_i)$ be its transport to $W(H_i^\vee)$, where $H_i$ denotes the corresponding simple factor. Put
\[
\mathcal O_i=\mathcal O_{\rm Spr}^{H_i}(\pi_i),\qquad
\mathcal O_i^\vee=\mathcal O_{\rm Spr}^{H_i^\vee}(\pi_i^\vee),\qquad
\mathcal Q_i=d_{\rm LS}^{H_i^\vee}(\mathcal O_i^\vee).
\]
For a classical factor, the partitions ${\bf p}_i$ and ${\bf p}_i^\vee$ specify the underlying partitions of $\mathcal O_i$ and $\mathcal O_i^\vee$, respectively. When a special partition determines a unique orbit, we also use it to denote the associated special representation via the Springer correspondence; see \cite[\S~13.3]{Ca85} and \cite{CM}. This shorthand does not identify the two orbits attached to a very even partition in type $D$.

We write ${\bf p}^t$ for the transpose of ${\bf p}$. For a classical factor of type $X_i$, the underlying partition of $\mathcal Q_i$ is
\[
{\bf q}_i=
\begin{cases}
(\mathbf p_i^\vee)^t,& X_i=A,\\
\bigl((\mathbf p_i^\vee)^t\bigr)_{X_i^\vee},& X_i=B,C,D,
\end{cases}
\]
with the usual collapse convention in types $B,C,D$ \cite{Sp82,CM}. For an exceptional factor, $\mathcal Q_i$ is determined from the local special character using the Springer correspondence and the Lusztig--Spaltenstein duality tables.
 When the group is clear, we abbreviate $\mathcal O_{\rm Spr}^K(\sigma)$ to $\mathcal O(\sigma)$.

When there is no  confusion, we use the same Dynkin type notation for a root subsystem and for the corresponding connected Levi or pseudo-Levi subgroup. The intended meaning will be clear from the context.

\begin{Cor}\label{cor:factorwise-sommers}
With the notation above, let $L(\lambda)$ be a highest weight module and retain the group $H$ from Section \ref{local-special-data}, so that
\[
\Phi(H)=\Phi_{[\lambda]}\simeq\Phi_1\times\cdots\times\Phi_k.
\]
For the cases used below, the local data are described as follows.
\begin{enumerate}
\item Suppose every $\Phi_i$ is of classical type $X_i$. Then
\[
{\pi_{w_\lambda}=\boxtimes_i\pi_i\in{\rm Irr}(W(H)),
\quad
\pi_{w_\lambda}^\vee=\boxtimes_i\pi_i^\vee\in{\rm Irr}(W(H^\vee)),}
\]
and
\[
d_{\rm LS}^{H^\vee}\bigl(\mathcal O(\pi_{w_\lambda}^\vee)\bigr)
{=\prod_i\mathcal Q_i}
={\rm sat}_{L^\vee}^{H^\vee}(\mathcal O_{L^\vee})
\]
for a Levi subgroup $L^\vee\subseteq H^\vee$ and a distinguished orbit $\mathcal O_{L^\vee}$. With the induced embedding of $(L^\vee,\mathcal O_{L^\vee})$ in $G^\vee$, we have
\[
\mathcal O_{{\rm Ann}(L(\lambda))}=d_{\rm Som}^{G^\vee}(L^\vee,\mathcal O_{L^\vee}).
\]


\item Suppose exactly one factor $\Phi_{i_0}$ is of exceptional type {$E_6$, $E_7$, or $E_8$}, while all remaining factors are simply laced and classical. Let $\chi_{i_0}$ be the special character attached to $w_{\lambda'_{(i_0)}}$, and suppose that its orbit has Bala--Carter label $L_{i_0}$. Then
\[
\pi_{w_\lambda}^\vee\simeq\pi_{w_\lambda}
{=\chi_{i_0}\boxtimes\mathop{\boxtimes}_{i\ne i_0}\pi_i,
\quad
\mathcal O(\pi_{w_\lambda}^\vee)
=\mathcal O_{\chi_{i_0}}\times\prod_{i\ne i_0}\mathcal O_i.}
\]
If
\[
d_{\rm LS}^{H^\vee}\bigl(\mathcal O(\pi_{w_\lambda}^\vee)\bigr)
={\rm sat}_{L^\vee}^{H^\vee}(\mathcal O_{L^\vee}),
\]
then
\[
\mathcal O_{{\rm Ann}(L(\lambda))}=d_{\rm Som}^{G^\vee}(L^\vee,\mathcal O_{L^\vee}).
\]
\end{enumerate}
\end{Cor}

\begin{proof}
Under the fixed Weyl group identifications, the local special representation is the external tensor product of the $\pi_i$. The Springer correspondence and Lusztig--Spaltenstein duality act factorwise. The assertions therefore follow from Corollary \ref{cor:direct-sommers} applied to the resulting embedded pair $(L^\vee,\mathcal O_{L^\vee})$.
\end{proof}

For explicit calculations, we may need the local special representations in addition to their underlying partitions. If $\mathbf p_i^\vee$ is very even, its branch is the one specified by $\pi_i^\vee$. In Section \ref{sec:local-branches}, we will give the corresponding selection of embedded Levi roots. 


\subsection{Distinguishing embedded pseudo-Levi subgroups in types \texorpdfstring{$E_7$ and $E_8$}{E7 and E8}}\label{dis}

In the following Table \ref{label}, we list the correspondence between partitions and Bala--Carter labels used below.

\begin{table}[htbp]
\caption{Partition and Bala--Carter correspondences}
\label{label}
\centering
\renewcommand{\arraystretch}{1.15}
\setlength{\tabcolsep}{6pt}
\begin{tabular}{@{}ccc@{}}
\toprule
\text{Type} & ${\bf p}$ & \text{Bala--Carter label} \\
\midrule
$A_k$ & $[k+1]$ & $A_k$ \\
\midrule
$B_k$ & $[2k+1]$ & $B_k$ \\
$B_4$ & $[5,3,1]$ & $B_4(a_2)$ \\
$B_{k+1}$ & $[k+1,k+1,1]$ & $A_k$ \\
\midrule
$C_k$ & $[2k]$ & $C_k$ \\
$C_3$ & $[4,2]$ & $C_3(a_1)$ \\
$C_{k+1}$ & $[k+1,k+1]$ & $\widetilde A_k$ \\
$C_3$ & $[2^3]$ & $A_1+\widetilde A_1$ \\
\midrule
$D_8$ & $[7,5,3,1]$ & $D_8(a_5)$ \\
$D_6$ & $[9,3]$ & $D_6(a_1)$ \\
$D_6$ & $[7,5]$ & $D_6(a_2)$ \\
$D_5$ & $[7,3]$ & $D_5(a_1)$ \\
$D_4$ & $[5,3]$ & $D_4(a_1)$ \\
$D_3$ & $[5,1]$ & $D_3\simeq A_3$ \\
$D_6$ & $[5,3,3,1]$ & $A_3+A_2$ \\
$D_2$ & $[3,1]$ & $D_2=2A_1$ \\
$D_k$ & $[2k-2i-1,2i+1]$ & $D_k(a_i)$ \\
$D_{k+1}$ & $[k+1,k+1]$ & $A_k$ \\
$D_k$ & $[2k-1,1]$ & $D_k$ \\
\bottomrule
\end{tabular}
\par\smallskip
\parbox{0.94\linewidth}{For the row $D_k(a_i)$, assume $k\geq4$ and $1\leq i\leq\lfloor(k-2)/2\rfloor$. In the row $D_{k+1}$ with partition $[k+1,k+1]$, a very even partition specifies only the underlying partition; the actual $A_k$ Levi embedding must also be retained.}
\end{table}

Recall that we use $\Delta=\{\alpha_1,\dots,\alpha_n\}$ to denote the simple system of the root system $\Phi$ for a simple Lie algebra $\mathfrak{g}$. Let $\theta=\sum_{1\leq i\leq n}c_i\alpha_i$ be the highest root, put $\alpha_0=-\theta$, and set $\widetilde{\Delta}=\Delta\cup\{\alpha_0\}$ and $c_0=1$. Following \cite[\S~2.2]{som-98}, for any proper subset $J\subseteq \widetilde{\Delta}$ let $d_J$ be the greatest common divisor of those $c_i$ for which $\alpha_i\in \widetilde{\Delta}-J$.

\begin{Rem}\label{prime1-2}
 Following \cite{som-98}, in types $E_7$ and $E_8$ the primes and double primes attached to generalized Bala--Carter data distinguish inequivalent $W$-conjugacy classes of the embedded pseudo-Levi subsystem together with its distinguished orbit. For the cases considered here, the primed class is used when the corresponding subset of $\widetilde\Delta$ is conjugate to a subset of $A_k\subseteq E_k$ ($k=7,8$); otherwise the double primed class is used. In type $E_8$, this is equivalent to $d_J=1$ for the primed class. These decorations refer to the embedding class and are not automatically inherited by the saturated orbit in the corresponding group. After saturation, we use the standard  Bala--Carter label of the resulting nilpotent orbit.
\end{Rem}

\begin{Rem}
    In Table \ref{primes} we elaborate on the different conventions and notation used in the literature regarding very even orbits and their associated Weyl group representations.

\begin{center}
\begin{table}[htbp]

\caption{Conventions for very even nilpotent orbits in
$\mathfrak{so}(2n,\mathbb C)$ and the corresponding
representations of $W(D_n)$}
\label{primes}
\centering
\small
\renewcommand{\arraystretch}{1.15}
\setlength{\tabcolsep}{6pt}
\begin{tabular}{@{}ccc@{}}
\toprule
{Our notation} & ${\bf p}^I$ & ${\bf p}^{II}$ \\
\midrule
\cite[Eq.~(3.4), Prop.~3.7]{BMW} & {$\pi^I$} & {$\pi^{II}$} \\
\addlinespace[2pt]
\cite[p.~89]{lusztig1984char} & $\pi'$ & $\pi''$ \\
\addlinespace[2pt]
\cite[p.~170]{GP} & $([\bf u],+)$ & $([\bf u],-)$ \\
\addlinespace[2pt]
PyCox & $([\bf u],-)$ & $([\bf u],+)$ \\
\addlinespace[2pt]
\cite[p.~83]{CM} & $\mathcal{O}^I$ if $n\equiv 0\pmod 4$ & $\mathcal{O}^{II}$ if $n\equiv 0\pmod 4$\\
& $\mathcal{O}^{II}$ if $n\equiv 2\pmod 4$ & $\mathcal{O}^{I}$ if $n\equiv 2\pmod 4$\\
\addlinespace[2pt]
\cite{Alv} & $([\bf u],[\bf u])^+$ if $n\equiv 0\pmod 4$ & $([\bf u],[\bf u])^-$ if $n\equiv 0\pmod 4$\\
& $([\bf u],[\bf u])^-$ if $n\equiv 2\pmod 4$ & $([\bf u],[\bf u])^+$ if $n\equiv 2\pmod 4$\\
\bottomrule
\end{tabular}
\end{table}
\end{center}

\end{Rem}

Let $J\subseteq\widetilde{\Delta}$ be a proper subset whose generated root subsystem $\Phi_J$ represents, up to  Weyl group conjugacy, the embedded pseudo-Levi root system under consideration. We keep the invariant $d_J$ defined above. As explained in \cite{som-98}, the following cases require some care because the abstract Dynkin type alone does not determine the embedded conjugacy class.
\begin{enumerate}
    \item For type $E_7$, $\Phi_J$ is one of the following: \begin{equation}\label{E7-case}
        A_5,A_5+A_1,4A_1,3A_1,A_3+2A_1 \text{~or~} A_3+A_1.
    \end{equation}
When $J$ is conjugate to a subset of $A_7\subseteq E_7$ in the above cases, we use the primed label; otherwise we use the double primed label.
\item For type $E_8$, $\Phi_J$ is one of the following: \begin{equation}\label{E8-case}
    A_7,A_5+A_1,4A_1,2A_3 \text{~or~} A_3+2A_1.
\end{equation}
When $J$ is conjugate to a subset of $A_8\subseteq E_8$, one has $d_J=1$ and we use the primed label; otherwise $d_J=2$ and we use the double primed label.
\end{enumerate}

\begin{ex}

The following three $E_8$ cases illustrate the conjugacy criterion for choosing between the primed and double primed labels.
  
Take
\[
J_1=\{\alpha_1,\alpha_2,\alpha_5,\alpha_6,\alpha_7\}.
\]
For $w_1=s_{\ep_2-\ep_1}s_{\ep_8-\ep_7}s_{\ep_7+\ep_2}s_{\ep_8+\ep_1}\in W(E_8)$, one has
\[
w_1(J_1)=\{\alpha_1,\alpha_0,\alpha_5,\alpha_6,\alpha_7\}\subseteq A_8.
\]
Hence Remark \ref{prime1-2} shows that the embedded pseudo-Levi class is $L=(A_3+2A_1)'$; here $\mathcal O_L$ is the regular nilpotent orbit of $L$.

For
\[
J_2=\{\alpha_1,\alpha_2,\alpha_5,\alpha_7\},
\]
the same element satisfies $w_1(J_2)=\{\alpha_1,\alpha_0,\alpha_5,\alpha_7\}\subseteq A_8$, so the embedded pseudo-Levi class is $L=(4A_1)'$, with $\mathcal O_L$ its regular nilpotent orbit.

Finally, let
\[
J_3=\{\alpha_2,\alpha_3,\alpha_5,\alpha_6,\alpha_7\}.
\]
The invariant of \cite[\S~2.2]{som-98} is
\[
d_{J_3}=\gcd(c_0,c_1,c_4,c_8)=\gcd(1,2,6,2)=1.
\]
Thus $J_3$ is conjugate to a subset of $A_8$, and Remark \ref{prime1-2} again gives the embedded pseudo-Levi class $L=(A_3+2A_1)'$, with $\mathcal O_L$ its regular nilpotent orbit.

\end{ex}

\subsection{Local branches and weighted Dynkin diagrams}\label{sec:local-branches}

With the local orbit notation of Section \ref{sec:practical-sommers}, we now select the local branch and its embedded Levi roots. These roots are the input to the weighted Dynkin diagram calculation that follows.

Suppose a type-$D_n$ factor gives a very even partition $\mathbf p_i^\vee$. We first determine its $I/II$ label from the transported weight, using the convention of \cite[Eq.~(3.4), Prop.~3.7]{BMW}, also recorded in Table \ref{primes}. Write $w_i$ in the standard signed permutation action on $\ep_1,\ldots,\ep_n$. To use the signed permutation notation of \cite[Eq.~(7.1)]{BMW}, put
\[
\widehat w_i(a)=\operatorname{sgn}(w_i(n+1-a))
\bigl(n+1-|w_i(n+1-a)|\bigr),\qquad 1\leq a\leq n.
\]
Thus $s_{\ep_j-\ep_{j+1}}$ corresponds to the adjacent transposition $s_{n-j}$, and $s_{\ep_{n-1}+\ep_n}$ to the signed transposition of the first two entries. From Garfinkle \cite{Garfinkle1}, we  apply the domino insertion  and pass to the corresponding special domino tableau $T_i$. Let $v(T_i)$ be its number of vertical dominoes. The local label is $\epsilon_i=I$ if $v(T_i)\equiv0\pmod4$, and $\epsilon_i=II$ if $v(T_i)\equiv2\pmod4$, by \cite[Lem.~7.8, Rem.~7.9, Thm.~7.10]{BMW}. Equivalently, $\widehat w_i$ lies in the two-sided cell of the canonical element $w_I$ or $w_{II}$ in \cite[Lem.~7.8]{BMW}. 

Write its transposed partition as
\[
\mathbf q_i=(\mathbf p_i^\vee)^t=[d_1,d_1,\ldots,d_t,d_t],\qquad
d_j\text{ even},\quad\sum_jd_j=n.
\]
Fix $\beta_j=\ep_j-\ep_{j+1}$ for $j<n$ and $\beta_n=\ep_{n-1}+\ep_n$. Define
\[
\begin{aligned}
J_I&=\{\beta_1,\ldots,\beta_{n-1}\}
\setminus\{\beta_{d_1+\cdots+d_j}:1\leq j<t\},\\
J_{II}&=(J_I\setminus\{\beta_{n-1}\})\cup\{\beta_n\}.
\end{aligned}
\]
Both have type $A_{d_1-1}+\cdots+A_{d_t-1}$. The longest element of $W_{J_I}$ has local label $I$, and that of $W_{J_{II}}$ has local label $II$. This follows from \cite[Lem.~7.8]{BMW}; the two inductions are taken through the two $A_{n-1}$ subgroups in \cite[Eq.~(3.4)]{BMW}. Choose $J_{\epsilon_i}$. Then
\[
\pi_i^\vee=j_{W_{J_{\epsilon_i}}}^{W(D_n)}(\varepsilon_{W_{J_{\epsilon_i}}}),\qquad
\mathcal Q_i={\rm sat}_{L_{J_{\epsilon_i}}}^{H_i^\vee}
(\mathcal O_{L_{J_{\epsilon_i}}}^{\rm reg}),
\]
by Springer induction and the duality identity used in Theorem \ref{P:Som}. Pull back $J_{\epsilon_i}$ under the fixed root system isomorphism to obtain the actual embedded Levi roots. This fixes the local orbit and the root subsystem used in the LS dual. 

When the conjugacy into $A_7\subseteq E_7$ or $A_8\subseteq E_8$ is not immediate, we can use weighted Dynkin diagrams to distinguish the decorated embedding classes in \eqref{E7-case} and \eqref{E8-case}; see \cite[\S 3.5--\S 3.6]{CM}. The diagram is computed for the  saturation
\[
{\rm sat}_{H^\vee}^{G^\vee}
\left(d_{\rm LS}^{H^\vee}(\mathcal O(\pi_{w_\lambda}^\vee))\right),
\quad G=E_7\text{ or }E_8.
\]

For the decorated data considered here, the distinguished orbit in $L^\vee$ is regular and each simple factor of $L^\vee$ has type $A$. The calculation is as follows.
\begin{enumerate}
\item Determine the simple system $\Delta_{[\lambda]}^\vee$ of the embedded dual integral root system $\Phi(H^\vee)=\Phi_{[\lambda]}^\vee$ by Lemma \ref{integral-simple-system}.
\item Use the actual embedded Levi root subsystem determined above, so that
\[
d_{\rm LS}^{H^\vee}(\mathcal O(\pi_{w_\lambda}^\vee))
={\rm sat}_{L^\vee}^{H^\vee}(\mathcal O_{L^\vee}).
\]
Retain the selected simple roots of its type-$A$ factors in the fixed embedding. For a very even factor, these are the pullbacks of $J_{\epsilon_i}$ selected in the preceding local calculation.
\item For a simple system $\{\beta_1,\ldots,\beta_r\}$ of type $A_r$, form
\[
h_{A_r}=\sum_{j=1}^r j(r+1-j)\beta_j^\vee.
\]
This is the neutral element of a principal $\mathfrak{sl}_2$-triple. Add these elements over the simple factors to obtain $h\in\operatorname{span}_{\mathbb R}\Phi(L^\vee)^\vee$.

\item After identifying the corresponding Cartan subalgebra with its dual via the invariant bilinear form, regard $h$ as a weight and apply Lemma \ref{find-w-lambda}; changing sign at the end gives the dominant element $h^*$ in the Weyl orbit of $h$ since $w_0=-1$ in types $E_7$ and $E_8$.
\item The tuple $(\alpha_i(h^*))_i$ is the weighted Dynkin diagram of the saturated orbit. For the fixed abstract type and distinguished orbit considered here, comparison with \cite{som-98} or Tables \ref{tabchare7} and \ref{tabchare8} distinguishes the candidate primed and double primed embeddings. Retain the full pair $(L^\vee,\mathcal O_{L^\vee})$ in the Sommers step.
\end{enumerate}

\begin{ex}\label{ex:e8-very-even-pair}

Let $\mathfrak g=E_8$, with the Bourbaki coordinates recalled in Section \ref{sec:E-cases}, and set
\[
\begin{aligned}
\lambda_-&=(-9,-7,-6,-5,-4,-3,-2,-1),\\
\lambda_+&=(9,-7,-6,-5,-4,-3,-2,-1).
\end{aligned}
\]
The two integral root systems coincide and are given by
\[
\Phi_{[\lambda_-]}=\Phi_{[\lambda_+]}
=\{\pm e_i\pm e_j:1\le i<j\le 8\}\simeq D_8.
\]
Set
\[
\delta_i=e_{9-i}-e_{8-i}\quad(1\leq i\leq7),\qquad
\delta_8=e_1+e_2.
\]
In a separate copy of $\mathbb R^8$, take
$\beta_i=\ep_i-\ep_{i+1}$ for $i\leq7$ and
$\beta_8=\ep_7+\ep_8$. The isomorphism
$\phi(e_j)=\ep_{9-j}$ preserves positive systems and sends $\delta_i$ to $\beta_i$. The transported weights are
\[
\begin{aligned}
\nu_-&=(-1,-2,-3,-4,-5,-6,-7,-9),\\
\nu_+&=(-1,-2,-3,-4,-5,-6,-7,9).
\end{aligned}
\]
Both extended sequences have the special RS shape
\[
p((\nu_-)^-)=p((\nu_+)^-)=[2^8].
\]
Each of the eight rows contains one even box, so Proposition \ref{integral} gives
\[
\aff_{D_8}(w_{\nu_\pm})
=\sum_{i=1}^8(i-1)\cdot1
=0+1+\cdots+7=28.
\]
Under the root system isomorphism above,
\[
\aff_{[\lambda_\pm]}(w_{\lambda_\pm})
=\aff_{D_8}(w_{\nu_\pm})=28.
\]
Since $|\Phi_{E_8}^+|=120$, Proposition \ref{pr:main1} yields
\[
\gkd L(\lambda_-)=\gkd L(\lambda_+)=120-28=92.
\]

 From \cite[Lem. 7.6]{BMW}, we have
\[
\begin{aligned}
w_-=w_{\nu_-}&=(8,7,6,5,4,3,2,1),\\
w_+=w_{\nu_+}&=(-8,7,6,5,4,3,2,-1).
\end{aligned}
\]
 They are precisely the canonical elements $w_I$ and $w_{II}$ of \cite[Lem.~7.8]{BMW} for $[2^8]$. Their special domino tableaux have, respectively, zero and two vertical dominoes; see \cite[Rem.~7.9]{BMW}. Hence the local Springer orbits have underlying partition $[2^8]$ and BMW labels $I$ and $II$, respectively. For the transposed partition $[8,8]$, we therefore select
\[
J_-=J_I=\{\beta_1,\ldots,\beta_7\},\qquad
J_+=J_{II}=\{\beta_1,\ldots,\beta_6,\beta_8\}.
\]

Pulling the selected roots back under $\phi$ gives
\[
K_-=\{\delta_1,\ldots,\delta_7\},\qquad
K_+=\{\delta_1,\ldots,\delta_6,\delta_8\}.
\]
We now apply the principal $A_7$ formula to these roots. Identifying roots and coroots as above, we obtain
\[
\begin{aligned}
h_-&=\sum_{i=1}^7 i(8-i)\delta_i
   =(-7,-5,-3,-1,1,3,5,7),\\
h_+&=\sum_{i=1}^6 i(8-i)\delta_i+7\delta_8
   =(7,-5,-3,-1,1,3,5,7).
\end{aligned}
\]
Applying Lemma \ref{find-w-lambda} to each $h_\pm$ and changing sign gives dominant elements with weighted Dynkin diagrams
\[
\begin{aligned}
(\alpha_i(h_-^*))_{1\leq i\leq8}&=(2,0,0,0,0,2,0,2),\\
(\alpha_i(h_+^*))_{1\leq i\leq8}&=(1,0,0,1,0,1,1,0).
\end{aligned}
\]
By Table \ref{tabchare8}, the embedded Sommers data have types $(A_7)''$ and $(A_7)'$, respectively, with the regular orbit in each case. Theorem \ref{P:Som} and \cite[Table~X]{Sommers} then give
\[
\begin{aligned}
\mathcal O_{{\rm Ann}(L(\lambda_-))}&=\mathcal O(D_4+A_1),\\
\mathcal O_{{\rm Ann}(L(\lambda_+))}&=\mathcal O(D_4(a_1)+A_2).
\end{aligned}
\]
The two orbits are distinct and both have dimension $184$. Thus, even with the integral root system fixed, the underlying local partition $[2^8]$ does not determine the annihilator orbit.

\end{ex}

\section{Integral exceptional types}\label{e-gkd}\label{sec:E-cases}

In this section, we collect the integral exceptional data required for the general reduction in Section \ref{S:gen}. For an integral weight, Corollary \ref{cor:integral-case} shows that a corresponding two-sided cell determines both the GK dimension and the annihilator orbit. We first treat the non-simply-laced exceptional types $G_2,F_4$, and then the simply-laced types $E_6,E_7,E_8$. The corresponding pseudo-maximal nonintegral cases are treated in Section \ref{sec:pseudo-nonintegral}.

In each case, let $\Delta=\{\alpha_1,\ldots,\alpha_n\}$ be the chosen simple system. For $I\subseteq\Delta$, let $W_I=\langle s_\alpha\mid\alpha\in I\rangle$ and denote its longest element by $w_I$. In the tables below, we use the abbreviation
$w_{i_1,\ldots,i_r}:=w_{\{\alpha_{i_1},\ldots,\alpha_{i_r}\}}$ for $r\geq 2$.
The third column of each table gives a representative of the two-sided cell containing $w_\lambda$, except that $\mathcal C$ and $\mathcal Cw_0$ denote the entire cells. The first and last rows correspond to $w_\lambda=w_0$ and $w_\lambda=e$, respectively.

\subsection{Type \texorpdfstring{$G_2$}{G2}}\label{g2-gkd}
For type $G_2$, the positive roots are \[
\alpha_1=\ep_1-\ep_2, \alpha_2=-2\ep_1+\ep_2+\ep_3
\]
and $ \alpha_1+\alpha_2=-\ep_1+\ep_3 $, $ 2\alpha_1+\alpha_2=-\ep_2+\ep_3 $, $ 3\alpha_1+\alpha_2=\ep_1-2\ep_2 +\ep_3 $, $ 3\alpha_1+2\alpha_2=-\ep_1-\ep_2+2\ep_3 $. The simple roots are $ \alpha_1 $ and $ \alpha_2 $. The highest root is $\beta=3\alpha_1+2\alpha_2=-\ep_1-\ep_2+2\ep_3.$

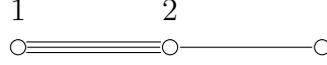
\begin{figure}[htpb]
	\centering 
	\begin{tikzpicture}
  \draw(6,0)circle[radius=0.1];
 	\draw(2,0)circle[radius=0.1];
 	\draw(4,0)circle[radius=0.1];
 	\draw (4.13,0)--(5.87,0) (2.12,0)--(3.87,0)  (2.11,0.07)--(3.88,0.07) (2.11,-0.07)--(3.88,-0.07);
 	\node at (2,0.5) {1};
 	\node at (4,0.5) {2};
 \end{tikzpicture}
	\caption{Extended Dynkin diagram of $G_2$}
	\label{dg2}
\end{figure}


We have $ \hs=\{ (\lam_1,\lam_2,\lam_3)\in\mathbb{C}^3\mid \lam_1+\lam_2+\lam_3 =0\} $. Fix $ \lam\in\hs $. We have \[
\bil{\lam}{(\ep_1- \ep_2)}=\lam_1- \lam_2,
\]
\[
\bil{\lam}{(-2\ep_1+\ep_2+\ep_3)}=\frac{-2\lam_1+\lam_2+\lam_3}{3}=-\lam_1.
\]
Then $ \lam=(\lambda_1,\lambda_2,\lambda_3) \in \mathfrak{h}^*$ is an integral weight if and only if $ (\lam_1,\lam_2)\in\mathbb{Z}^2 $.

All nilpotent orbits of $G_2$ have distinct dimensions \cite[p.~128]{CM}.

\begin{Prop}\label{g2}
Let $L(\lambda)$ be an irreducible integral highest weight module of $G_2$. Then the two-sided cell of $w_\lambda$, its $\mathbf a$-value, the GK dimension, and the annihilator orbit are related as in Table \ref{tab:int-g2}.
\end{Prop}

\begingroup
\small
\setlength{\tabcolsep}{4pt}
\renewcommand{\arraystretch}{1.08}
\begin{longtable}{@{}c c >{\centering\arraybackslash}p{0.28\textwidth} p{0.25\textwidth}@{}}
\caption{Integral two-sided cell data for $G_2$}\label{tab:int-g2}\\
\toprule
$\gkd L(\lambda)$ & $\mathbf a$ & cell representative & $\mathcal O_{{\rm Ann}(L(\lambda))}$ \\
\midrule
\endfirsthead
\multicolumn{4}{c}{\small Table \ref{tab:int-g2}, continued}\\
\toprule
$\gkd L(\lambda)$ & $\mathbf a$ & cell representative & $\mathcal O_{{\rm Ann}(L(\lambda))}$ \\
\midrule
\endhead
\midrule
\multicolumn{4}{r}{\small continued on the next page}\\
\endfoot
\bottomrule
\endlastfoot
0 & 6 & $w_0$ & $0$ \\
5 & 1 & $s_1$ & $G_2(a_1)$ \\
6 & 0 & $e$ & $G_2$ \\
\end{longtable}
\endgroup
\begin{proof}
By \cite[\S~4.8]{lusztig1984char}, $W(G_2)$ has three two-sided cells, corresponding under the Springer correspondence to the special orbits $\{0\}$, $G_2(a_1)$, and $G_2$. The cells containing $w_0$, $s_1$, and $e$ have $\mathbf a$-values $6$, $1$, and $0$, respectively. Proposition \ref{pr:main1} therefore gives the three GK dimensions $0$, $5$, and $6$. Corollary \ref{cor:integral-case} identifies the associated annihilator orbits with $\{0\}$, $G_2(a_1)$, and $G_2$, respectively.
\end{proof}

There are in total five nilpotent orbits for $G_2$ as listed in \cite[\S~8.4]{CM}, of which there are only three special ones, given as above. 

\subsection{Type \texorpdfstring{$F_4$}{F4}}\label{f4-gkd}
For type $F_4$, the set $\Phi^+$ of positive roots is
\[
\{\ep_i\mid 1\lest i \lest 4\}\cup \{  \ep_i\pm \ep_j\mid 1\lest i<j \lest 4\} \cup\left \{ \frac12(\ep_1\pm \ep_2\pm\ep_3\pm \ep_4) \right\}.
\]
The  simple roots are $ \alpha_1=\ep_2-\ep_3,\alpha_2=\ep_3-\ep_4,\alpha_3=\ep_4, \alpha_4=\frac12(\ep_1- \ep_2-\ep_3- \ep_4)$.  The highest root is $\beta=2\alpha_1+3\alpha_2+4\alpha_3+2\alpha_4=\ep_1+\ep_2.$

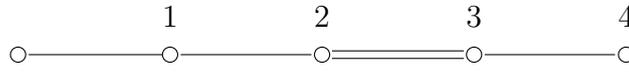
\begin{figure}[htpb]
	\centering 
	\begin{tikzpicture}
 	\foreach \i in {0,2,4,...,8}{\draw(\i,0)circle[radius=0.1];}
 		\draw (0.14,0)--(1.9,0)(2.14,0)--(3.86,0)  (4.13,0.05)--(5.865,0.05) (4.13,-0.05)--(5.865,-0.05)  (6.14,0)--(7.855,0) ;
 	\node at (0,0.5) {};
  \node at (2,0.5) {1};
 	\node at (4,0.5) {2};
 	\node at (6,0.5) {3};
 	\node at (8,0.5) {4};
 \end{tikzpicture}
	\caption{Extended Dynkin diagram of $F_4$}
	\label{df4}
\end{figure}

For $ \lam=(\lam_1,\lam_2,\lam_3,\lam_4)\in\hs $, we have \[
\bil{\lam}{(\ep_i)}=2\lam_i,
\]
\[
\bil{\lam}{(\ep_i\pm \ep_j)}=\lam_i\pm \lam_j,
\]
\[
\left\langle{\lam},{\left(\frac12(\ep_1\pm \ep_2\pm\ep_3\pm \ep_4) \right)^{\vee}}\right\rangle=\lam_1\pm\lam_2\pm \lam_3\pm \lam_4.
\]


We write $\mathcal{O}(L)$ for the nilpotent orbit with Bala--Carter label $L$, following \cite{Ca85}. In many rows the orbit is uniquely determined by its dimension; additional cell data are needed only when several nilpotent orbits have the same dimension. 

For $\lambda\in\mathfrak{h}^*$, write $\lambda=w_{\lambda}\mu$, where $\mu$ is antidominant and $w_{\lambda}\in W_{[\mu]}^J$.

\begin{Prop}\label{f4}
Let $L(\lambda)$ be an irreducible integral highest weight module of $F_4$. Then the two-sided cell of $w_\lambda$, its $\mathbf a$-value, the GK dimension, and the annihilator orbit are related as in Table \ref{tab:int-f4}.
\end{Prop}

\begingroup
\small
\setlength{\tabcolsep}{4pt}
\renewcommand{\arraystretch}{1.08}
\begin{longtable}{@{}c c >{\centering\arraybackslash}p{0.30\textwidth} p{0.25\textwidth}@{}}
\caption{Integral two-sided cell data for $F_4$}\label{tab:int-f4}\\
\toprule
$\gkd L(\lambda)$ & $\mathbf a$ & cell representative & $\mathcal O_{{\rm Ann}(L(\lambda))}$ \\
\midrule
\endfirsthead
\multicolumn{4}{c}{\small Table \ref{tab:int-f4}, continued}\\
\toprule
$\gkd L(\lambda)$ & $\mathbf a$ & cell representative & $\mathcal O_{{\rm Ann}(L(\lambda))}$ \\
\midrule
\endhead
\midrule
\multicolumn{4}{r}{\small continued on the next page}\\
\endfoot
\bottomrule
\endlastfoot
0 & 24 & $w_0$ & $0$ \\
11 & 13 & $\mathcal Cw_0$ & $\widetilde A_1$ \\
14 & 10 & $x_{14}^{F_4}$ & $A_1+\widetilde A_1$ \\
15 & 9 & $w_{1,2,3}$ & $\widetilde A_2$ \\
 &  & $w_{2,3,4}$ & $A_2$ \\
20 & 4 & $w_{1,3,4}$ & $F_4(a_3)$ \\
21 & 3 & $w_{3,4}$ & $B_3$ \\
 &  & $w_{1,2}$ & $C_3$ \\
22 & 2 & $w_{1,3}$ & $F_4(a_2)$ \\
23 & 1 & $\mathcal C$ & $F_4(a_1)$ \\
24 & 0 & $e$ & $F_4$ \\
\end{longtable}
\endgroup
\begin{proof}
The tables in \cite[\S~8.4]{CM} give the $11$ special nilpotent orbits, and hence the $11$ two-sided cells of $W(F_4)$. For a standard parabolic longest element $w_I$, Lemma \ref{alem1} gives
\[
\aff(w_I)=\ell(w_I)=|\Phi_I^+|,
\]
so the corresponding GK dimension is $24-|\Phi_I^+|$. The extremal rows of Table \ref{tab:int-f4} come from $e$ and $w_0$, while \cite[Lem.~3.5]{BMXX} identifies the cells $\mathcal C$ and $\mathcal Cw_0$.

For the remaining parabolic representatives, the Dynkin type of $\Phi_I$ determines the $\aff$-value, and the corresponding special orbit is identified from \cite{BC76}, \cite[\S~8.4]{CM}, and the Springer tables in \cite[Table~C.3]{GP} and \cite[\S~13.3]{Ca85}. For example, $w_{1,3,4}$ and $w_{1,3}$ lie in the cells corresponding to $F_4(a_3)$ and $F_4(a_2)$, respectively. The two cells with $\aff=3$ are distinguished by the Springer characters of $w_{1,2}$ and $w_{3,4}$, giving the orbits $C_3$ and $B_3$. Similarly, the two parabolic longest elements $w_{1,2,3}$ and $w_{2,3,4}$ give $\widetilde A_2$ and $A_2$. Thus the parabolic representatives occurring in the table give
\[
\begin{array}{c|c}
w_{1,3,4}&F_4(a_3)\\
w_{1,3}&F_4(a_2)\\
w_{1,2}&C_3\\
w_{3,4}&B_3\\
w_{1,2,3}&\widetilde A_2\\
w_{2,3,4}&A_2.
\end{array}
\]

There remains the cell with $\aff=10$, for which no convenient representative is given by a parabolic longest element. Using the PyCox procedure of Remark \ref{represent}, we choose the representative $x_{14}^{F_4}$ with $\aff(x_{14}^{F_4})=10$. Hence the corresponding GK dimension is $24-10=14$, and the Springer character identifies the orbit as $A_1+\widetilde A_1$. This completes Table \ref{tab:int-f4}.
\end{proof}

The PyCox procedure used to choose representatives with a prescribed $\mathbf a$-value is described in Remark \ref{represent}.

The table contains the $11$  two-sided cells of $W(F_4)$. The full PyCox word for $x_{14}^{F_4}$ is recorded in Appendix \ref{app:cell-reps}; the special orbit data are compared with \cite[Table~C.3]{GP} and \cite[\S~13.3]{Ca85}.

We next treat the simply-laced exceptional types $E_6,E_7,E_8$.

To keep the classification statements readable, when a two-sided cell has no convenient representative given by a parabolic longest element, we write $x_d^{E_n}$ for a fixed PyCox representative, where $d$ is the corresponding GK dimension. The full Geck words for these representatives are collected in Appendix \ref{app:cell-reps}.

For the types $E_n$, $n=6,7,8$, we use a common realization in $\mathbb{R}^8$ with simple roots
\begin{align*}
\alpha_1&=\frac{1}{2}(\ep_1-\ep_2-\ep_3-\ep_4-\ep_5-\ep_6-\ep_7+\ep_8) ,\\
\alpha_2&=\ep_1+\ep_2,\\
\alpha_k&=\ep_{k-1}-\ep_{k-2} \text{ with }3\lest k\lest n,
\end{align*}
where $n=6, 7,8$.

\subsection{Type \texorpdfstring{$E_6$}{E6}}
{For type $E_6$,}
\[
\Phi^+=
\{\ep_i \pm \ep_j \mid 5\geq i > j\geq 1\}\cup 
\Big\{\frac{1}{2}(\sum_{i=1}^5 (-1)^{n(i)}\ep_i -\ep_6-\ep_7+\ep_8)\mid \sum_{i=1}^5n(i) \text{ even}\Big\}.
\]
Moreover, $ \hs=\{ (\lam_1,\lam_2,\cdots,\lam_8)\in\mathbb{C}^8\mid \lam_6=\lam_7=-\lam_8\} $.  The highest root is \[\beta=\alpha_1+2\alpha_2+2\alpha_3+3\alpha_4+2\alpha_5+\alpha_6=\frac{1}{2}(\ep_1+\ep_2+\ep_3+\ep_4+\ep_5-\ep_6-\ep_7+\ep_8).\]

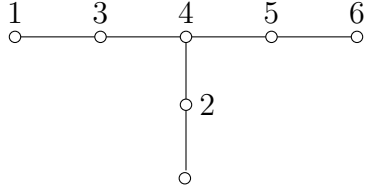
\begin{figure}[htpb]
	\centering 
	
	\begin{tikzpicture}[scale=1.5,baseline=0]

 \draw (0,0.05) node[above=1pt]{$ 1 $} circle [radius=0.05];
	\draw (0.05,0.05)--++(0.65,0); 
	\draw (0.75,0.05) node[above=1pt]{$ 3 $} circle[radius=0.05];  
	\draw (0.8,0.05)--++(0.65,0);
	\draw  (1.45,0) ++(0.05,0.05) node[above=1pt]{$ 4$} circle[radius=0.05];
	\draw (1.55,0.05)--++(0.65,0);
	\draw  (2.2,0) ++(0.05,0.05) node[above=1pt]{$ 5$} circle[radius=0.05];
	\draw (2.3,0.05)--++(0.65,0);
	\draw  (2.95,0) ++(0.05,0.05) node[above=1pt]{$ 6$} circle[radius=0.05];
	\draw (1.5,0)--++(0,-0.5);
	\draw  (1.45,-0.6) ++(0.05,0.05) node[right=1pt]{$ 2$} circle[radius=0.05];

\draw (1.49,-1.2) node[above=1pt]{} circle [radius=0.05];
	\draw (1.5,-0.6)--++(0,-0.53); 
 
	\end{tikzpicture}

	\caption{Extended Dynkin diagram of $E_6$}
	\label{E6-dy}
\end{figure}

Fix $ \lam\in\hs $. We have \[
\bil{\lam}{(\alpha_1)}=\frac{1}{2}(\lam_1- \lam_2-\cdots-\lam_7+\lam_8),
\]
\[
\bil{\lam}{(\ep_i\pm \ep_j)}=\lam_i\pm \lam_j.
\]
Then $ \lam=(\lam_1,\lam_2,\cdots,\lam_8) $ is an integral weight if and only if $\lam_1- \lam_2-\cdots-\lam_7+\lam_8\in 2\mathbb{Z}$,  $ \lam_1-\lam_2\in\mathbb{Z}, \lam_2-\lam_3\in\mathbb{Z},\lam_3-\lam_4\in\mathbb{Z},\lam_4-\lam_5\in\mathbb{Z}$ and $2\lam_i\in \mathbb{Z} $ for $1\lest i\lest 5$.

\begin{Prop}\label{E6}
Let $L(\lambda)$ be an irreducible integral highest weight module of $E_6$. Then the two-sided cell of $w_\lambda$, its $\mathbf a$-value, the GK dimension, and the annihilator orbit are related as in Table \ref{tab:int-e6}.
\end{Prop}

\begingroup
\small
\setlength{\tabcolsep}{4pt}
\renewcommand{\arraystretch}{1.08}
\begin{longtable}{@{}c c >{\centering\arraybackslash}p{0.30\textwidth} p{0.25\textwidth}@{}}
\caption{Integral two-sided cell data for $E_6$}\label{tab:int-e6}\\
\toprule
$\gkd L(\lambda)$ & $\mathbf a$ & cell representative & $\mathcal O_{{\rm Ann}(L(\lambda))}$ \\
\midrule
\endfirsthead
\multicolumn{4}{c}{\small Table \ref{tab:int-e6}, continued}\\
\toprule
$\gkd L(\lambda)$ & $\mathbf a$ & cell representative & $\mathcal O_{{\rm Ann}(L(\lambda))}$ \\
\midrule
\endhead
\midrule
\multicolumn{4}{r}{\small continued on the next page}\\
\endfoot
\bottomrule
\endlastfoot
0 & 36 & $w_0$ & $0$ \\
11 & 25 & $\mathcal Cw_0$ & $A_1$ \\
16 & 20 & $w_{2,3,4,5,6}$ & $2A_1$ \\
21 & 15 & $w_{1,3,4,5,6}$ & $A_2$ \\
23 & 13 & $x_{23}^{E_6}$ & $A_2+A_1$ \\
24 & 12 & $w_{2,3,4,5}$ & $2A_2$ \\
25 & 11 & $w_{1,2,4,5,6}$ & $A_2+2A_1$ \\
26 & 10 & $w_{3,4,5,6}$ & $A_3$ \\
29 & 7 & $w_{1,4,5,6}$ & $D_4(a_1)$ \\
30 & 6 & $w_{4,5,6}$ & $A_4$ \\
 &  & $w_{1,3,5,6}$ & $D_4$ \\
31 & 5 & $w_{1,2,5,6}$ & $A_4+A_1$ \\
32 & 4 & $w_{1,5,6}$ & $D_5(a_1)$ \\
33 & 3 & $w_{5,6}$ & $E_6(a_3)$ \\
34 & 2 & $w_{1,6}$ & $D_5$ \\
35 & 1 & $\mathcal C$ & $E_6(a_1)$ \\
36 & 0 & $e$ & $E_6$ \\
\end{longtable}
\endgroup

\subsection{Type \texorpdfstring{$E_7$}{E7}}
{The positive roots are}
\begin{align*}
  \Phi^+=&
\{\ep_i \pm \ep_j \mid 6\geq i > j\geq 1\}\cup\{\ep_8-\ep_7\} \\
&\cup
\Big\{ \frac{1}{2}(\sum_{i=1}^6 (-1)^{n(i)}\ep_i -\ep_7+\ep_8)\mid \sum_{i=1}^6n(i) \text{ odd} \Big\}  
\end{align*}
and $ \hs=\{ (\lam_1,\lam_2,\cdots,\lam_8)\in\mathbb{C}^8\mid \lam_7=-\lam_8\} $. The highest root is \[\beta=2\alpha_1+2\alpha_2+3\alpha_3+4\alpha_4+3\alpha_5+2\alpha_6+\alpha_7=-\ep_7+\ep_8.\]

\begin{figure}[htpb]
	\centering

	\begin{tikzpicture}[scale=1.5,baseline=0]
	\draw (-0.75,0.05) node[above=1pt]{} circle [radius=0.05];
	\draw (-0.7,0.05)--++(0.63,0);

 \draw (0,0.05) node[above=1pt]{$ 1 $} circle [radius=0.05];
	\draw (0.05,0.05)--++(0.65,0); 
	\draw (0.75,0.05) node[above=1pt]{$ 3 $} circle[radius=0.05];  
	\draw (0.8,0.05)--++(0.65,0);
	\draw  (1.45,0) ++(0.05,0.05) node[above=1pt]{$ 4$} circle[radius=0.05];
	\draw (1.55,0.05)--++(0.65,0);
	\draw  (2.2,0) ++(0.05,0.05) node[above=1pt]{$ 5$} circle[radius=0.05];
	\draw (2.3,0.05)--++(0.65,0);
	\draw  (2.95,0) ++(0.05,0.05) node[above=1pt]{$ 6$} circle[radius=0.05];
	\draw (3.05,0.05)--++(0.65,0);
	\draw  (3.7,0) ++(0.05,0.05) node[above=1pt]{$ 7$} circle[radius=0.05];
	\draw (1.5,0)--++(0,-0.5);
	\draw  (1.45,-0.6) ++(0.05,0.05) node[right=1pt]{$ 2$} circle[radius=0.05];
	\end{tikzpicture}
	
	\caption{Extended Dynkin diagram of $E_7$}
	\label{E7-dy}
\end{figure}
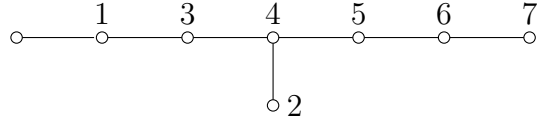

Fix $ \lam\in\hs $. We have \[
\bil{\lam}{(\alpha_1)}=\frac{1}{2}(\lam_1- \lam_2-\cdots-\lam_7+\lam_8),
\]
\[
\bil{\lam}{(\ep_i\pm \ep_j)}=\lam_i\pm \lam_j.
\]
Then $ \lam=(\lam_1,\lam_2,\cdots,\lam_8) $ is an integral weight if and only if $\lam_1- \lam_2-\cdots-\lam_7+\lam_8\in 2\mathbb{Z}$,  $ \lam_1-\lam_2\in\mathbb{Z}, \lam_2-\lam_3\in\mathbb{Z},\lam_3-\lam_4\in\mathbb{Z},\lam_4-\lam_5\in\mathbb{Z}, \lam_5-\lam_6\in\mathbb{Z}$ and $2\lam_i\in \mathbb{Z} $ for $1\lest i\lest 6$.


\begin{Prop}\label{E7}
Let $L(\lambda)$ be an irreducible integral highest weight module of $E_7$. Then the two-sided cell of $w_\lambda$, its $\mathbf a$-value, the GK dimension, and the annihilator orbit are related as in Table \ref{tab:int-e7}.
\end{Prop}

\begingroup
\small
\setlength{\tabcolsep}{4pt}
\renewcommand{\arraystretch}{1.08}
\begin{longtable}{@{}c c >{\centering\arraybackslash}p{0.30\textwidth} p{0.25\textwidth}@{}}
\caption{Integral two-sided cell data for $E_7$}\label{tab:int-e7}\\
\toprule
$\gkd L(\lambda)$ & $\mathbf a$ & cell representative & $\mathcal O_{{\rm Ann}(L(\lambda))}$ \\
\midrule
\endfirsthead
\multicolumn{4}{c}{\small Table \ref{tab:int-e7}, continued}\\
\toprule
$\gkd L(\lambda)$ & $\mathbf a$ & cell representative & $\mathcal O_{{\rm Ann}(L(\lambda))}$ \\
\midrule
\endhead
\midrule
\multicolumn{4}{r}{\small continued on the next page}\\
\endfoot
\bottomrule
\endlastfoot
0 & 63 & $w_0$ & $0$ \\
17 & 46 & $\mathcal Cw_0$ & $A_1$ \\
26 & 37 & $x_{26}^{E_7}$ & $2A_1$ \\
27 & 36 & $w_{1,2,3,4,5,6}=w_{E_6}$ & $(3A_1)''$ \\
33 & 30 & $w_{2,3,4,5,6,7}$ & $A_2$ \\
38 & 25 & $x_{38}^{E_7}$ & $A_2+A_1$ \\
41 & 22 & $x_{41}^{E_7}$ & $A_2+2A_1$ \\
42 & 21 & $w_{1,2,3,4,5,7}$ {\scriptsize (\texttt{168\_a'})} & $2A_2$ \\
 &  & $w_{1,3,4,5,6,7}$ {\scriptsize (\texttt{105\_b'})} & $A_2+3A_1$ \\
 &  & $x_{42}^{E_7}$ {\scriptsize (\texttt{210\_a'})} & $A_3$ \\
43 & 20 & $w_{2,3,4,5,6}$ & $(A_3+A_1)''$ \\
47 & 16 & $w_{1,2,4,5,6,7}$ & $D_4(a_1)$ \\
48 & 15 & $w_{1,3,4,5,6}$ {\scriptsize (\texttt{405\_a'})} & $D_4(a_1)+A_1$ \\
 &  & $w_{2,4,5,6,7}$ {\scriptsize (\texttt{105\_c'})} & $D_4$ \\
49 & 14 & $x_{49}^{E_7}$ & $A_3+A_2$ \\
50 & 13 & $w_{2,3,4,5,7}$ {\scriptsize (\texttt{420\_a'})} & $A_4$ \\
 &  & $w_{1,2,3,4,6,7}$ {\scriptsize (\texttt{210\_b'})} & $A_3+A_2+A_1$ \\
51 & 12 & $w_{2,3,4,5}$ & $(A_5)''$ \\
52 & 11 & $w_{1,2,4,5,6}$ & $A_4+A_1$ \\
53 & 10 & $w_{3,4,5,6}$ {\scriptsize (\texttt{420\_a})} & $D_5(a_1)$ \\
 &  & $w_{1,2,3,5,6,7}$ {\scriptsize (\texttt{210\_b})} & $A_4+A_2$ \\
54 & 9 & $w_{1,3,5,6,7}$ & $D_5(a_1)+A_1$ \\
55 & 8 & $w_{1,2,5,6,7}$ & $E_6(a_3)$ \\
56 & 7 & $w_{1,4,5,6}$ {\scriptsize (\texttt{315\_a'})} & $E_7(a_5)$ \\
 &  & $w_{2,4,5,7}$ {\scriptsize (\texttt{189\_c'})} & $D_5$ \\
57 & 6 & $w_{4,5,6}$ {\scriptsize (\texttt{210\_a})} & $D_6(a_1)$ \\
 &  & $w_{1,3,5,6}$ {\scriptsize (\texttt{168\_a})} & $D_5+A_1$ \\
 &  & $w_{1,2,3,5,7}$ {\scriptsize (\texttt{105\_b})} & $A_6$ \\
58 & 5 & $w_{1,2,5,6}$ & $E_7(a_4)$ \\
59 & 4 & $w_{1,5,6}$ & $E_6(a_1)$ \\
60 & 3 & $w_{5,6}$ {\scriptsize (\texttt{56\_a'})} & $E_7(a_3)$ \\
 &  & $w_{2,5,7}$ {\scriptsize (\texttt{21\_b'})} & $E_6$ \\
61 & 2 & $w_{1,6}$ & $E_7(a_2)$ \\
62 & 1 & $\mathcal C$ & $E_7(a_1)$ \\
63 & 0 & $e$ & $E_7$ \\
\end{longtable}
\endgroup

When several two-sided cells have the same $\mathbf a$-value, the corresponding special representation is needed to distinguish their Springer orbits in PyCox.  For precisely these repeated-$\mathbf a$ rows, Table \ref{tab:int-e7} records in parentheses the PyCox label returned by the field \texttt{special} of \texttt{klcellrepelm}.

\subsection{Type \texorpdfstring{$E_8$}{E8}}
{The positive roots are}
\begin{align*}
  \Phi^+=&
\{\ep_i \pm \ep_j \mid 8\geq i > j\geq 1\} \cup
\Big\{\frac{1}{2}(\sum_{i=1}^7 (-1)^{n(i)}\ep_i+\ep_8)\mid \sum_{i=1}^7n(i) \text{ even}\Big\}
\end{align*}
and $ \hs=\mathbb{C}^8$. The highest root is \[\beta=2\alpha_1+3\alpha_2+4\alpha_3+6\alpha_4+5\alpha_5+4\alpha_6+3\alpha_7+2\alpha_8=\ep_7+\ep_8.\]

\begin{figure}[htpb]
	\centering

	\begin{tikzpicture}[scale=1.5,baseline=0]

 \draw (0,0.05) node[above=1pt]{$ 1 $} circle [radius=0.05];
	\draw (0.05,0.05)--++(0.65,0); 
	\draw (0.75,0.05) node[above=1pt]{$ 3 $} circle[radius=0.05];  
	\draw (0.8,0.05)--++(0.65,0);
	\draw  (1.45,0) ++(0.05,0.05) node[above=1pt]{$ 4$} circle[radius=0.05];
	\draw (1.55,0.05)--++(0.65,0);
	\draw  (2.2,0) ++(0.05,0.05) node[above=1pt]{$ 5$} circle[radius=0.05];
	\draw (2.3,0.05)--++(0.65,0);
	\draw  (2.95,0) ++(0.05,0.05) node[above=1pt]{$ 6$} circle[radius=0.05];
	\draw (3.05,0.05)--++(0.65,0);
	\draw  (3.7,0) ++(0.05,0.05) node[above=1pt]{$ 7$} circle[radius=0.05];
	\draw (3.8,0.05)--++(0.65,0);
	\draw  (4.45,0) ++(0.05,0.05) node[above=1pt]{$ 8$} circle[radius=0.05];
	\draw (1.5,0)--++(0,-0.5);
	\draw  (1.45,-0.6) ++(0.05,0.05) node[right=1pt]{$ 2$} circle[radius=0.05];

\draw (5.2,0.05) node[above=1pt]{} circle [radius=0.05];
	\draw (4.53,0.05)--++(0.65,0);

	\end{tikzpicture}
	\caption{Extended Dynkin diagram of $E_8$}
	\label{E8-dy}
\end{figure}
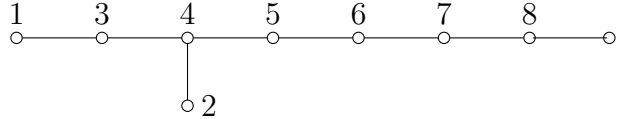

Fix $ \lam=\sum_{i=1}^8 \lambda_i e_i \in\hs $. We have \[
\bil{\lam}{(\alpha_1)}=\frac{1}{2}(\lam_1- \lam_2-\cdots-\lam_7+\lam_8),
\]
\[
\bil{\lam}{(\ep_i\pm \ep_j)}=\lam_i\pm \lam_j.
\]
Then $ \lam=(\lam_1,\lam_2,\cdots,\lam_8) $ is an integral weight if and only if $\lam_1- \lam_2-\cdots-\lam_7+\lam_8\in 2\mathbb{Z}$,  $ \lam_1-\lam_2\in\mathbb{Z}, \lam_2-\lam_3\in\mathbb{Z},\lam_3-\lam_4\in\mathbb{Z},\lam_4-\lam_5\in\mathbb{Z}, \lam_5-\lam_6\in\mathbb{Z},\lam_6-\lam_7\in\mathbb{Z}$ and $2\lam_i\in \mathbb{Z} $ for $1\lest i\lest 7$.

\begin{Prop}\label{E8}
Let $L(\lambda)$ be an irreducible integral highest weight module of $E_8$. Then the two-sided cell of $w_\lambda$, its $\mathbf a$-value, the GK dimension, and the annihilator orbit are related as in Table \ref{tab:int-e8}.
\end{Prop}

\begingroup
\small
\setlength{\tabcolsep}{4pt}
\renewcommand{\arraystretch}{1.08}
\begin{longtable}{@{}c c >{\centering\arraybackslash}p{0.33\textwidth} p{0.25\textwidth}@{}}
\caption{Integral two-sided cell data for $E_8$}\label{tab:int-e8}\\
\toprule
$\gkd L(\lambda)$ & $\mathbf a$ & cell representative & $\mathcal O_{{\rm Ann}(L(\lambda))}$ \\
\midrule
\endfirsthead
\multicolumn{4}{c}{\small Table \ref{tab:int-e8}, continued}\\
\toprule
$\gkd L(\lambda)$ & $\mathbf a$ & cell representative & $\mathcal O_{{\rm Ann}(L(\lambda))}$ \\
\midrule
\endhead
\midrule
\multicolumn{4}{r}{\small continued on the next page}\\
\endfoot
\bottomrule
\endlastfoot
0 & 120 & $w_0$ & $0$ \\
29 & 91 & $\mathcal Cw_0$ & $A_1$ \\
46 & 74 & $x_{46}^{E_8}$ & $2A_1$ \\
57 & 63 & $w_{1,2,3,4,5,6,7}=w_{E_7}$ & $A_2$ \\
68 & 52 & $x_{68}^{E_8}$ & $A_2+A_1$ \\
73 & 47 & $x_{73}^{E_8}$ & $A_2+2A_1$ \\
74 & 46 & $x_{74}^{E_8}$ & $A_3$ \\
78 & 42 & $w_{2,3,4,5,6,7,8}$ & $2A_2$ \\
83 & 37 & $w_{1,2,3,4,5,6,8}=w_{E_6}s_8$ & $D_4(a_1)$ \\
84 & 36 & $w_{1,2,3,4,5,6}=w_{E_6}$ & $D_4$ \\
88 & 32 & $x_{88}^{E_8}$ & $D_4(a_1)+A_1$ \\
89 & 31 & $x_{89}^{E_8}$ & $A_3+A_2$ \\
90 & 30 & $w_{2,3,4,5,6,7}$ & $A_4$ \\
92 & 28 & $w_{1,3,4,5,6,7,8}$ & $D_4(a_1)+A_2$ \\
94 & 26 & $x_{94}^{E_8}$ & $A_4+A_1$ \\
95 & 25 & $x_{95}^{E_8}$ & $D_5(a_1)$ \\
96 & 24 & $x_{96}^{E_8}$ & $A_4+2A_1$ \\
97 & 23 & $w_{1,2,3,4,5}w_{7,8}$ & $A_4+A_2$ \\
98 & 22 & $w_{1,2,4,5,6,7,8}$ {\scriptsize (\texttt{2835\_x'})} & $A_4+A_2+A_1$ \\
 &  & $x_{98}^{E_8}$ {\scriptsize (\texttt{6075\_x'})} & $D_5(a_1)+A_1$ \\
99 & 21 & $w_{2,3,4,5,6,8}$ {\scriptsize (\texttt{5600\_z'})} & $E_6(a_3)$ \\
 &  & $w_{1,3,4,5,6,7}$ {\scriptsize (\texttt{4200\_z'})} & $D_4+A_2$ \\
100 & 20 & $w_{2,3,4,5,6}$ & $D_5$ \\
104 & 16 & $w_{1,2,4,5,6,7}$ & $E_8(a_7)$ \\
105 & 15 & $w_{1,3,4,5,6}$ {\scriptsize (\texttt{5600\_z})} & $D_6(a_1)$ \\
 &  & $w_{2,3,4,5,7,8}$ {\scriptsize (\texttt{4200\_z})} & $A_6$ \\
106 & 14 & $x_{106}^{E_8}$ {\scriptsize (\texttt{6075\_x})} & $E_7(a_4)$ \\
 &  & $w_{1,2,3,5,6,7,8}$ {\scriptsize (\texttt{2835\_x})} & $A_6+A_1$ \\
107 & 13 & $w_{2,3,4,5,7}$ {\scriptsize (\texttt{2800\_z})} & $E_6(a_1)$ \\
 &  & $w_{1,2,3,4,6,7}$ {\scriptsize (\texttt{4536\_z})} & $D_5+A_2$ \\
108 & 12 & $w_{2,3,4,5}$ {\scriptsize (\texttt{525\_x})} & $E_6$ \\
{108} &{12} & $w_{1,3,4,6,7,8}$ {\scriptsize (\texttt{4200\_x})} & $D_7(a_2)$ \\
109 & 11 & $w_{1,2,4,5,6}$ & $E_6(a_1)+A_1$ \\
110 & 10 & $w_{3,4,5,6}$ {\scriptsize (\texttt{2268\_x})} & $E_7(a_3)$ \\
 &  & $w_{1,2,3,5,6,7}$ {\scriptsize (\texttt{2240\_x})} & $E_8(b_6)$ \\
111 & 9 & $w_{1,3,5,6,7}$ & $D_7(a_1)$ \\
112 & 8 & $w_{1,2,5,6,7}$ & $E_8(a_6)$ \\
113 & 7 & $w_{1,4,5,6}$ & $E_8(b_5)$ \\
114 & 6 & $w_{4,5,6}$ {\scriptsize (\texttt{567\_x})} & $E_7(a_1)$ \\
 &  & $w_{1,3,5,6}$ {\scriptsize (\texttt{700\_x})} & $E_8(a_5)$ \\
115 & 5 & $w_{1,2,5,6}$ & $E_8(b_4)$ \\
116 & 4 & $w_{1,5,6}$ & $E_8(a_4)$ \\
117 & 3 & $w_{7,8}$ & $E_8(a_3)$ \\
118 & 2 & $w_{1,8}$ & $E_8(a_2)$ \\
119 & 1 & $\mathcal C$ & $E_8(a_1)$ \\
120 & 0 & $e$ & $E_8$ \\
\end{longtable}
\endgroup



\begin{Rem}\label{rem:E-series-compatibility}
The three integral calculations above are compatible with the standard parabolic inclusions
\[
W(E_6)\subseteq W(E_7)\subseteq W(E_8).
\]
Indeed, if $6\leq m<n\leq 8$ and $w\in W(E_m)\subseteq W(E_n)$, then property \textup{(P12)} for Lusztig's $\aff$-function gives
\[
\aff_{E_m}(w)=\aff_{E_n}(w);
\]
see \cite[\S~14.2, \textup{(P12)}]{Lus03}. Consequently, if
\[
d_{E_k}(w):=|\Phi^+_{E_k}|-\aff_{E_k}(w),
\]
then
\[
d_{E_7}(w)=d_{E_6}(w)+27,\qquad
d_{E_8}(w)=d_{E_7}(w)+57,
\]
and hence $d_{E_8}(w)=d_{E_6}(w)+84$, since
$|\Phi^+_{E_6}|=36$, $|\Phi^+_{E_7}|=63$, and $|\Phi^+_{E_8}|=120$.

The same tower is compatible with the $j$-induction formalism used in Section \ref{sec:sommers}: by Lemma \ref{j.1}, $j$-induction is transitive along the above inclusions, while the proof of Theorem \ref{P:Som} recalls the compatibility of the Springer correspondence with induction of nilpotent orbits. Thus the three tables are not independent computations and provide useful compatibility checks for one another. On the other hand, the  $E_8$-data do not replace the local $E_6$- and $E_7$-data, since passage to a larger Weyl group need not retain the local cell and embedding information needed to identify the annihilator orbit in the smaller group. We therefore retain the three integral tables separately.
\end{Rem}

\section{Nonintegral exceptional cases}\label{sec:pseudo-nonintegral}

In Section \ref{sec:E-cases}, we provide the integral exceptional data entering the general reduction. In this section, we first make the reduction explicit for nonintegral highest weights whose dual integral root systems $\Phi_{[\lambda]}^\vee$ are pseudo-maximal in the sense of Definition \ref{thmbd}.

 Let $\Phi$ be irreducible, with simple roots $\Delta=\{\alpha_i:1\leq i\leq n\}$ and highest root
$\theta=\sum_i c_i\alpha_i.$
Put $\alpha_0=-\theta$ and $\widetilde\Delta=\Delta\cup\{\alpha_0\}$. For a subset $S\subseteq\Phi$, write $\langle S\rangle_{\rm rt}$ for the root subsystem of $\Phi$ generated by $S$. For $1\leq i,j\leq n$, set
\[
\widetilde\Phi(i):=\bigl\langle\widetilde\Delta\setminus\{\alpha_i\}\bigr\rangle_{\rm rt},
\qquad
\Phi(j):=\bigl\langle\Delta\setminus\{\alpha_j\}\bigr\rangle_{\rm rt}.
\]
Let $\mathscr P(\Phi)$ be the collection of the proper root subsystems among the $\widetilde\Phi(i)$ and $\Phi(j)$.
\begin{Defff}\label{thmbd}
A proper root subsystem of $\Phi$ is called \emph{pseudo-maximal} if it is $W$-conjugate to a maximal element of $\mathscr P(\Phi)$ under inclusion.
\end{Defff}
Such a subsystem has rank $n-1$ or $n$.



Throughout this section, $H$ is the group introduced in Section \ref{local-special-data}, so that $\Phi(H)=\Phi_{[\lambda]}$. If $\Phi_{[\lambda]}=\bigsqcup_{i=1}^k\Phi_{[\lambda]_i}$, we write $\lambda'=(\lambda'_{(1)},\ldots,\lambda'_{(k)})$ for the transported integral weight defined in Section \ref{nonint}. For integers $m<n$, write $[m,n]=\{m,m+1,\ldots,n\}$.



\begin{LemUnnumbered}
Suppose $\Psi=\Phi_{[\lambda_0]}$. For every $w\in W(\Psi)$, there is a weight $\lambda$ with $\Phi_{[\lambda]}=\Psi$ and $w_\lambda=w$.
\end{LemUnnumbered}
\begin{proof}
Let $\rho$ be half the sum of the ambient positive roots and put $\mu=\lambda_0-N\rho$ for a sufficiently large positive integer $N$. Since $\rho$ pairs integrally with every coroot, $\Phi_{[\mu]}=\Psi$. For $\alpha\in\Psi^+$, the pairing $\langle\lambda_0,\alpha^\vee\rangle$ is an integer and $\langle\rho,\alpha^\vee\rangle>0$. Hence $\mu$ is strictly antidominant for $\Psi$ when $N$ is large. Set $\lambda=w\mu$. Then $\Phi_{[\lambda]}=w\Psi=\Psi$, and strict antidominance gives $w_\lambda=w$.
\end{proof}

The subsystems in the following tables are realized by integral root systems. Indeed, write the corresponding pseudo-Levi subgroup in $G^\vee$ as $C_{G^\vee}(t)^\circ$ and choose a logarithm of $t$ in the Lie algebra of a maximal torus. The exponential construction in Lemma \ref{lem:pseudolevi-existence} gives the required $\lambda_0$. The observation above then realizes each local special $\aff$-value.

  Pseudo-maximality is used only to make the explicit classification concise. Theorem \ref{thm:general-reduction} applies to arbitrary embedded integral root subsystems. If $2\gkd L(\lambda)$ determines a unique corresponding nilpotent orbit, no further computation is needed. The theorems below cover the remaining allowed cases, where Sommers data are needed to distinguish or identify the orbit and the orbit dimensions are taken from \cite[\S~8.4]{CM}.

\subsection{Type \texorpdfstring{$G_2$}{G2}}

\begin{Prop}\label{g2-nonint}
Let $L(\lambda)$ be a simple nonintegral highest weight module of $G_2$ such that $\Phi_{[\lambda]}^\vee$ is pseudo-maximal. Then the possible GK dimensions are listed below.
\begin{center}
\small
\renewcommand{\arraystretch}{1.12}
\setlength{\tabcolsep}{7pt}
\begin{tabular}{@{}ll@{}}
\toprule
$\Phi_{[\lambda]}$ & allowed $\gkd L(\lambda)$ \\
\midrule
$A_2$ & $3,5,6$ \\
$A_1\times\widetilde A_1$ & $4,5,6$ \\
\bottomrule
\end{tabular}
\end{center}
\end{Prop}

\begin{proof}
The extended Dynkin diagram gives the two pseudo-maximal possibilities $A_2$ and $A_1\times\widetilde A_1$. The corresponding local special $\aff$-value sets are listed in Table \ref{tab:pseudomax-a-values}. Subtracting these values from $|\Phi^+|=6$ gives the two rows.
\end{proof}

For type $G_2$, no further orbit distinction is needed: all nilpotent orbits have different dimensions \cite[\S~8.4]{CM}. Hence $V(\Ann(L(\lam)))$ is determined by $\gkd L(\lam)$.

\subsection{Type \texorpdfstring{$F_4$}{F4}}

\begin{Prop}\label{f4-nonint}
Let $L(\lambda)$ be a simple nonintegral highest weight module of $F_4$ such that $\Phi_{[\lambda]}^\vee$ is pseudo-maximal. Then the possible GK dimensions are listed below.
\begin{center}
\small
\renewcommand{\arraystretch}{1.12}
\setlength{\tabcolsep}{7pt}
\begin{tabular}{@{}ll@{}}
\toprule
$\Phi_{[\lambda]}$ & allowed $\gkd L(\lambda)$ \\
\midrule
$C_4$ & $8,15,[18,24]$ \\
$A_2\times\widetilde A_2$ & $18,[20,24]$ \\
$A_1\times\widetilde A_3$ & $17,18,[20,24]$ \\
$B_3\times\widetilde A_1$ & $14,15,[19,24]$ \\
\bottomrule
\end{tabular}
\end{center}
\end{Prop}

\begin{proof}
The extended Dynkin diagram gives the four pseudo-maximal types displayed in the table. Their local special $\aff$-value sets are listed in Table \ref{tab:pseudomax-a-values}. Since $|\Phi^+|=24$, applying
\[
\gkd L(\lambda)=24-\aff(w_\lambda)
\]
to those sets gives the four rows.
\end{proof}

The  remaining orbit distinctions  that require Sommers duality occur at GK dimensions $15,18$ and $21$ by the table in  \cite[p. 128]{CM}.

\begin{Thm}\label{f4-av}
Let $L(\lambda)$ be a simple nonintegral highest weight module of $F_4$ such that $\Phi_{[\lambda]}^\vee$ is pseudo-maximal. The cases requiring an explicit Sommers calculation are as follows.
\begin{enumerate}
\item If $\Phi(H)=\Phi_{[\lambda]}\simeq C_4$ and $\gkd L(\lambda)\in\{15,18,21\}$, then $H^\vee$ has type $B_4$,
\[
\pi_{w_\lambda}={\bf p}_1=p_C((\lambda')^-)^s,
\qquad
\pi_{w_\lambda}^\vee={\bf p}_1^\vee=p_B((\lambda')^-)^s=({\bf p}_1+[1,0,\ldots,0])_B,
\]
and
\[
d_{\rm LS}^{H^\vee}(\mathcal O(\pi_{w_\lambda}^\vee))=({\bf p}_1^\vee)^t_B={\rm sat}_{L^\vee}^{H^\vee}(\mathcal O_{L^\vee}).
\]
\item If $\Phi(H)\simeq A_2\times\widetilde A_2$ and $\gkd L(\lambda)\in\{18,21\}$, then $H^\vee$ has type $\widetilde A_2\times A_2$,
\[
\pi_{w_\lambda}^\vee=p_A(\lambda'|_{A_2})\times p_A(\lambda'|_{\widetilde A_2}),
\qquad
d_{\rm LS}^{H^\vee}(\mathcal O(\pi_{w_\lambda}^\vee))={\bf p}_1^t\times{\bf p}_2^t={\rm sat}_{L^\vee}^{H^\vee}(\mathcal O_{L^\vee}).
\]
\item If $\Phi(H)\simeq A_1\times\widetilde A_3$ and $\gkd L(\lambda)\in\{18,21\}$, then $H^\vee$ has type $\widetilde A_1\times A_3$,
\[
\pi_{w_\lambda}^\vee=p_A(\lambda'|_{A_1})\times p_A(\lambda'|_{\widetilde A_3}),
\qquad
d_{\rm LS}^{H^\vee}(\mathcal O(\pi_{w_\lambda}^\vee))={\bf p}_1^t\times{\bf p}_2^t={\rm sat}_{L^\vee}^{H^\vee}(\mathcal O_{L^\vee}).
\]
\item If $\Phi(H)\simeq B_3\times\widetilde A_1$ and $\gkd L(\lambda)\in\{15,21\}$, then $H^\vee$ has type $C_3\times A_1$,
\[
\pi_{w_\lambda}=p_B((\lambda'|_{B_3})^-)^s\times p_A(\lambda'|_{\widetilde A_1}),
\quad
\pi_{w_\lambda}^\vee=p_C((\lambda'|_{B_3})^-)^s\times p_A(\lambda'|_{\widetilde A_1}),
\]
and
\[
d_{\rm LS}^{H^\vee}(\mathcal O(\pi_{w_\lambda}^\vee))=({\bf p}_1^\vee)^t_C\times{\bf p}_2^t={\rm sat}_{L^\vee}^{H^\vee}(\mathcal O_{L^\vee}).
\]
\end{enumerate}
In every case,
\[
\mathcal O_{{\rm Ann}(L(\lambda))}=d_{\rm Som}^{G^\vee}(L^\vee,\mathcal O_{L^\vee}).
\]
\end{Thm}
\begin{proof}
We treat the $C_4$ and $B_3\times\widetilde A_1$ cases; the two products of type $A$ follow from the same argument.

For the $C_4$ case, $\Phi(H)=\Phi_{[\lambda]}\simeq C_4$ and $H^\vee$ has type $B_4$. Applying the H-algorithm of \cite{BMW} to $\lambda'$ gives two special partitions $\pi_{w_{\lam}}={\bf p}_{1}=p_C((\lam')^-)^s$ of type $C_4$ and  $\pi^{\vee}_{w_{\lam}}={\bf p}^{\vee}_{1}=p_B((\lam')^-)^s$ of type $B_4$, which both come from the Young tableau $P((\lam')^-)$. The partition $p_C((\lambda')^-)^s$ corresponds to the special bipartition
$\sigma:=\pi_{w_\lambda}=({\bf d},{\bf f})\in {\rm Irr}(W(H))$.
Then $\sigma^\vee=\tau_{H,H^\vee}(\sigma)=({\bf d},{\bf f})\in {\rm Irr}(W(H^\vee))$. The H-algorithm gives
\[d_{\rm LS}^{H^\vee}\!\left(\mathcal{O}_{\rm Spr}^{H^\vee}(\sigma^\vee)\right)
= \bigl(p_B((\lam')^-)^s\bigr)^t
= {\rm sat}_{L^\vee}^{H^\vee}(\mathcal{O}_{L^\vee}),\]
where $L^\vee\subseteq H^\vee\subseteq F_4^\vee$ and $\mathcal{O}_{L^\vee}$ is distinguished in $L^\vee$.
It follows from Theorem \ref{P:Som} that \[\mathcal{O}_{{\rm Ann}(L(\lambda))} = d_{\rm Som}^{F_4^\vee}(L^{\vee}, \mathcal{O}_{L^\vee}).\]

  For $\gkd L(\lambda)=15,18,$ or $21$, several nilpotent orbits have dimension $2\gkd L(\lambda)$.
If $\gkd L(\lambda)=15$, then $\aff(w_\lambda)=9$. There is only one special partition of type $C_4$ whose $\aff$-value is $9$: ${\bf p}=[2,2,1^4]$.
  The corresponding special partition of type $B_4$ is ${\bf p}^{\vee}=({\bf p}+[1,0,\cdots,0])_B=[3,1^6]$. Thus $({\bf p}^{\vee})^t=[7,1,1]={\rm sat}_{L^\vee}^{H^\vee}(\mathcal{O}_{L^\vee})$, where $L^\vee$ has type $B_3$ and $\mathcal{O}_{L^\vee}$ is its regular nilpotent orbit.
  Its saturation in the corresponding group $F_4^\vee$ is the Bala--Carter orbit $B_3$. Hence Theorem \ref{P:Som}, together with the LS duality in \cite[p.~440]{Ca85} or Table \ref{tabcharf4}, gives $V(\Ann(L(\lambda)))=\overline{\mathcal O(A_2)}$, since
\[
d_{\rm Som}^{F_4^\vee}(L^{\vee},\mathcal{O}_{L^\vee})
=d_{\rm BV}^{F_4^\vee}(B_3)
=\tau_{F_4^\vee,F_4}\!\left(d_{\rm LS}^{F_4^\vee}(B_3)\right)
=\tau_{F_4^\vee,F_4}(\widetilde A_2)=\mathcal O(A_2).
\]

The same calculation, with the corresponding special partitions, treats $\gkd L(\lambda)=18$ and $21$.

For the $B_3\times\widetilde A_1$ case, $H^\vee$ has type $C_3\times A_1$. Only $\gkd L(\lambda)=15$ or $21$ requires a separate orbit distinction, since $\aff(w_\lambda)\neq6$. If $\gkd L(\lambda)=21$, then $\aff(w_\lambda)=3$. Note that there are two special partitions of type $B_3\times {A}_1$ whose $\aff$-value is $3$: ${\bf p}=[3,3,1] \times [1,1]$ and ${\bf q}=[3,2,2] \times [2]$.
  The corresponding two special partitions of type $C_3\times A_1$ are:  ${\bf p}^{\vee}=[3,3] \times [1,1]$ and ${\bf q}^{\vee}=[2^3] \times [2]$. Thus $({\bf p}^{\vee})^t=[2^3]\times [2]={\rm sat}_{L^\vee}^{H^\vee}(\mathcal{O}_{L^\vee})$ for a pseudo-Levi datum of type $2A_1+\widetilde A_1$, while $({\bf q}^{\vee})^t=[3,3]\times [1,1]={\rm sat}_{L^\vee}^{H^\vee}(\mathcal{O}_{L^\vee})$ for a Levi subgroup $L^\vee$ of type $\widetilde A_2$ with regular orbit $\mathcal{O}_{L^\vee}$.
  Theorem \ref{P:Som}, the LS duality in \cite[p.~440]{Ca85}, and the Sommers duality in \cite[Table~VII]{Sommers} or Table \ref{tabcharf4} give $V(\Ann(L(\lambda)))=\overline{\mathcal O(B_3)}$ for $\pi_{w_\lambda}={\bf p}$ from the Sommers datum of type $2A_1+\widetilde A_1$. For $\pi_{w_\lambda}={\bf q}$, the saturation of the regular orbit of the Levi $L^\vee$ in the corresponding group $F_4^\vee$ is the Bala--Carter orbit $\widetilde A_2$, and hence
\[
d_{\rm Som}^{F_4^\vee}(L^{\vee},\mathcal{O}_{L^\vee})
=d_{\rm BV}^{F_4^\vee}(\widetilde A_2)
=\tau_{F_4^\vee,F_4}\!\left(d_{\rm LS}^{F_4^\vee}(\widetilde A_2)\right)
=\tau_{F_4^\vee,F_4}(B_3)=\mathcal O(C_3).
\]
Thus $V(\Ann(L(\lambda)))=\overline{\mathcal O(C_3)}$ in this case. 


The remaining case $\gkd L(\lambda)=15$ is obtained in the same way from the corresponding special partitions.

\end{proof}

\begin{ex}\label{ex:f4-annihilator}
Let $\mathfrak g=F_4$ and
\[
\lambda=\left(4,5,\frac32,\frac12\right).
\]
Then $\Phi(H)=\Phi_{[\lambda]}$ is a subsystem with simple roots $\{\ep_2,\alpha_4,\alpha_3,\alpha_2  \}$ and $\Phi_{[\lambda]}\simeq C_4$. Thus  $\Phi_{[\lambda]}^\vee \simeq B_4$. Suppose the simple root system of $C_4$ is $\Delta=\{\beta_1=\ep_1-\ep_2, \beta_2=\ep_2-\ep_3, \beta_3=\ep_3-\ep_4,\beta_4=2\ep_4\}$. We define a map $\phi: \Phi_{[\lambda]} \rightarrow C_4$ such that $\phi(\varepsilon_2)=\beta_1, \phi(\alpha_4)=\beta_2, \phi(\alpha_3)=\beta_3, \phi(\alpha_2)=\beta_4$. Then we have $\phi(\lam)=\lam'=(9,-1,2,1)$. 



By using the RS algorithm in \cite{BXX} and the H-algorithm in \cite{BMW}, we have  $\pi_{w_\lambda}=p_C((\lam')^-)^s=[2,2,1^4]$ and $\pi^{\vee}_{w_\lambda}=p_B((\lam')^-)^s=[3,1^6]$.   A simple computation gives that
$$d_{\rm LS}^{B_4}(\mathcal{O}(\pi^{\vee}_{w_\lambda}))=(p_B((\lam')^-)^s)^t= [7,1,1] = {\rm sat}_{B_3}^{B_4} (\mathcal{O}_{B_3}^{\rm reg}).$$
Thus the corresponding Sommers datum has pseudo-Levi type $B_3$ and regular distinguished orbit.

By Theorem \ref{P:Som} and the LS duality in \cite[p. 440]{Ca85} or Table \ref{tabcharf4}, $$\mathcal{O}_{{\rm Ann}(L(\lambda))} = d_{\rm Som}^{F_4^\vee}(L^{\vee}, \mathcal{O}_{L^\vee}) =\mathcal{O}(A_2).$$

In this example $\mathbf a(w_\lambda)=9$ by Proposition \ref{integral}, so $\gkd L(\lambda)=24-9=15$.
\end{ex}

\subsection{Type \texorpdfstring{$E_6$}{E6}}

\begin{Prop}\label{e6-non}
Let $L(\lambda)$ be a simple nonintegral highest weight module of $E_6$ such that $\Phi_{[\lambda]}^\vee$ is pseudo-maximal. Then the possible GK dimensions are listed below.
\begin{center}
\small
\renewcommand{\arraystretch}{1.12}
\setlength{\tabcolsep}{7pt}
\begin{tabular}{@{}ll@{}}
\toprule
$\Phi_{[\lambda]}$ & allowed $\gkd L(\lambda)$ \\
\midrule
$D_5$ & $16,23,24,26,[29,36]$ \\
$A_5\times A_1$ & $20,21,25,26,[28,36]$ \\
$A_2^3$ & $27,[29,36]$ \\
\bottomrule
\end{tabular}
\end{center}
\end{Prop}
\begin{proof}
The local special $\aff$-value sets for the three pseudo-maximal subsystem types are listed in Table \ref{tab:pseudomax-a-values}. Since $|\Phi^+|=36$, the formula $\gkd L(\lambda)=36-\aff(w_\lambda)$ gives the table.
\end{proof}

The  remaining orbit distinctions  that require Sommers duality occur at GK dimensions $30$ and $32$ by the table in \cite[p.129]{CM}.

\begin{Thm}\label{e6-av}
Let $L(\lambda)$ be a simple nonintegral highest weight module of $E_6$ such that $\Phi_{[\lambda]}^\vee$ is pseudo-maximal. For $\gkd L(\lambda)\in\{30,32\}$, the cases requiring an explicit Sommers calculation are as follows.
\begin{enumerate}
\item Suppose $\Phi(H)\simeq D_5$. Then $H^\vee$ has type $D_5$ and
\[
\pi_{w_\lambda}^\vee={\bf p}_1=p_D((\lambda')^-)^s,
\qquad
d_{\rm LS}^{H^\vee}(\mathcal O(\pi_{w_\lambda}^\vee))
=({\bf p}_1^t)_D
={\rm sat}_{L^\vee}^{H^\vee}(\mathcal O_{L^\vee}).
\]
\item Suppose $\Phi(H)\simeq A_5\times A_1$. Then $H^\vee$ has type $A_5\times A_1$ and
\[
\pi_{w_\lambda}^\vee={\bf p}_1\times{\bf p}_2
=p_A(\lambda'|_{A_5})\times p_A(\lambda'|_{A_1}),
\]
with
\[
d_{\rm LS}^{H^\vee}(\mathcal O(\pi_{w_\lambda}^\vee))
={\bf p}_1^t\times{\bf p}_2^t
={\rm sat}_{L^\vee}^{H^\vee}(\mathcal O_{L^\vee}^{\rm reg}).
\]
\item Suppose $\Phi(H)\simeq A_2\times A_2 \times A_2$. Writing $A_{2(i)}$ for the three factors, one has $H^\vee$ of type $A_{2(1)}\times A_{2(2)} \times A_{2(3)}$ and
\[
\pi_{w_\lambda}^\vee={\bf p}_1\times{\bf p}_2\times{\bf p}_3
=\prod_{i=1}^3p_A(\lambda'|_{A_{2(i)}}),
\]
with
\[
d_{\rm LS}^{H^\vee}(\mathcal O(\pi_{w_\lambda}^\vee))
={\bf p}_1^t\times{\bf p}_2^t\times{\bf p}_3^t
={\rm sat}_{L^\vee}^{H^\vee}(\mathcal O_{L^\vee}^{\rm reg}).
\]
\end{enumerate}
Here $\mathcal O_{L^\vee}$ is distinguished (and is regular in the two products of type $A$). In every case,
\[
\mathcal O_{{\rm Ann}(L(\lambda))}=d_{\rm Som}^{G^\vee}(L^\vee,\mathcal O_{L^\vee}).
\]
\end{Thm}

\begin{proof}
For each pseudo-maximal subsystem, the special representation of $W(H)$ is the product of the special representations attached to the irreducible factors of the transported weight. For the classical factors, the partition algorithms of \cite{BXX,BMW} give the representations displayed above. Lusztig--Spaltenstein duality is then computed factorwise: by transpose in type $A$ and by the type-$D$ transpose--collapse operation in the $D_5$ case. The resulting orbit is the saturation of a distinguished orbit in a Levi subgroup $L^\vee\subseteq H^\vee$, as prescribed by the Bala--Carter classification.

Applying Theorem \ref{P:Som} with the group $E_6^\vee$ gives
\[
\mathcal O_{{\rm Ann}(L(\lambda))}
=d_{\rm Som}^{E_6^\vee}(L^\vee,\mathcal O_{L^\vee}).
\]
The orbit dimension tables in \cite[\S~8.4]{CM} show that the only allowed dimensions in these pseudo-maximal cases for which the dimension alone does not determine the corresponding nilpotent orbit are the ones corresponding to $\gkd L(\lambda)=30$ and $32$. This gives the stated list.
\end{proof}

\subsection{Type \texorpdfstring{$E_7$}{E7}}

\begin{Prop}
Let $L(\lambda)$ be a simple nonintegral highest weight module of $E_7$ such that $\Phi_{[\lambda]}^\vee$ is pseudo-maximal. Then the possible GK dimensions are listed below.
\begin{center}
\small
\renewcommand{\arraystretch}{1.12}
\setlength{\tabcolsep}{7pt}
\begin{tabular}{@{}ll@{}}
\toprule
$\Phi_{[\lambda]}$ & allowed $\gkd L(\lambda)$ \\
\midrule
$D_6\times A_1$ & $32,33,41,42,43,[46,63]$ \\
$A_7$ & $35,42,47,48,[50,63]$ \\
$A_5\times A_2$ & $45,47,48,50,[52,63]$ \\
$A_3\times A_3\times A_1$ & $50,51,[53,63]$ \\
$E_6$ & $27,38,43,48,50,51,52,53,[56,63]$ \\
\bottomrule
\end{tabular}
\end{center}
\end{Prop}

\begin{proof}
The local special $\aff$-value sets are listed in Table \ref{tab:pseudomax-a-values}. Since $|\Phi^+|=63$, the formula $\gkd L(\lambda)=63-\aff(w_\lambda)$ gives the table.

\end{proof}

The  remaining orbit distinctions  that require Sommers duality occur when $\gkd L(\lambda)\in\{42,47,48,50,51,53,54,55,56,57,59,60\}$ by the tables in \cite[p.130-p.131]{CM}.

\begin{Thm}\label{e7-av}
Let $L(\lambda)$ be a simple nonintegral highest weight module of $E_7$ such that $\Phi_{[\lambda]}^\vee$ is pseudo-maximal. The cases requiring an explicit Sommers calculation are as follows.
\begin{enumerate}
\item If $\Phi(H)\simeq A_7$ and $\gkd L(\lambda)\in\{42,47,48,50,51,53,54,55,56,57,59,60\}$, then
\[
\pi_{w_\lambda}^\vee={\bf p}_1=p_A(\lambda'),
\qquad
d_{\rm LS}^{H^\vee}(\mathcal O(\pi_{w_\lambda}^\vee))
={\bf p}_1^t
={\rm sat}_{L^\vee}^{H^\vee}(\mathcal O_{L^\vee}^{\rm reg}).
\]
\item If $\Phi(H)\simeq A_5\times A_2$ and $\gkd L(\lambda)\in\{47,48,50,53,54,55,56,57,59,60\}$, then
\[
\pi_{w_\lambda}^\vee={\bf p}_1\times{\bf p}_2
=p_A(\lambda'|_{A_5})\times p_A(\lambda'|_{A_2}),
\]
and
\[
d_{\rm LS}^{H^\vee}(\mathcal O(\pi_{w_\lambda}^\vee))
={\bf p}_1^t\times{\bf p}_2^t
={\rm sat}_{L^\vee}^{H^\vee}(\mathcal O_{L^\vee}^{\rm reg}).
\]
\item If $\Phi(H)\simeq A_3\times A_3\times A_1$ and $\gkd L(\lambda)\in\{50,51,53,54,55,56,57,59,60\}$, then
\[
\pi_{w_\lambda}^\vee={\bf p}_1\times{\bf p}_2\times{\bf p}_3
=p_A(\lambda'|_{A_3(1)})\times p_A(\lambda'|_{A_3(2)})\times p_A(\lambda'|_{A_1}),
\]
and
\[
d_{\rm LS}^{H^\vee}(\mathcal O(\pi_{w_\lambda}^\vee))
={\bf p}_1^t\times{\bf p}_2^t\times{\bf p}_3^t
={\rm sat}_{L^\vee}^{H^\vee}(\mathcal O_{L^\vee}^{\rm reg}).
\]
\item If $\Phi(H)\simeq D_6\times A_1$ and $\gkd L(\lambda)\in\{42,47,48,50,51,53,54,55,56,57,59,60\}$, then
\[
{{\pi_{w_\lambda}^\vee=\pi_1^\vee\boxtimes\pi_2^\vee},\quad
{\bf p}_1=p_D((\lambda'|_{D_6})^-)^s,\quad
{\bf p}_2=p_A(\lambda'|_{A_1}),}
\]
and
\[
d_{\rm LS}^{H^\vee}(\mathcal O(\pi_{w_\lambda}^\vee))
{=\mathcal Q_1\times\mathcal Q_2}
={\rm sat}_{L^\vee}^{H^\vee}(\mathcal O_{L^\vee}),
\]
{where $\pi_i^\vee$ and $\mathcal Q_i$ are as in Section \ref{sec:practical-sommers}.}
The underlying partitions of $\mathcal Q_1$ and $\mathcal Q_2$ are $({\bf p}_1^t)_D$ and ${\bf p}_2^t$. If ${\bf p}_1$ is very even, determine its local $I/II$ label from the transported weight by the BMW criterion, and pull back the corresponding roots $J_I$ or $J_{II}$ as in Section \ref{sec:local-branches}. 
\item If $\Phi(H)\simeq E_6$ and $\gkd L(\lambda)\in\{48,50,51,53,56,57,59,60\}$, write $\chi_1:=\pi_{w_\lambda}^\vee$. If the Springer orbit $\mathcal O(\chi_1)$ has Bala--Carter label $L_1$, then
\[
d_{\rm LS}^{H^\vee}(\mathcal O(\pi_{w_\lambda}^\vee))
=d_{\rm LS}^{E_6}(\mathcal O(L_1))
={\rm sat}_{L^\vee}^{H^\vee}(\mathcal O_{L^\vee}).
\]
\end{enumerate}
When the embedded Sommers datum has a primed or double-primed
label, the label is determined by the criteria in
Sections~\ref{dis} and~\ref{sec:local-branches}.
In every case $\mathcal O_{L^\vee}$ is distinguished, and
\[
\mathcal O_{{\rm Ann}(L(\lambda))}=d_{\rm Som}^{E_7^\vee}(L^\vee,\mathcal O_{L^\vee}).
\]
\end{Thm}

\begin{proof}
The proof is similar to that of Theorem \ref{e6-av}; only the cases
$\Phi(H)\simeq D_6\times A_1$ and $\Phi(H)\simeq E_6$
require additional attention.

Recall that for a $D_{6}$-type factor, a very even partition corresponds to two orbits, labelled $I$ and $II$; our convention is recorded in Table \ref{primes}. The label is determined from the transported weight by \cite[Lem.~7.8, Rem.~7.9, Thm.~7.10]{BMW}, as in Section \ref{sec:local-branches}. 

For the exceptional $E_6$ factor, the special representation $\chi_1$ and its Springer orbit are read from the Springer tables in \cite[Table~C.4]{GP} and \cite[\S~13.3]{Ca85}. Then applying $d_{\rm LS}^{E_6}$ gives the displayed Bala--Carter datum.

\end{proof}

\subsection{Type \texorpdfstring{$E_8$}{E8}}

\begin{Prop}
Let $L(\lambda)$ be a simple nonintegral highest weight module of $E_8$ such that $\Phi_{[\lambda]}^\vee$ is pseudo-maximal. Then the possible GK dimensions are listed below.
\begin{center}
\small
\renewcommand{\arraystretch}{1.12}
\setlength{\tabcolsep}{7pt}
\begin{tabular}{@{}ll@{}}
\toprule
$\Phi_{[\lambda]}$ & allowed $\gkd L(\lambda)$ \\
\midrule
$D_8$ & $64,77,78,86,89,90,91,92,96,97,98,99,100,[102,120]$ \\
$A_8$ & $84,92,98,99,102,104,105,[107,120]$ \\
$E_7\times A_1$ & $56,57,73,74,82,83,84,89,90,94,95,97,98,99,100,[103,120]$ \\
$A_7\times A_1$ & $91,92,98,99,[103,120]$ \\
$A_5\times A_2\times A_1$ & $[101,120]$ \\
$A_4\times A_4$ & $100,104,[106,120]$ \\
$D_5\times A_3$ & $94,[97,102],[104,120]$ \\
$E_6\times A_2$ & $81,83,84,92,94,95,97,99,100,102,[104,120]$ \\
\bottomrule
\end{tabular}
\end{center}
\end{Prop}

\begin{proof}
The local special $\aff$-value sets are listed in Table \ref{tab:pseudomax-a-values}. Since $|\Phi^+|=120$, the formula $\gkd L(\lambda)=120-\aff(w_\lambda)$ gives the table.

\end{proof}

Similarly, the remaining orbit distinctions  that require Sommers duality  occur when $$\gkd L(\lambda)\in\{84,92,94,98,99,100,101,102,116\}\cup[104,114]$$ by the tables in \cite[p.132-p.134]{CM}.

\begin{Thm}\label{e8-av}
Let $L(\lambda)$ be a simple nonintegral highest weight module of $E_8$ such that $\Phi_{[\lambda]}^\vee$ is pseudo-maximal. The cases requiring an explicit Sommers calculation are as follows.
\begin{enumerate}
\item If $\Phi(H)\simeq A_8$ and $\gkd L(\lambda)\in\{84,92,98,99,102,104,105,116\}\cup[107,114]$, then
\[
\pi_{w_\lambda}^\vee={\bf p}_1=p_A(\lambda'),
\qquad
d_{\rm LS}^{H^\vee}(\mathcal O(\pi_{w_\lambda}^\vee))
={\bf p}_1^t
={\rm sat}_{L^\vee}^{H^\vee}(\mathcal O_{L^\vee}^{\rm reg}).
\]
\item If $\Phi(H)\simeq A_7\times A_1$ and $\gkd L(\lambda)\in\{92,98,99,116\}\cup[104,114]$, then
\[
\pi_{w_\lambda}^\vee={\bf p}_1\times{\bf p}_2
=p_A(\lambda'|_{A_7})\times p_A(\lambda'|_{A_1}),
\]
and
\[
d_{\rm LS}^{H^\vee}(\mathcal O(\pi_{w_\lambda}^\vee))
={\bf p}_1^t\times{\bf p}_2^t
={\rm sat}_{L^\vee}^{H^\vee}(\mathcal O_{L^\vee}^{\rm reg}).
\]
\item If $\Phi(H)\simeq A_5\times A_2\times A_1$ and $\gkd L(\lambda)\in\{101,102,116\}\cup[104,114]$, then
\[
\pi_{w_\lambda}^\vee={\bf p}_1\times{\bf p}_2\times{\bf p}_3
=p_A(\lambda'|_{A_5})\times p_A(\lambda'|_{A_2})\times p_A(\lambda'|_{A_1}),
\]
and
\[
d_{\rm LS}^{H^\vee}(\mathcal O(\pi_{w_\lambda}^\vee))
={\bf p}_1^t\times{\bf p}_2^t\times{\bf p}_3^t
={\rm sat}_{L^\vee}^{H^\vee}(\mathcal O_{L^\vee}^{\rm reg}).
\]
\item If $\Phi(H)\simeq A_4\times A_4$ and $\gkd L(\lambda)\in\{100,104,116\}\cup[106,114]$, then
\[
\pi_{w_\lambda}^\vee={\bf p}_1\times{\bf p}_2
=p_A(\lambda'|_{A_4(1)})\times p_A(\lambda'|_{A_4(2)}),
\]
and
\[
d_{\rm LS}^{H^\vee}(\mathcal O(\pi_{w_\lambda}^\vee))
={\bf p}_1^t\times{\bf p}_2^t
={\rm sat}_{L^\vee}^{H^\vee}(\mathcal O_{L^\vee}^{\rm reg}).
\]
\item If $\Phi(H)\simeq D_8$ and $\gkd L(\lambda)\in\{92,98,99,100,102,116\}\cup[104,114]$, then
\[
{{\pi_{w_\lambda}^\vee=\pi_1^\vee},\quad {\bf p}_1=p_D((\lambda')^-)^s,}
\]
and
\[
d_{\rm LS}^{H^\vee}(\mathcal O(\pi_{w_\lambda}^\vee))
{=\mathcal Q_1}
={\rm sat}_{L^\vee}^{H^\vee}(\mathcal O_{L^\vee}),
\]
{where $\pi_1^\vee$ and $\mathcal Q_1$ are as in Section \ref{sec:practical-sommers}.}
The underlying partition of $\mathcal Q_1$ is $({\bf p}_1^t)_D$. If ${\bf p}_1$ is very even, determine its local $I/II$ label from the transported weight by the BMW criterion, and pull back the corresponding roots $J_I$ or $J_{II}$ as in Section \ref{sec:local-branches}. 
\item If $\Phi(H)\simeq D_5\times A_3$ and $\gkd L(\lambda)\in\{94,98,99,100,101,102,116\}\cup[104,114]$, then
\[
\pi_{w_\lambda}^\vee={\bf p}_1\times{\bf p}_2
=p_D((\lambda'|_{D_5})^-)^s\times p_A(\lambda'|_{A_3}),
\]
and
\[
d_{\rm LS}^{H^\vee}(\mathcal O(\pi_{w_\lambda}^\vee))
=({\bf p}_1^t)_D\times{\bf p}_2^t
={\rm sat}_{L^\vee}^{H^\vee}(\mathcal O_{L^\vee}).
\]
\item If $\Phi(H)\simeq E_6\times A_2$ and $\gkd L(\lambda)\in\{84,92,94,99,100,102,116\}\cup[104,114]$, write
\[
\pi_{w_\lambda}^\vee=\chi_1\times{\bf p}_2,
\qquad {\bf p}_2=p_A(\lambda'|_{A_2}),
\]
where $\chi_1$ is the special representation attached to the $E_6$ factor. If $\mathcal O(\chi_1)$ has Bala--Carter label $L_1$, then
\[
d_{\rm LS}^{H^\vee}(\mathcal O(\pi_{w_\lambda}^\vee))
=d_{\rm LS}^{E_6}(\mathcal O(L_1))\times{\bf p}_2^t
={\rm sat}_{L^\vee}^{H^\vee}(\mathcal O_{L^\vee}).
\]
\item If $\Phi(H)\simeq E_7\times A_1$ and $\gkd L(\lambda)\in\{84,94,98,99,100,116\}\cup[104,114]$, write
\[
\pi_{w_\lambda}^\vee=\chi_1\times{\bf p}_2,
\qquad {\bf p}_2=p_A(\lambda'|_{A_1}),
\]
where $\chi_1$ is the special representation attached to the $E_7$ factor. If $\mathcal O(\chi_1)$ has Bala--Carter label $L_1$, then
\[
d_{\rm LS}^{H^\vee}(\mathcal O(\pi_{w_\lambda}^\vee))
=d_{\rm LS}^{E_7}(\mathcal O(L_1))\times{\bf p}_2^t
={\rm sat}_{L^\vee}^{H^\vee}(\mathcal O_{L^\vee}).
\]
\end{enumerate}
When the embedded Sommers datum has a primed or double-primed
label, the label is determined by the criteria in
Sections~\ref{dis} and~\ref{sec:local-branches}.
In every case $\mathcal O_{L^\vee}$ is distinguished (and is regular in the products of type $A$), and
\[
\mathcal O_{{\rm Ann}(L(\lambda))}=d_{\rm Som}^{E_8^\vee}(L^\vee,\mathcal O_{L^\vee}).
\]
\end{Thm}

\begin{proof}

For the $E_6$ and $E_7$ factors, the special representations and their Springer orbits are read from {\cite[Tables~C.4--C.5]{GP}} and \cite[\S~13.3]{Ca85}.

The remaining argument is analogous to that in the proof of Theorem~\ref{e7-av}.
\end{proof}

\begin{ex}\label{ex:e8-annihilator}
Let $\mathfrak g=E_8$ and
\[
\lambda=\left(1,1,1,1,1,1,\frac12,\frac52\right).
\]
Recall that, in the realization of $E_8$ used above,
\[
\alpha_1=\frac12(\ep_1-\ep_2-\ep_3-\ep_4-\ep_5-\ep_6-\ep_7+\ep_8),
\quad
\alpha_2=\ep_1+\ep_2,
\]
\[
\alpha_3=-\ep_1+\ep_2,\quad
\alpha_4=-\ep_2+\ep_3,\quad
\alpha_5=-\ep_3+\ep_4,\quad
\alpha_6=-\ep_4+\ep_5,\quad
\alpha_7=-\ep_5+\ep_6,
\]
and the highest root is $\beta=\ep_7+\ep_8$. A direct application of Lemma \ref{integral-simple-system} gives
\[
\Delta_{[\lambda]}
=
\underbrace{\{\alpha_1,\ldots,\alpha_7\}}_{E_7\text{-factor}}
\sqcup
\underbrace{\{\beta\}}_{A_1\text{-factor}}.
\]
Thus $\Phi_{[\lambda]}\simeq E_7\times A_1$. If $\gamma_1=\ep_1-\ep_2$ denotes the simple root of an independent standard $A_1$ factor, the correspondence
\[
\phi(\alpha_i)=\alpha_i\quad(1\le i\le7),
\qquad
\phi(\beta)=\gamma_1
\]
defines an isomorphism preserving positive systems
\[
\phi:\Phi_{[\lambda]}\xrightarrow{\sim}E_7\times A_1.
\]
The transported weights are therefore
\[
\lambda'_{(1)}=(1,1,1,1,1,1,-1,1),
\qquad
\lambda'_{(2)}=\left(\frac32,-\frac32\right).
\]
For the $E_7$ factor, the fundamental weight coordinates are
\[
\lambda'_{(1)}=-\omega_1+2\omega_2=[-1,2,0,0,0,0,0].
\]
Positive index reduction terminates at
\[
\mu=-\omega_2-\omega_7=[0,-1,0,0,0,0,-1],
\]
and gives $\lambda'_{(1)}=w_1\mu$, where, in the PyCox notation of Subsection \ref{subsec:pycox},
\[
\begin{aligned}
w_1={}&[1,3,4,5,6,2,3,4,5,1,3,4,2,3,1,0,2,3,4,5,6,\\
&1,3,4,5,2,3,4,1,3,2,0,2,3,4,5,1,3,4,2,3,1].
\end{aligned}
\]
PyCox yields $\mathbf a(w_1)=6$ and the special representation $\chi_1=105_b$ attached to its two-sided cell. By the Springer tables in \cite[Table~C.5]{GP} and \cite[\S~13.3]{Ca85}, its Springer orbit is $A_6$. For the $A_1$ factor, $w_2=s_1$, $\mathbf a(w_2)=1$, and the corresponding special partition is $[1,1]$. Hence, under
\[
W_{[\lambda]}\simeq W(E_7)\times W(A_1),
\]
one has $w_\lambda=(w_1,w_2)$ and
\[
\pi_{w_\lambda}^{\vee}=\chi_1\times[1,1],
\quad
\mathbf a(w_\lambda)=6+1=7.
\]
Since $d_{\rm LS}^{E_7}(A_6)=A_2+3A_1$ (see, for example, \cite[Table~IX]{Sommers}), we obtain
\[
d_{\rm LS}^{H^\vee}\!\left(\mathcal O(\pi_{w_\lambda}^{\vee})\right)
=(A_2+3A_1)\times[2]
={\rm sat}_{L^\vee}^{H^\vee}(\mathcal O_{L^\vee}^{\rm reg}),
\]
where $L^\vee$ has the root system of type $A_2+4A_1$ and $\mathcal O_{L^\vee}^{\rm reg}$ is its regular nilpotent orbit. Theorem \ref{P:Som} and \cite[Table~X]{Sommers} therefore give
\[
\mathcal O_{{\rm Ann}(L(\lambda))}=\mathcal O(D_7),
\quad
V(\Ann(L(\lambda)))=\overline{\mathcal O(D_7)}.
\]
Finally,
\[
\gkd L(\lambda)=120-\mathbf a(w_\lambda)=120-7=113.
\]
\end{ex}

\subsection{Further nonintegral examples}\label{sec:further-nonintegral}

The following two examples lie outside the pseudo-maximal cases treated above.  Both use the procedure of Subsection \ref{sec:local-branches} to identify the embedded pseudo-Levi class.

\begin{ex}\label{e7-ex}
Let $\mathfrak g=E_7$ and
\[
\lambda=\frac14(0,1,1,2,3,3,-7,7).
\]
The integral root system $\Phi_{[\lambda]} \simeq D_4 \times A_1 \times A_1$ has simple roots
\[
\left\{\alpha:=\frac{1}{2}(-e_1-e_2-e_3+e_4-e_5-e_6-e_7+e_8),
 e_2+e_5, e_6-e_5, e_3-e_2\right\}
\]
for $D_4$ and $\{\beta:=\frac{1}{2}(e_1+e_2+e_3-e_4-e_5-e_6-e_7+e_8)\}$ for the first $A_1$ and $\{\gamma:=\frac{1}{2}(e_1-e_2-e_3-e_4+e_5+e_6-e_7+e_8)\}$  for the second $A_1$. 

The transported restrictions determined by Section \ref{nonint} are
\[
\lambda'_{(1)}=(2,1,0,0),\qquad
\lambda'_{(2)}=\left(\frac12,-\frac12\right),\qquad
\lambda'_{(3)}=(1,-1),
\]
which are integral weights for the corresponding factors.

Applying the RS algorithm of \cite{BXX} and the H-algorithm of \cite{BMW} gives 
\[
{\bf p}_{1}=p_{D}((\lambda'_{(1)})^-)^s=[3,1^5],\qquad
{\bf p}_{2}=p_{A}(\lambda'_{(2)})=[1,1],\qquad
{\bf p}_{3}=p_{A}(\lambda'_{(3)})=[1,1].
\]
 
Thus
\[
\pi^{\vee}_{w_{\lam}}
={\bf p}_{1}\times {\bf p}_{2}\times {\bf p}_{3}
=[3,1^5]\times [1,1]\times [1,1].
\] 

Its Lusztig--Spaltenstein dual is
\begin{align*}
    d_{\rm LS}^{H^\vee} (\mathcal{O}(\pi_{w_{\lam}}^\vee))&= ({\bf p}_{1}^t)_D\times {\bf p}^t_{2}\times {\bf p}^t_{3} \\
    &=[5,1^3] \times [2]\times [2]\\
    &= {\rm sat}_{L^\vee}^{H^\vee} (\mathcal{O}_{L^\vee}).
\end{align*}
Here $L^\vee$ has type
\[
D_3+A_1+A_1\simeq A_3+2A_1,
\]
and $\mathcal O_{L^\vee}$ is the regular nilpotent orbit of $L^\vee$.
A neutral element for this orbit is
\[
h=3(e_3-e_2)+4(e_2+e_5)+3(e_6-e_5)+\beta+\gamma
=e_1+e_2+3e_3-e_4+e_5+3e_6-e_7+e_8.
\]
After identifying the  Cartan subalgebra with its dual via the invariant form, applying Lemma \ref{find-w-lambda} to $h$ and changing sign yields a dominant element $h^*$ with weighted Dynkin diagram
\[
(\alpha_i(h^*))_{1\leq i\leq7}=(0,0,2,0,0,0,0)
\]
by \cite[\S~3.5]{CM}. This identifies the saturation in $E_7$ as $D_4(a_1)$ \cite[p.~130]{CM}, and hence the embedded pseudo-Levi class is $L^\vee=(A_3+2A_1)'$ by \cite[p.~559]{som-98} or Table \ref{tabchare7}. Theorem \ref{P:Som} and \cite[Table~IX]{Sommers} then give
\[
\mathcal O_{{\rm Ann}(L(\lambda))}=\mathcal O(D_6(a_2)).
\]
Finally, the ordinary RS shape for the $D_4$ factor and its even row counts are
\[
p((\lambda'_{(1)})^-)=[4,1^4],\qquad
p((\lambda'_{(1)})^-)^{\rm ev}=(2,0,1,0,1).
\]
By Proposition \ref{integral},
\[
\aff_{D_4}(w_{\lambda'_{(1)}})
=0\cdot2+2\cdot1+4\cdot1=6.
\]
For each $A_1$ factor, $p(\lambda'_{(j)})=[1,1]$, so
\[
\aff_{A_1}(w_{\lambda'_{(j)}})=0\cdot1+1\cdot1=1,
\qquad j=2,3.
\]
Since $|\Phi_{E_7}^+|=63$, Theorem \ref{T:main1} and Proposition \ref{pr:main1} give
\[
\begin{aligned}
\aff_{[\lambda]}(w_\lambda)&=6+1+1=8,\\
\gkd L(\lambda)&=|\Phi_{E_7}^+|-\aff_{[\lambda]}(w_\lambda)
=63-8=55.
\end{aligned}
\]

\end{ex}

\begin{ex}\label{ex:e8-d6-roots}
Let $\mathfrak g=E_8$ and
\[
\lambda=\frac14(\omega_2+\omega_7+2\omega_8)
=\frac18(1,1,1,1,1,3,7,13).
\]
The integral root system $\Phi_{[\lambda]} \simeq D_6 \times A_1$ has simple roots
\[
\begin{aligned}
r_1&=\tfrac12(e_1-e_2-e_3-e_4-e_5-e_6-e_7+e_8),& r_2&=e_1+e_7,\\
r_3&=-e_1+e_2,&r_4&=-e_2+e_3,\\
r_5&=-e_3+e_4,&r_6&=-e_4+e_5
\end{aligned}
\]
for $D_6$ and $\gamma=e_6+e_8$ for $A_1$. Let $\beta_1,\ldots,\beta_6$ be the standard $D_6$ simple roots, with $\beta_5=\ep_5-\ep_6$ and $\beta_6=\ep_5+\ep_6$. The correspondence
\[
(r_6,r_5,r_4,r_3,r_1,r_2)\longmapsto
(\beta_1,\beta_2,\beta_3,\beta_4,\beta_5,\beta_6)
\]
gives the transported weights
\[
\nu:=\lambda'_{(1)}=(\tfrac12,\tfrac12,\tfrac12,\tfrac12,\tfrac12,\tfrac12),
\qquad\lambda'_{(2)}=(1,-1).
\]
The RS shapes are
\[
p(\nu^-)=[6,6],\qquad p(\lambda'_{(2)})=[1,1].
\]
The two rows of $p(\nu^-)$ each contain three even boxes. Hence Proposition \ref{integral} gives
\[
\begin{aligned}
\aff_{D_6}(w_\nu)&=0\cdot3+1\cdot3=3,\\
\aff_{A_1}(w_{\lambda'_{(2)}})&=0\cdot1+1\cdot1=1.
\end{aligned}
\]
By Theorem \ref{T:main1} and Proposition \ref{pr:main1},
\[
\begin{aligned}
\aff_{[\lambda]}(w_\lambda)&=3+1=4,\\
\gkd L(\lambda)&=|\Phi_{E_8}^+|-\aff_{[\lambda]}(w_\lambda)
=120-4=116.
\end{aligned}
\]

We determine the local very even branch by the BMW criterion. Put $s_i=s_{\beta_i}$ and represent $w$ in the standard action by $w(\ep_i)=\operatorname{sgn}(w(i))\ep_{|w(i)|}$. From \cite[Lem. 7.6]{BMW}, the minimal length element taking the antidominant weight to $\nu$ is
\[
w=w_{\nu}=(-6,-5,-4,-3,-2,-1)=w_0w_{\{\beta_1,\ldots,\beta_5\}}.
\]

The domino insertion of $\widehat w$ in \cite{Garfinkle1} gives
\[
T_L(\widehat w)=
\begin{tikzpicture}[scale=0.42,baseline=1.8ex]
\foreach \j in {1,...,6}{
  \draw (\j-1,0) rectangle (\j,2);
  \node at (\j-0.5,1) {$\j$};
}
\end{tikzpicture}.
\]
Each column is one vertical domino. 
Thus there are six vertical dominoes, and the tableau has the special shape $[6,6]$. Since $6\equiv2\pmod4$, \cite[Rem.~7.9, Thm.~7.10]{BMW} gives
\[
\mathcal O_{\rm Spr}^{D_6}(\pi_w)=\mathcal O_{[6,6]}^{II}
\]
in the convention of Table \ref{primes}.

For the transposed partition $[2^6]$, the two candidate $3A_1$ root sets are
\[
J_I=\{\beta_1,\beta_3,\beta_5\},\qquad
J_{II}=\{\beta_1,\beta_3,\beta_6\}.
\]
Their longest elements, written in the BMW convention, are
\[
\widehat w_{J_I}=(2,1,4,3,6,5),\qquad
\widehat w_{J_{II}}=(-2,-1,4,3,6,5).
\]
These are precisely $w_I$ and $w_{II}$ of \cite[Lem.~7.8]{BMW} for $[6,6]$. Hence $w\sim_{LR}w_{J_{II}}$ and
\[
\pi_w=\pi_{w_{J_{II}}}
=j_{W_{J_{II}}}^{W(D_6)}(\varepsilon_{W_{J_{II}}}).
\]
By Springer induction, its Springer orbit is induced from zero in this $3A_1$ Levi. The duality identity in the proof of Theorem \ref{P:Som} gives
\[
d_{\rm LS}^{D_6}(\mathcal O_{\rm Spr}^{D_6}(\pi_w))
={\rm sat}_{L_{J_{II}}}^{D_6}(\mathcal O_{L_{J_{II}}}^{\rm reg}),
\]
with underlying partition $[2^6]$. Pulling back $J_{II}$ gives
\[
\{\beta_1,\beta_3,\beta_6\}\longmapsto\{r_6,r_4,r_2\}.
\]
Together with $\gamma$, these roots determine the Sommers datum whose pseudo-Levi subgroup $L^\vee$ has type
\[
(3A_1)+A_1\simeq 4A_1,
\]
and $\mathcal O_{L^\vee}$ is the regular nilpotent orbit of $L^\vee$.
A neutral element for the orbit $3A_1+A_1$ is
\[
{h=r_6+r_4+r_2+\gamma
=(e_5-e_4)+(e_3-e_2)+(e_1+e_7)+(e_6+e_8).}
\]
After identifying the  Cartan subalgebra with its dual via the invariant form, applying Lemma \ref{find-w-lambda} to $h$ and changing sign yields a dominant $h^*$ with weighted Dynkin diagram
\[
(\alpha_i(h^*))_{1\leq i\leq8}=(0,0,0,0,0,0,0,2)
\]
by \cite[\S~3.5]{CM}. This identifies the saturation in $E_8$ as $A_2$ {\cite[\S~8.4]{CM}}, so the embedded pseudo-Levi class is $L^\vee=(4A_1)''$ by {\cite{som-98} or Table \ref{tabchare8}}.  Theorem \ref{P:Som} and \cite[Table~X]{Sommers} then give
\[
\mathcal O_{{\rm Ann}(L(\lambda))}=\mathcal O(E_7).
\]
\end{ex}

\appendix
\section{Computational reference data}\label{app:tables}\label{app:cell-reps}

For brevity, in the main text we write  $x_d^{E_n}$ (and $x_{14}^{F_4}$) for selected representatives of two-sided cells that do not have a convenient description by a parabolic longest element. We record these representatives here using the same Geck notation as in Section \ref{a-value-compu}; they are exactly the representatives suppressed from the classification statements in Section \ref{sec:E-cases}.

\medskip
\noindent\textit{Type $F_4$.}
\[
\begin{aligned}
x_{14}^{F_4}&=[1,2,1,2,3,2,1,0,2,1,3,2,1,0].\quad \quad\quad\quad\quad\quad \quad\quad\quad\\
\end{aligned}
\]

\medskip
\noindent\textit{Type $E_6$.}
\[
\begin{aligned}
x_{23}^{E_6}&=[3,1,4,3,2,5,4,3,1,5,4,3,2,5,4,3,5,4,5].\quad\quad\quad\quad
\end{aligned}
\]

\medskip
\noindent\textit{Type $E_7$.}
{\allowdisplaybreaks[1]
\begin{align*}
x_{26}^{E_7}&=[0,1,2,0,3,1,2,0,3,2,4,3,1,2,0,3,2,4,3,1,5,4,3,1,\\
&\phantom{=[}2,0,3,2,4,3,1,5,4,3,2,0,6,5,4,3,1,2,3,4,5,6]\\[2pt]
x_{38}^{E_7}&=[1,2,3,1,2,0,3,4,3,1,2,0,3,2,4,3,5,4,3,1,2,0,3,2,\\
&\phantom{=[}4,3,1,5,4,3,2]\\[2pt]
x_{41}^{E_7}&=[0,2,0,3,2,0,4,3,2,0,5,4,3,2,0,6,5,4,3,1,2,0,3,2,\\
&\phantom{=[}4,3,1,5,4,3,2]\\[2pt]
x_{42}^{E_7}&=[1,3,1,2,4,3,1,2,3,5,4,3,1,2,3,4,6,5,4,3,1,2,3,4,5]\\[2pt]
x_{49}^{E_7}&=[0,1,2,0,3,1,2,0,3,2,4,3,1,2,0,3,4,5,4,3,6,5,4,3,\\
&\phantom{=[}1,2,0,3].
\end{align*}
}

\medskip
\noindent\textit{Type $E_8$.}
{\allowdisplaybreaks[1]
\begin{align*}
x_{46}^{E_8}&=[0,1,2,0,3,1,2,0,3,2,4,3,1,2,0,3,2,4,3,1,5,4,3,1,\\
&\phantom{=[}2,0,3,2,4,3,1,5,4,3,2,0,6,5,4,3,1,2,0,3,2,4,3,1,\\
&\phantom{=[}5,4,3,2,6,5,4,3,1,7,6,5,4,3,1,2,0,3,2,4,3,1,5,4,\\
&\phantom{=[}3,2,0,6,5,4,3,1,2,3,4,5,7,6,5,4,3,1,2,0,3,2,4,3,\\
&\phantom{=[}1,5,4,3,2,0]\\[2pt]
x_{68}^{E_8}&=[1,2,3,1,2,0,3,4,3,1,2,0,3,2,4,3,5,4,3,1,2,0,3,2,\\
&\phantom{=[}4,3,1,5,4,3,2,6,5,4,3,1,2,0,3,2,4,3,1,5,4,3,2,0,\\
&\phantom{=[}6,5,4,3,1,2,3,4,5,6,7,6,5,4,3,1,2,0,3,2,4,3,1,5,\\
&\phantom{=[}4,3,2,0,6,5,4,3,1,2,3,4,5,6]\\[2pt]
x_{73}^{E_8}&=[0,1,2,0,3,1,2,0,3,2,4,3,1,2,0,3,2,4,3,1,5,4,3,1,\\
&\phantom{=[}2,0,3,2,4,3,1,5,4,3,2,0,6,5,4,3,1,2,0,3,2,4,3,1,\\
&\phantom{=[}5,4,3,2,0,6,5,4,3,1,7,6,5,4,3,1,2,0,3,2,4,3,1,5,\\
&\phantom{=[}4,3,2,6,5,4,7,6,5]\\[2pt]
x_{74}^{E_8}&=[0,1,2,0,3,1,2,0,3,2,4,3,1,2,0,3,2,4,3,1,5,4,3,1,\\
&\phantom{=[}2,0,3,2,4,3,1,5,4,3,2,0,6,5,4,3,1,2,0,3,2,4,3,1,\\
&\phantom{=[}5,4,3,2,6,5,4,3,1]\\[2pt]
x_{88}^{E_8}&=[1,2,3,1,2,3,4,3,1,2,3,4,5,4,3,1,2,0,3,4,5,6,5,4,\\
&\phantom{=[}3,1,2,0,3,2,4,3,1,5,4,3,2,0,6,5,4,7,6,5,4,3,1,2,\\
&\phantom{=[}0,3,2,4,3,1,5,4,3,2,0,6,5,4,3,1,2,3,4,5]\\[2pt]
x_{89}^{E_8}&=[1,3,1,2,4,3,1,2,3,5,4,3,1,2,3,4,5,6,5,4,3,1,2,0,\\
&\phantom{=[}3,2,4,3,5,6,7,6,5,4,3,1,2,0,3,2,4,3,1,5,4,3,2,0,\\
&\phantom{=[}6,5,4,3,1,2,3,4,5,6,7,6,5,4,3,1,2,0,3]\\[2pt]
x_{94}^{E_8}&=[0,1,2,0,3,1,2,0,3,2,4,3,1,2,3,4,6,5,4,3,1,2,0,3,\\
&\phantom{=[}2,4,3,1,5,4,3,2,0,6,7,6,5,4,3,1,2,3,4,5]\\[2pt]
x_{95}^{E_8}&=[1,2,3,1,2,0,3,4,3,1,2,0,3,2,4,3,5,4,3,1,2,0,3,2,\\
&\phantom{=[}4,3,1,5,4,3,2]\\[2pt]
x_{96}^{E_8}&=[0,1,2,3,1,2,3,4,3,1,2,3,4,5,4,3,1,6,5,4,3,1,2,0,\\
&\phantom{=[}3,2,4,3,5,4,6,5,7,6,5,4,3,1,2,0,3,2,4,5]\\[2pt]
x_{98}^{E_8}&=[0,2,0,3,2,0,4,3,2,0,5,4,3,2,0,6,5,4,3,1,2,0,3,2,\\
&\phantom{=[}4,3,1,5,4,3,2]\\[2pt]
x_{106}^{E_8}&=[0,1,2,0,3,1,2,0,3,2,4,3,1,2,0,3,4,5,4,3,6,5,4,3,\\
&\phantom{=[}1,2,0,3].
\end{align*}
}

\medskip

 For the repeated-$\mathbf a$ rows in Tables \ref{tab:int-e7} and \ref{tab:int-e8}, the parenthesized PyCox labels record the \texttt{special} field returned by \texttt{klcellrepelm}; this is the datum needed when the GK dimension alone does not distinguish the cell. 

For reference, we record the local special $\aff$-value sets used in Section \ref{sec:pseudo-nonintegral}. For a root system $\Psi$, write
\[
\mathcal A(\Psi)=\{\aff(\pi):\pi\in{\rm Irr}(W(\Psi))\text{ is special}\}.
\]
For a product root system, $\mathcal A$ is obtained by additivity of the $\aff$-function. The exceptional factor values are read from the special character data of Lusztig \cite{lusztig1984char} and Geck--Pfeiffer \cite{GP}. The required sets are as follows.

\begin{table}[htbp]
\caption{Local special $\aff$-value sets for pseudo-maximal subsystems}
\label{tab:pseudomax-a-values}
\centering
\scriptsize
\renewcommand{\arraystretch}{1.08}
\setlength{\tabcolsep}{4pt}
\begin{tabular}{@{}clp{0.57\textwidth}@{}}
\toprule
$\Phi$ & $\Phi_{[\lambda]}$ & $\mathcal A(\Phi_{[\lambda]})$ \\
\midrule
$G_2$ & $A_2$ & $\{0,1,3\}$ \\
      & $A_1\times\widetilde A_1$ & $[0,2]$ \\
\midrule
$F_4$ & $C_4$ & $[0,6]\cup\{9,16\}$ \\
      & $A_2\times\widetilde A_2$ & $[0,4]\cup\{6\}$ \\
      & $A_1\times\widetilde A_3$ & $[0,4]\cup\{6,7\}$ \\
      & $B_3\times\widetilde A_1$ & $[0,5]\cup\{9,10\}$ \\
\midrule
$E_6$ & $D_5$ & $[0,7]\cup\{10,12,13,20\}$ \\
      & $A_5\times A_1$ & $[0,8]\cup\{10,11,15,16\}$ \\
      & $A_2^3$ & $[0,7]\cup\{9\}$ \\
\midrule
$E_7$ & $D_6\times A_1$ & $[0,17]\cup\{20,21,22,30,31\}$ \\
      & $A_7$ & $[0,13]\cup\{15,16,21,28\}$ \\
      & $A_5\times A_2$ & $[0,11]\cup\{13,15,16,18\}$ \\
      & $A_3\times A_3\times A_1$ & $[0,10]\cup\{12,13\}$ \\
      & $E_6$ & $[0,7]\cup\{10,11,12,13,15,20,25,36\}$ \\
\midrule
$E_8$ & $D_8$ & $[0,18]\cup\{20,21,22,23,24,28,29,30,31,34,42,43,56\}$ \\
      & $A_8$ & $[0,13]\cup\{15,16,18,21,22,28,36\}$ \\
      & $E_7\times A_1$ & $[0,17]\cup\{20,21,22,23,25,26,30,31,36,37,38,46,47,63,64\}$ \\
      & $A_7\times A_1$ & $[0,17]\cup\{21,22,28,29\}$ \\
      & $A_5\times A_2\times A_1$ & $[0,19]$ \\
      & $A_4\times A_4$ & $[0,14]\cup\{16,20\}$ \\
      & $D_5\times A_3$ & $[0,16]\cup[18,23]\cup\{26\}$ \\
      & $E_6\times A_2$ & $[0,16]\cup\{18,20,21,23,25,26,28,36,37,39\}$ \\
\bottomrule
\end{tabular}
\end{table}
\FloatBarrier


The remaining tables record only the Sommers duality and weighted Dynkin  diagram data used in Sections \ref{sec:embedding-examples} and \ref{sec:pseudo-nonintegral}. We use $(L^\vee,\mathcal O_{L^\vee})$ for the Sommers datum occurring in Theorem \ref{P:Som}, where $\mathcal O_{L^\vee}$ is a distinguished orbit of the pseudo-Levi subgroup $L^\vee$. In the last column, $\mathrm{ns}$ indicates that the nilpotent orbit is not special. We write
\[
\mathcal O_{L^\vee}^{\rm BC}:={\rm sat}_{L^\vee}^{G^\vee}(\mathcal O_{L^\vee})
\]
for its saturation, written in the standard Bala--Carter notation of the corresponding group. The first column gives a  label for the Sommers datum $(L^\vee,\mathcal O_{L^\vee})$. Primes and double primes there distinguish embedded pseudo-Levi data as in Remark \ref{prime1-2}; the second column records the corresponding saturation and therefore does not inherit these decorations automatically. In the Springer character column we give Carter's label $\phi_{d,b}$ first and the corresponding PyCox label in parentheses. For type $F_4$, primes on Carter's labels distinguish the two characters of the same degree and $b$-invariant. The character and orbit correspondences are taken from \cite[Tables~C.3--C.6]{GP} and \cite[\S~13.3]{Ca85}.

\begin{table}[htbp]\caption{Sommers duality and weighted diagrams for $ F_4$ }
	\label{tabcharf4}
	\centering
	\renewcommand{\arraystretch}{1.15}
	\setlength{\tabcolsep}{3.5pt}
{\small\begin{tabular}{@{}cccccc@{}}
\toprule
Sommers datum & $\mathcal{O}_{L^\vee}^{\rm BC} $ & diagram of $\mathcal{O}_{L^\vee}^{\rm BC} $ &$\mathcal{O}:=d_{\rm Som}^{G^\vee}(L^\vee, \mathcal{O}_{L^\vee})$ &  \shortstack{Springer character\\(Carter; PyCox)} &  $\dim \mathcal{O}$ \\
\midrule
$B_3$
& $B_3$ & $[2,2,0,0]$& $A_2$  & $\phi_{8,9}^{\prime\prime}\;(8_4)$  &  $30$ \\
$C_3$
&$C_3$ & $[1,0,1,2]$&$\widetilde A_2$  & $\phi_{8,9}^{\prime}\;(8_2)$  &   $30$ \\
\midrule
$A_3$
& $B_2$ &$[2,0,0,1] $&$B_2$  &  $\phi_{9,6}^{\prime\prime}\;(9_2)$  &   $36\,(\mathrm{ns})$ \\
$B_4(a_2)$
& $F_4(a_3)$ & $[0,2,0,0]$  & $B_2$  & $\phi_{9,6}^{\prime\prime}\;(9_2)$  &   $36\,(\mathrm{ns})$ \\
$A_2+\widetilde A_2$
& $F_4(a_3)$ & $[0,2,0,0]$  & $\widetilde A_2+A_1$  & $\phi_{6,6}^{\prime}\;(6_1)$  &   $36\,(\mathrm{ns})$ \\
\midrule
$A_2$
& $A_2$ & $[2,0,0,0]$ & $B_3$  & $\phi_{8,3}^{\prime\prime}\;(8_1)$  &   $42$ \\
$2A_1+\widetilde A_1$
&$A_2$ & $[2,0,0,0]$&  $B_3$  &$\phi_{8,3}^{\prime\prime}\;(8_1)$  &   $42$ \\
$\widetilde A_2$
& $\widetilde A_2$ & $[0,0,0,2]$& $C_3$  &  $\phi_{8,3}^{\prime}\;(8_3)$  & $42$ \\
\bottomrule\end{tabular}}
	\bigskip

\end{table}

\begin{table}[htbp]\caption{Sommers duality and weighted diagrams for $ E_7$ }
	\label{tabchare7}
	\centering
	\renewcommand{\arraystretch}{1.15}
	\setlength{\tabcolsep}{3.5pt}
{\small\begin{tabular}{@{}cccccc@{}}
\toprule
Sommers datum & $\mathcal{O}_{L^\vee}^{\rm BC} $ & diagram of $\mathcal{O}_{L^\vee}^{\rm BC} $ &$\mathcal{O}:=d_{\rm Som}^{G^\vee}(L^\vee, \mathcal{O}_{L^\vee})$ &  \shortstack{Springer character\\(Carter; PyCox)}&  $\dim \mathcal{O}$ \\
\midrule
$(A_5+A_1)'$ &$E_6(a_3)$
&$[0,0,2,0,0,2,0]$& $A_3+2A_1$  & $\phi_{216,16}\;(216_a)$  &   $94\,(\mathrm{ns})$ \\
$(A_5+A_1)''$ 
&$A_5+A_1$ 
&$[1,0,0,1,0,1,2]$&$D_4(a_1)$  &  $\phi_{315,16}\;(315_a)$  &   $94$ \\
$E_7(a_5)$ 
&$E_7(a_5)$ 
&$[0,0,0,2,0,0,2]$& $D_4(a_1)$  & $\phi_{315,16}\;(315_a)$  &   $94$ \\
$D_6(a_2)$ 
&$D_6(a_2)$ 
&$[0,1,1,0,1,0,2]$& $D_4(a_1)$  & $\phi_{315,16}\;(315_a)$  &   $94$ \\
\midrule
$E_6(a_3)$ 
& $E_6(a_3)$ 
&$[0,0,2,0,0,2,0]$&$D_4(a_1)+A_1$  & $\phi_{405,15}\;(405_a')$  &   $96$ \\
$(A_5)'$ 
&$(A_5)'$ 
&$[1,0,0,1,0,2,0]$& $D_4(a_1)+A_1$  &  $\phi_{405,15}\;(405_a')$  &  $96$ \\
$(A_5)''$ 
&$(A_5)''$ 
&$[2,0,0,0,0,2,2]$& $D_4$  & $\phi_{105,15}\;(105_c')$  &   $96$ \\
\midrule

$2A_3$ 
&$A_4$&$[2,0,0,0,0,2,0]$& $D_4+A_1$  & $\phi_{84,12}\;(84_a)$  &   $102\,(\mathrm{ns})$ \\
$D_4$ 
& $D_4$ 
&$[2,0,2,0,0,0,0]$&$(A_5)''$  &  $\phi_{105,12}\;(105_c)$  &   $102$ \\
\midrule

$3A_2$ 
&$D_4(a_1)$  &$[0,0,2,0,0,0,0]$& $A_5+A_1$  &  $\phi_{70,9}\;(70_a')$  &   $108\,(\mathrm{ns})$ \\
$A_3+3A_1$ &$D_4(a_1)+A_1$
&$[0,1,1,0,0,0,1]$& $(A_5)'$  & $\phi_{216,9}\;(216_a')$  &   $108\,(\mathrm{ns})$ \\
$A_3+A_2$ 
& $A_3+A_2$ 
&$[0,0,0,1,0,1,0]$&$D_5(a_1)+A_1$  &  $\phi_{378,9}\;(378_a')$  &  $108$ \\
$D_4(a_1)+2A_1$ 
&  $A_3+A_2$ 
&$[0,0,0,1,0,1,0]$&   $D_5(a_1)+A_1$  &   $\phi_{378,9}\;(378_a')$  &  $108$ \\
\midrule
$(A_3+2A_1)'$&$D_4(a_1)$  
&$[0,0,2,0,0,0,0]$& $D_6(a_2)$  & $\phi_{280,8}\;(280_b)$  &   $110\,(\mathrm{ns})$ \\
$(A_3+2A_1)''$
&$A_3+2A_1$
&$[1,0,0,0,1,0,1]$& $E_6(a_3)$  & $\phi_{405,8}\;(405_a)$  &   $110$ \\
$D_4(a_1)+A_1$
&$D_4(a_1)+A_1$
&$[0,1,1,0,0,0,1]$& $E_6(a_3)$  & $\phi_{405,8}\;(405_a)$  &   $110$ \\
\midrule
$(A_3+A_1)''$
& $(A_3+A_1)''$
&$[2,0,0,0,0,0,2]$& $D_5$  & $\phi_{189,7}\;(189_c')$  &  $112$ \\

$(A_3+A_1)'$
& $(A_3+A_1)'$
&$[1,0,0,1,0,0,0]$&$E_7(a_5)$  & $\phi_{315,7}\;(315_a')$  &   $112$ \\
$2A_2+A_1$
&$2A_2+A_1$
&$[0,0,1,0,0,1,0]$& $E_7(a_5)$  & $\phi_{315,7}\;(315_a')$  &   $112$ \\
$D_4(a_1)$
&$D_4(a_1)$
&$[0,0,2,0,0,0,0]$&$E_7(a_5)$  & $\phi_{315,7}\;(315_a')$  &   $112$ \\
\midrule

$(4A_1)'$ &$A_2$
&$[2,0,0,0,0,0,0]$& $D_6$  & $\phi_{35,4}\;(35_b)$  &   $118\,(\mathrm{ns})$ \\
$(4A_1)''$
&  $4A_1$
&$[0,1,0,0,0,0,1]$&  $E_6(a_1)$  &  $\phi_{120,4}\;(120_a)$  &   $118$ \\
$A_2+A_1$
& $A_2+A_1$
&$[1,0,0,0,0,1,0]$&$E_6(a_1)$  & $\phi_{120,4}\;(120_a)$  &   $118$ \\
\midrule
$A_2$
& $A_2$
&$[2,0,0,0,0,0,0]$&  $E_7(a_3)$  &  $\phi_{56,3}\;(56_a')$  &   $120$ \\
$(3A_1)'$
& $(3A_1)'$
&$[0,0,1,0,0,0,0]$&   $E_7(a_3)$  &  $\phi_{56,3}\;(56_a')$  &   $120$ \\
$(3A_1)''$
& $(3A_1)''$
&$[0,0,0,0,0,0,2]$& $E_6$  & $\phi_{21,3}\;(21_b')$  &   $120$ \\
\bottomrule\end{tabular}}
	\bigskip
	
We list all the possible $\mathcal{O}_{L^\vee}$ when $\mathcal{O}_{L^\vee}$ or $\mathcal{O}$ carries a primed or double primed label.
\end{table}

\begin{table}[htbp]\caption{Sommers duality and weighted diagrams for $ E_8$ }
	\label{tabchare8}
	\centering

	\renewcommand{\arraystretch}{1.15}
	\setlength{\tabcolsep}{3.5pt}
{\small\begin{tabular}{@{}cccccc@{}}
\toprule
Sommers datum & $\mathcal{O}_{L^\vee}^{\rm BC} $ & diagram of $\mathcal{O}_{L^\vee}^{\rm BC} $ &$\mathcal{O}:=d_{\rm Som}^{G^\vee}(L^\vee, \mathcal{O}_{L^\vee})$ &  \shortstack{Springer character\\(Carter; PyCox)} &  $\dim \mathcal{O}$  \\
\midrule
$(A_7)''$
&    $E_6(a_1)$ &$[2,0,0,0,0,2,0,2]$&$D_4+A_1$  &  $\phi_{700,28}\;(700_{xx}')$  &   $184\,(\mathrm{ns})$ \\

$(A_7)'$
& $A_7$&$[1,0,0,1,0,1,1,0]$& $D_4(a_1)+A_2$  &  $\phi_{2240,28}\;(2240_x')$  &  $184$ \\
$E_8(b_6)$ &$E_8(b_6)$
&$[0,0,0,2,0,0,0,2]$& $D_4(a_1)+A_2$  & $\phi_{2240,28}\;(2240_x')$  &   $184$ \\
$E_6(a_1)+A_2$
& $E_8(b_6)$&$[0,0,0,2,0,0,0,2]$& $D_4(a_1)+A_2$  & $\phi_{2240,28}\;(2240_x')$  &   $184$ \\
\midrule
$(A_5+A_1)''$
& $E_6(a_3)$ &$[2,0,0,0,0,0,2,0]$& $D_5+A_1$  &  $\phi_{3200,16}\;(3200_x)$  &  $208\,(\mathrm{ns})$ \\
$(A_5+A_1)'$
& $A_5+A_1$
&$[1,0,0,1,0,0,0,1]$& $E_8(a_7)$  & $\phi_{4480,16}\;(4480_y)$  &  $208$ \\
$E_7(a_5)$
& $E_7(a_5)$
&$[0,0,0,1,0,1,0,0]$&$E_8(a_7)$  & $\phi_{4480,16}\;(4480_y)$  &   $208$ \\
$E_6(a_3)+A_1$
& $E_6(a_3)+A_1$
&$[1,0,0,0,1,0,1,0]$&$E_8(a_7)$  &  $\phi_{4480,16}\;(4480_y)$  &  $208$ \\
$D_5(a_1)+A_2$
&$D_5(a_1)+A_2$
&$[0,0,1,0,0,1,0,1]$ &$E_8(a_7)$  & $\phi_{4480,16}\;(4480_y)$  &   $208$ \\
$D_6(a_2)$
& $D_6(a_2)$
&$[0,1,1,0,0,0,1,0]$&$E_8(a_7)$  & $\phi_{4480,16}\;(4480_y)$  &   $208$ \\
$A_4+A_3$
&$A_4+A_3$
&$[0,0,0,1,0,0,1,0]$& $E_8(a_7)$  &  $\phi_{4480,16}\;(4480_y)$  &  $208$ \\
 $E_8(a_7)$  & $E_8(a_7)$&$[0,0,0,0,2,0,0,0]$&$E_8(a_7)$  &  $\phi_{4480,16}\;(4480_y)$  &  $208$ \\
\midrule
$(2A_3)''$& $A_4$
&$[2,0,0,0,0,0,0,2]$&$D_6$  &  $\phi_{972,12}\;(972_x)$  &   $216\,(\mathrm{ns})$ \\
$(2A_3)'$
&$2A_3$
&$[1,0,0,0,1,0,0,0]$& $D_7(a_2)$  &  $\phi_{4200,12}\;(4200_x)$  &   $216$ \\
$A_4+2A_1$
& $A_4+2A_1$
&$[0,0,0,1,0,0,0,1]$&$D_7(a_2)$  &  $\phi_{4200,12}\;(4200_x)$  &  $216$ \\
$D_4$
&$D_4$
&$[0,0,0,0,0,0,2,2]$& $E_6$  &  $\phi_{525,12}\;(525_x)$  &  $216$ \\
\midrule
$(A_3+2A_1)''$ & $D_4(a_1)$
&$[0,0,0,0,0,0,2,0]$& $E_7(a_2)$  &  $\phi_{1344,8}\;(1344_x)$  &  $224\,(\mathrm{ns})$ \\
$(A_3+2A_1)'$
& $A_3+2A_1$
&$[0,0,1,0,0,0,0,1]$& $E_8(a_6)$  & $\phi_{1400,8}\;(1400_x)$  &  $224$ \\
$2A_2+2A_1$
& $2A_2+2A_1$
& $[0,0,0,0,1,0,0,0]$&  $E_8(a_6)$  &  $\phi_{1400,8}\;(1400_x)$  &   $224$ \\
$D_4(a_1)+A_1$
&$D_4(a_1)+A_1$
&$[0,1,0,0,0,0,1,0]$ &$E_8(a_6)$  &  $\phi_{1400,8}\;(1400_x)$  &  $224$ \\
\midrule
$(4A_1)''$
& $A_2$ &$[0,0,0,0,0,0,0,2]$& $E_7$  & $\phi_{84,4}\;(84_x)$  &   $232\,(\mathrm{ns})$ \\

$(4A_1)'$
&$4A_1$
&$[0,1,0,0,0,0,0,0]$& $E_8(a_4)$  & $\phi_{210,4}\;(210_x)$  &   $232$ \\
$A_2+A_1$
&$A_2+A_1$
&$[1,0,0,0,0,0,0,1]$& $E_8(a_4)$  & $\phi_{210,4}\;(210_x)$  &   $232$ \\
\bottomrule\end{tabular}}
\vspace{0.2cm}

\end{table}

\hspace{3cm}

\end{document}